%% file: ms.tex
\documentclass[11pt]{amsart}
	
	\usepackage[utf8]{inputenc}
	\usepackage{amsmath}
	\usepackage{amsfonts}
	\usepackage{amssymb}
	\usepackage{amsthm}
	\usepackage{mathtools}
	\usepackage[margin=3.0cm]{geometry}
	\usepackage{tikz-cd}
	\usepackage{enumitem}
	\usepackage{hyperref}
	\usepackage{placeins}

	\usepackage{pgfplots}
	
	\usepackage{grffile}
	\pgfplotsset{compat=1.15}
	\usetikzlibrary{plotmarks}
	\usetikzlibrary{arrows.meta}
	\usepgfplotslibrary{patchplots, groupplots}

	\definecolor{mycolor1}{RGB}{57,106,177}

	\newdimen\tikzwidth
	\newdimen\tikzheight
	\pgfplotsset{
    compat=newest,
    scale only axis,
    width=\tikzwidth, height=\tikzheight,
    legend style={font=\footnotesize,},
    legend cell align={left},
    xticklabel style={font=\footnotesize,},
    yticklabel style={font=\footnotesize,},
    x tick label style={
        /pgf/number format/.cd,
        fixed,
        fixed zerofill,
        precision=4,
        /tikz/.cd
    }
    hugeData/.style={
	    every axis plot/.append style={each nth point=100, filter discard warning=false},
    }
}
	
\usepackage{xcolor}
\usepackage[most]{tcolorbox}
\usepackage{listings}

\definecolor{white}{rgb}{1,1,1}
\definecolor{mygreen}{rgb}{0,0.4,0}
\definecolor{light_gray}{rgb}{0.97,0.97,0.97}
\definecolor{mykey}{rgb}{0.117,0.403,0.713}

\tcbuselibrary{listings}
\newlength\inwd
\newcounter{ipythcntr}
\renewcommand{\theipythcntr}{\texttt{[\arabic{ipythcntr}]}}

\newtcblisting{pyin}[1][]{%
  sharp corners,
  enlarge left by=\inwd,
  width=\linewidth-\inwd,
  enhanced,
  boxrule=0pt,
  colback=light_gray,
  listing only,
  top=0pt,
  bottom=0pt,
  overlay={
    \node[
      anchor=north east,
      text width=\inwd,
      font=\footnotesize\ttfamily\color{mykey},
      inner ysep=2mm,
      inner xsep=0pt,
      outer sep=0pt
      ] 
      at (frame.north west)
      {\refstepcounter{ipythcntr}\label{#1}In \theipythcntr:};
  }
  listing engine=listing,
  listing options={
    aboveskip=1pt,
    belowskip=1pt,
    basicstyle=\footnotesize\ttfamily,
    language=Python,
    keywordstyle=\color{mykey},
    showstringspaces=false,
    stringstyle=\color{mygreen}
  },
}
\newtcblisting{pyprint}{
  sharp corners,
  enlarge left by=\inwd,
  width=\linewidth-\inwd,
  enhanced,
  boxrule=0pt,
  colback=white,
  listing only,
  top=0pt,
  bottom=0pt,
  overlay={
    \node[
      anchor=north east,
      text width=\inwd,
      font=\footnotesize\ttfamily\color{mykey},
      inner ysep=2mm,
      inner xsep=0pt,
      outer sep=0pt
      ] 
      at (frame.north west)
      {};
  }
  listing engine=listing,
  listing options={
      aboveskip=1pt,
      belowskip=1pt,
      basicstyle=\footnotesize\ttfamily,
      language=Python,
      keywordstyle=\color{mykey},
      showstringspaces=false,
      stringstyle=\color{mygreen}
    },
}
\newtcblisting{pyout}[1][\theipythcntr]{
  sharp corners,
  enlarge left by=\inwd,
  width=\linewidth-\inwd,
  enhanced,
  boxrule=0pt,
  colback=white,
  listing only,
  top=0pt,
  bottom=0pt,
  overlay={
    \node[
      anchor=north east,
      text width=\inwd,
      font=\footnotesize\ttfamily\color{mykey},
      inner ysep=2mm,
      inner xsep=0pt,
      outer sep=0pt
      ] 
      at (frame.north west)
      {\setcounter{ipythcntr}{\value{ipythcntr}}Out#1:};
  }
  listing engine=listing,
  listing options={
      aboveskip=1pt,
      belowskip=1pt,
      basicstyle=\footnotesize\ttfamily,
      language=Python,
      keywordstyle=\color{mykey},
      showstringspaces=false,
      stringstyle=\color{mygreen}
    },
}
						
	\hypersetup{colorlinks = true, linkbordercolor = {white}, linkcolor = {blue}, citecolor = {blue}}
	\usetikzlibrary{matrix}
	
	\numberwithin{subsection}{section}
\newtheoremstyle{1}
{6pt} 
{0pt} 
{\itshape} 
{} 
{\bfseries} 
{.} 
{.5em} 
{} 

\newtheoremstyle{2}
{6pt} 
{0pt} 
{} 
{} 
{\bfseries} 
{.} 
{.5em} 
{} 

\theoremstyle{1}
	
	\newtheorem{theorem}{Theorem}[section]
	\newtheorem*{theorem*}{Theorem}
	\newtheorem{lemma}[theorem]{Lemma}
	\newtheorem{proposition}[theorem]{Proposition}
	\newtheorem{corollary}[theorem]{Corollary}
	\newtheorem{definition}[theorem]{Definition}

\theoremstyle{2}
	\newtheorem{example}[theorem]{Example}
	\newtheorem{notation}[theorem]{Notation}
	
	\newtheorem{rmrk}[theorem]{Remark}

\newcommand{\xnumber}{}
\theoremstyle{1}
	\newtheorem{xlemmainner}{Lemma}
	\newtheorem{xpropositioninner}{Proposition}
	\newtheorem{xcorollaryinner}{Corollary}
\theoremstyle{2}
	\newtheorem{xremarkinner}{Remark}

\newenvironment{xlemma}[1]{\renewcommand{\xnumber}{#1}\begin{xlemmainner}}{\end{xlemmainner}}
\newenvironment{xproposition}[1]{\renewcommand{\xnumber}{#1}\begin{xpropositioninner}}{\end{xpropositioninner}}
\newenvironment{xcorollary}[1]{\renewcommand{\xnumber}{#1}\begin{xcorollaryinner}}{\end{xcorollaryinner}}
\newenvironment{xremark}[1]{\renewcommand{\xnumber}{#1}\begin{xremarkinner}}{\end{xremarkinner}}
	
	\DeclareMathOperator{\Z}{\mathbb{Z}}
	
	\DeclareMathOperator{\Q}{\mathbb{Q}}
	\DeclareMathOperator{\N}{\mathbb{N}}
	\DeclareMathOperator{\F}{\mathbb{F}}

	\DeclareMathOperator{\Spec}{Spec}

	\DeclareMathOperator{\Sch}{Sch}

	\DeclareMathOperator{\Ext}{Ext}
	\DeclareMathOperator{\Lie}{Lie}
	\DeclareMathOperator{\Rep}{Rep}
	\DeclareMathOperator{\Ind}{Ind}

	\DeclareMathOperator{\GL}{GL}
	\DeclareMathOperator{\Sp}{Sp}
	
	\DeclareMathOperator{\SL}{SL}
	
	\DeclareMathOperator{\ZipFlag}{-ZipFlag}
	\DeclareMathOperator{\Zip}{-Zip}
	\DeclareMathOperator{\Sh}{Sh}

	\DeclareMathOperator{\gr}{gr}

	\DeclareMathOperator{\dR}{dR}
	\DeclareMathOperator{\std}{std}
	\DeclareMathOperator{\Sym}{Sym}
	\DeclareMathOperator{\Proj}{Proj}

	\DeclareMathOperator{\tor}{tor}
	\DeclareMathOperator{\sub}{sub}
	
	\DeclareMathOperator{\Brh}{Brh}
	\DeclareMathOperator{\Lc}{\mathcal{L}}
	\DeclareMathOperator{\amp}{ample}
	\DeclareMathOperator{\van}{van}
	\DeclareMathOperator{\Ha}{Ha}
	
	\DeclareMathOperator{\ch}{ch}
	\DeclareMathOperator{\divi}{div}
	\DeclareMathOperator{\red}{red}
	\DeclareMathOperator{\Orb}{Orb}

\title[Vanishing results over the Siegel modular variety]{Vanishing results for the coherent cohomology of automorphic vector bundles over the Siegel variety in positive characteristic}

\author{Thibault Alexandre}

\begin{document}
\input{abstract.tex}
\maketitle
\tableofcontents

\input{part1.tex}
\input{part2.tex}
\input{part3.tex}
\input{part4.tex}

\input{part5.tex}
\input{part6.tex}

\input{part7.tex}

\bibliography{mybib}{}
\bibliographystyle{alpha}
\end{document}

%% file: abstract.tex
\begin{abstract}
We prove vanishing results for the coherent cohomology of the good reduction modulo $p$ of the Siegel modular variety with coefficients in some automorphic bundles. We show that for an automorphic bundle with highest weight $\lambda$ near the walls of the anti-dominant Weyl chamber, there is an integer $e \geq 0$ such that the cohomology is concentrated in degrees $[0, e]$. The accessible weights with our method are not necessarily regular and not necessarily $p$-small. Since our method is technical, we also provide an algorithm written in SageMath that computes explicitly the vanishing results.
\end{abstract}

%% file: part1.tex
\section{Introduction}\label{part1}
\noindent\textbf{Note on this version.} The original version of this article was published in Algebra \& Number Theory \textbf{19} (2025), no.~1, 143--193, \href{https://doi.org/10.2140/ant.2025.19.143}{doi:10.2140/ant.2025.19.143}. This version corrects errors in the published version; a corrigendum has been submitted to Algebra \& Number Theory. Statements added in this version are numbered 6.7a, 6.9a, etc., so that the numbering of the published version is preserved.
\subsection{History and motivation}
The cohomology of automorphic vector bundles on Shimura varieties has played important roles in the study of arithmetic properties of automorphic representations as explained in \cite{MR808114}. Let $p$ be a prime number and $N \geq 3$ be an integer such that $p \nmid N$. Consider the Siegel modular variety $\Sh$ of level $N$ and genus $g \geq 1$ over $\F_p$. It is defined as the fine moduli space of abelian schemes of dimension $g$ over $\F_p$ with a principal polarization and a basis of their $N$-torsion. This scheme is smooth and not proper over $\F_p$ but we can consider a smooth toroidal compactification $\Sh^{\tor}$ as defined in \cite{MR1083353}. This article is concerned with the coherent cohomology of $\Sh^{\tor}$. We consider automorphic vector bundles that are defined as the contracted product of the Hodge vector bundle $\Omega$ with an algebraic representation of $\GL_g$. Over a field of characteristic $p >0$, there are two\footnote{In the context of a highest weight category (see \cite[3.7]{riche:tel-01431526} for a general introduction to this notion), one is also concerned with the simple module $L(\lambda)$ and the tilting module $T(\lambda)$.} important indecomposable (but not necessarily irreducible) algebraic representations of highest weight $\lambda$: the standard module $\Delta(\lambda)$ and the costandard module $\nabla(\lambda)$. Note that these two algebraic representations are isomorphic and irreducible when the weight $\lambda$ is $p$-small\footnote{It means that $\forall \alpha \in \phi^{+} \ \langle \lambda + \rho, \alpha^{\vee} \rangle \leq p$ where $\rho$ is the half-sum of positive roots.} for $\GL_g$. We use the same notation to denote the corresponding automorphic vector bundles on the Siegel modular variety. To better understand the coherent cohomology of $\Sh^{\tor}$, it is convenient to know that all but certain cohomological degrees must be zero. In their articles \cite{MR2913102} and \cite{MR3055995}, Lan and Suh prove many vanishing results for the coherent cohomology of PEL Shimura varieties. Let $W$ denote the Weyl group of $\Sp_{2g}$ and $I$ denote the type of the parabolic subgroup $P \subset \Sp_{2g}$ that stabilizes the Hodge filtration on the Siegel upper half-plane $\mathbb{H}_g$. In the Siegel case, Lan and Suh were able to access automorphic bundles $\Delta(\lambda)^{\vee}$ in all Weyl chambers as long as the weight $\lambda$ can be written
\begin{equation}\label{eq3}
\lambda = w \cdot \mu + \underline{k}
\end{equation}
where $w$ is an element of the minimal left coset representatives $\prescript{I}{}{W}$ of type $I$, $\mu$ is a sufficiently regular weight\footnote{See \cite[Def. 7.18]{MR2913102}.} which is $p$-small for $\Sp_{2g}$ and such that $|\mu|_{\text{re},+}\footnote{See \cite[7.22]{MR2913102}.} < p$ and $\underline{k}$ is a positive parallel weight\footnote{It means of the form $(k,k,\cdots,k)$ for some $k > 0$.}. They use results from \cite{MR1944175} on dual Bernstein-Gelfand-Gelfand complexes and a geometric plethysm that imposes many restrictions on the size of the weight compared to $p$. We note that for such a weight $\lambda$, we have $\Delta(\lambda)^{\vee} = \nabla(-w_0\lambda) = \Delta(-w_0\lambda)$. As there is only a finite number of $p$-small characters for $\Sp_{2g}$, their method accesses only a finite number of weights up to positive parallel weights.
\subsection{Main results}\label{sect1}
Let $D_{\red}$ denote the boundary of the toroidal compactification and let $\nabla^{\sub}({\lambda})$ denote the subcanonical extension $\nabla({\lambda})(-D_{\red})$ of the costandard automorphic vector bundle of highest weight $\lambda$. Our main result is a general recipe to produce new vanishing results from old ones. Namely, we define a non-decreasing function $g_{I_0,e}$ on the power set of characters
\begin{equation*}
g_{I_0,e} : \mathcal{P}(X^*) \rightarrow \mathcal{P}(X^*)
\end{equation*}
that depends on a subset $I_0 \subset I$ where $I$ is the type of the parabolic subgroup of $\Sp_{2g}$ that stabilises the Hodge filtration and an integer $0 \leq e \leq d-1$ where $d= g(g+1)/2$ is the dimension of $\Sh^{\tor}$. Note that the definition of $g_{I_0,e}$ is technical and not very helpful because it is a byproduct of our method which relies on the partial degeneration of multiple spectral sequences. We describe these spectral sequences and give the exact definition of $g_{I_0,e}$ in the overview of the strategy.
\begin{theorem*}[Theorem \ref{th1}]
Assume that $p > g^2$ (so that $p > d_0$ for all $I_0$). Let $\mathcal{C}$ be a set of characters $\lambda$ for which the cohomology $H^i(\Sh^{\tor},\nabla^{\sub}({\lambda}))$ is concentrated in degrees $[0,e+1]$.Then, the image of $\mathcal{C}$ by the function $g_{I_0,e}$ is a set of characters $\lambda$ for which the cohomology $H^i(\Sh^{\tor},\nabla^{\sub}({\lambda}))$ is concentrated in degrees $[0,e]$.
\end{theorem*}
\vspace{0.2cm}
Moreover, in the extreme cases $e = 0$ and $e = d-1$, our method produces new vanishing results without any prior knowledge. These results can then be used in the other cases $0 < e < d-1$. We illustrate our results in the special case $g = 2$, $p = 5$ with the following figure.
\begin{figure}[h]
\centering
\begin{tikzpicture}[scale = 1.1]
\begin{axis}[
 axis lines=middle,
	grid=both,
	grid style={black!10},
	xmin=-44,
	xmax=5,
	ymin=-44,
	ymax=5,
	legend style={at={(0.1,0.65)},anchor=west, font = \tiny},
 	xlabel=$k_2$,
	ylabel=$k_1$,
	minor tick num=9,
	x axis line style = {stealth-},
	y axis line style = {stealth-},
	xticklabel shift={0.0cm},
	xlabel style={yshift=-7.2cm},
	yticklabel shift={-0.9cm},
	ylabel style={xshift=-7.2cm},
]

\addplot [only marks,
	color=green,
	mark=x,
	mark options={scale=0.7, fill=white}]
	table{results/g2p5_psmall.txt};
	\addlegendentry[align = left]{$p$-small weights for $\Sp_4$}
	 
\addplot [only marks,
	color=black,
	mark=o,
	mark options={scale=0.6, fill=white}]
	table{results/g2p5_0.txt};
	\addlegendentry{Concentrated in $[0]$}

\addplot [only marks,
	color=blue,
	mark=triangle,
	mark options={scale=0.6, fill=white}]
	table{results/g2p5_1.txt};
	\addlegendentry{Concentrated in $[0,1]$}

\addplot [only marks,
	color=red,
	mark=square,
	mark options={scale=0.6, fill=white}]
	table{results/g2p5_2.txt};
	\addlegendentry{Concentrated in $[0,2]$}

\draw[scale=0.5, domain=-70:70, smooth, variable=\x, black] plot ({\x}, {\x});

\node at (-2.8,1.3) {\tiny $(0,0)$};
\node at (-8.1,-4.0) {\tiny $(-4,-4)$};

\end{axis}

\end{tikzpicture}
\caption{$g = 2$, $p = 5$.}
\end{figure}
\FloatBarrier
The accessible weights with this method are not necessarily regular and not necessarily $p$-small (even up to a positive parallel weight) but they belong to the anti-dominant Weyl chamber\footnote{It corresponds to the dominant Weyl chamber in the work of Lan and Suh.}. Since the definition of the function $g_{I_0,e}$ is hard to grasp, we have implemented on SageMath an algorithm that compute the vanishing results with our method. Our method produces vanishing results for automorphic bundles $\nabla(\lambda)$ where $\lambda$ is not necessarily of the form $w \cdot \mu + \underline{k}$ as in equation \eqref{eq3}. In particular, we have no reason to expect that $\Delta(\lambda) = \nabla(\lambda)$. The $p$-smallness restriction is replaced with a much weaker restriction coming from the theory of $G\Zip$ called orbitally $p$-closeness. We note that in the special case $g =2$, $p = 5$, the only $p$-small character for $\Sp_4$ is $(0,0)$, which means that the method of Lan and Suh is only able to access weights of the form $w\rho-\rho + \underline{k}$.
\subsection{Overview of the strategy}
The first step is to consider the flag bundle $\pi : Y^{\tor}_{I_0} \rightarrow \Sh^{\tor}$ over the Siegel modular variety that parametrizes refinements of type $I_0 \subset I$ of the Hodge filtration of the universal semi-abelian scheme on $\Sh^{\tor}$. Let $d_0$ denote the dimension of $Y^{\tor}_{I_0}$ over $\mathbb{F}_p$. Let $P_0$ denote the parabolic subgroup of $\Sp_{2g}$ of type $I_0$. For each character $\lambda \in X^*(P_0)$, we have a line bundle $\mathcal{L}_{\lambda}$ on $Y^{\tor}_{I_0}$ such that
\begin{equation*}
\pi_*\mathcal{L}_{\lambda} = \nabla(\lambda).
\end{equation*}
Following the result from \cite{StrohPrep}, the second step is to use a result of \cite{MR3989256} about the existence of generalized Hasse invariants on the stack $\Sp_{2g}\ZipFlag^{\mu,I_0}$ to prove that certain line bundles $\mathcal{L}_{\lambda}$ are $D$-ample\footnote{See definition \ref{def5}.} on $Y^{\tor}_{I_0}$ where $D$ is a certain effective Cartier divisor whose associated reduced divisor is the boundary $D_{\red}$. The third step is to use a logarithmic version of the Kodaira-Nakano vanishing in positive characteristic due to Esnault and Viehweg to see that under the hypothesis $p > d_0 := \dim Y^{\tor}_{I_0}$, we have
\begin{equation*}
H^i(Y^{\tor}_{I_0}, \Omega^{d_0-e}_{Y^{\tor}_{I_0}}(\log D_{\red}) \otimes \mathcal{L}^{\sub}_{\lambda}) = 0,
\end{equation*}
for all $i > e$ and all $\lambda$ that admits generalized Hasse invariants. The fourth step is to filter the bundle
\begin{equation*}
\Omega^{d_0-e}_{Y^{\tor}_{I_0}}( \log D_{\red}) \otimes \mathcal{L}^{\sub}_{\lambda}
\end{equation*}
with an increasing filtration $F_\bullet$:
\begin{equation*}
F_k = \pi^*\Omega^{d_{0}-e-k}_{\Sh^{\tor}} (\log D_{\red}) \wedge \Omega^k_{Y^{\tor}_{I_0}} (\log D_{\red}) \otimes \mathcal{L}^{\sub}_{\lambda}
\end{equation*}
and then consider the corresponding spectral sequence. It is a spectral sequence starting at the second page $E_2^{i,j}$ whose limit is zero when $i+j > e$ by the logarithmic Kodaira-Nakano vanishing considered above. In general, it is impossible to extract information on the second page of a spectral sequence whose limit is zero. However, if we can show that the second page degenerates (at least partially), then we can deduce that some terms $E_2^{i,j}$ must be zero\footnote{In the case $e = 0$, the spectral sequence is concentrated on one row which explains why we do not need any prior vanishing results.}. The aim is to determine the vanishing results needed to ensure the partial degeneration of this spectral sequence. From the partial degeneration, we can deduce new vanishing results. Our method is technical as it involves recursively an unknown number of spectral sequences. Moreover, in the course of the argument, we are forced to contemplate tensor products of automorphic bundles $\nabla(\lambda) \otimes \nabla(\mu)$. To relate the cohomology of this tensor product to the cohomology of other automorphic bundles, we consider the spectral sequence associated to a $\nabla$-filtration\footnote{See definition \ref{def6}.} whose existence is ensured by \cite{MR1072820} and, like before, we determine the vanishing results needed to ensure its partial degeneration. The definition of the function $g_{I_0,e}$ on the power set of characters is a byproduct of our method that relies on the partial degeneration of relevant spectral sequences. More precisely, let $\mathcal{C}_{\amp,I_0}$ denote the set of characters such that $\Lc_{\lambda}$ is $D$-ample on $Y_{I_0}^{\tor}$, $r_0$ denote the relative dimension of $\pi : Y_{I_0}^{\tor} \rightarrow \Sh^{\tor}$, $(\mu^{k}_{j})_j$ denote the highest weights of the graded pieces of a $\nabla$-filtration of the $\GL_g$-module $\Lambda^k\Sym^2\std$\footnote{Actually, to follow the convention in definition \ref{def3}, we need to twist these weights by $w_0w_{0,\GL_g}$.}, $s_M = \sum_{\alpha \in M} \alpha$ for any $M \subset \phi^+$ and $\rho_{I_0} = \frac{1}{2} \sum_{\alpha \in \phi_L^{+} \backslash \phi_{I_0}^{+}} \alpha$. For any set $\mathcal{C}$ of characters, $g_{I_0,e}(\mathcal{C})$ is the set of characters $\lambda$ for which there is an index $j_0$ such that $\chi := \lambda - \mu^{d-e}_{j_0}$ and $\lambda' := \chi + 2\rho_{I_0}$ satisfy:
\begin{itemize}
\item $\chi$ is a character of $P_0$ such that $\langle \chi, \alpha^\vee \rangle \geq \delta_{I_0,e}(\alpha)$ for all $\alpha \in I - I_0$, where $\delta_{I_0,e} \geq 0$ is an explicit threshold which vanishes when $e \leq 1$ or $g \leq 3$ (see lemma \ref{lem:threshold}),
\item $\lambda' \in \mathcal{C}_{\amp,I_0}$,
\item the weights of the graded pieces of $\nabla$-filtrations of $\nabla(\mu^{d-e+k}_j) \otimes \nabla(\lambda'-s_M)$, for all $1 \leq k \leq \min(e,r_0)$, all $j$ and certain $M \subset \phi_L^+-\phi_{I_0}^+$ with $|M| = r_{0}-k$, belong to $\mathcal{C}$,
\item the weights of the graded pieces of $\nabla$-filtrations of $\nabla(\mu^{d-e}_j) \otimes \nabla(\chi)$, for all $j$, belong to $\mathcal{C}$, except $\lambda$ itself.
\end{itemize}
See definition \ref{def:615} for the exact definition.
\subsection{Organization of the paper}
In section \ref{part2}, we recall some results of algebraic representation of reductive groups in positive characteristic. In section \ref{part3}, we recall the definition of the Siegel modular variety and the different automorphic vector bundles. In section \ref{part4}, we recall how the theory of $G\Zip$ can be used to study the EO stratification of the Siegel modular variety. In particular, we recall the main result of \cite{MR3989256} about generalized Hasse invariants. In section \ref{part5}, we prove that line bundles of weight $\lambda$ on the flag bundle over the Siegel modular variety that admit generalized Hasse invariants are $D$-ample. We also recall a logarithmic version of the Kodaira-Nakano vanishing in positive characteristic due to Esnault and Viehweg. In section \ref{part6}, we present our general method for producing new vanishing results and we give more details in the case $g = 2$. In section \ref{part7}, we explain how to compute new vanishing results with an algorithm written in SageMath and we plot the results we have obtained in some special cases with $g = 2$ and $g =3$. See \href{https://github.com/ThibaultAlexandre/vanishing-results-over-the-siegel-variety}{github.com/ThibaultAlexandre/vanishing-results-over-the-siegel-variety} to download the algorithm.
\subsection*{Acknowledgement}
I heartily thank Benoit Stroh for encouraging me to write this paper and for explaining me the relevance of generalized Hasse invariants to obtain positivity results on the flag bundle over the Siegel modular variety. I am very grateful to Jean-Stefan Koskivirta for answering my questions on the theory of $G\Zip$s. I also thank Diego Berger, Yohan Brunebarbe and Arnaud Eteve for very helpful discussions.

\noindent\textbf{Use of AI tools.} Claude (Anthropic) and GPT (OpenAI) were used to help find the errors in the published version of this article, to help repair the affected statements and their proofs, to write the accompanying SageMath code, and to draft and edit the text of this revision. The author has checked every statement, proof, computation and citation, and takes full responsibility for the text.

%% file: part2.tex
\section{Recollection on group theory}\label{part2}
In this section, we follow mostly \cite{MR2015057} for generalities about algebraic representations of reductive groups over a field of positive characteristic. Let $k$ be a field of positive characteristic $p > 0$ and $G$ a geometrically connected split reductive algebraic group over $k$. The weights of the adjoint representation of $G$ on its Lie algebra $\mathfrak{g}$ define a set of roots $\phi$. We fix a Borel pair $(B,T)$ defined over $k$ where $B$ is a Borel subgroup and $T \subset B$ is a maximal torus. This choice of Borel pair determines a subset of simple roots $\Delta$ and positive roots $\phi^+$. We use a non-standard convention\footnote{It simplify the statement of the proposition \ref{prop13}.} for the positive roots as we declare $\alpha$ to be positive if the root group $U_{-\alpha}$ is contained in $B$. Let $\rho$ denote the half-sum of positive roots as $\Q$-character of $T$. If $I \subset \Delta$, we write $\phi_I$ (resp. $\phi_I^+$) for the set of roots (resp. positive roots) generated from $I$. We write $W$ for the Weyl group of $G$, $l : W \rightarrow \N$ for its length function and $w_0$ for its longest element. If $I \subset \Delta$, let $W_I \subset W$ denote the subgroup generated by the reflections $s_{\alpha}$ where $\alpha \in I$ and let ${}^IW \subset W$ denote the set of minimal length representatives of $W_I\backslash W$. We write $\langle \cdot ,\cdot \rangle : X^*(T) \times X_*(T) \rightarrow \mathbb{Z}$ for the perfect pairing between the characters $X^*(T)$ of $T$ and the cocharacters $X_*(T)$ of $T$. Since the characteristic of $k$ is assumed to be positive, $G$ is endowed with a relative Frobenius morphism $\varphi : G \rightarrow G^{(p)}$ where $G^{(p)} := G \times_{k,\sigma} k$ (with $\sigma : k \rightarrow k$ the Frobenius morphism of $k$) is again a reductive group over $k$. Unlike when the characteristic of $k$ is $0$, the category of algebraic representations of $G$ on finite dimensional vector spaces is no longer semi-simple. The simple objects $L(\lambda)$ are still indexed by their highest weight $\lambda$ but not every representation can be split into a direct sum of simple objects. This category $\Rep_k(G)$ has the structure of a highest weight category (see \cite[3.7]{riche:tel-01431526} for a general introduction to this notion).
\begin{definition}[{\cite[Part I, sect. 5.8]{MR2015057}}]
Let $\lambda : T \rightarrow \mathbb{G}_m$ be a character of $T$. We define a line bundle $\Lc_{\lambda}$ on the flag variety $G/B$ as the $B$-quotient of the vector bundle $G \times_k \mathbb{A}^1 \rightarrow G$, where $B$ acts on $G \times_k \mathbb{A}^1$ by
\begin{equation*}
(g,x)b = (gb^{-1},\lambda(b^{-1})x),
\end{equation*}
and where $\lambda$ is naturally extended by 0 on the unipotent part of $B$. The global section group $H^0(G/B,\Lc_\lambda)$ is given the structure of a $G$-module through left translation. As a consequence we get an algebraic representation of $G$, and we will denote it simply $\nabla(\lambda)$.
\end{definition}
\begin{proposition}[{\cite[Part II, sect. 2.6]{MR2015057}}]
The $G$-module $\nabla(\lambda)$ is non-zero exactly when $\lambda$ is dominant. Moreover, its highest $T$-weight is $\lambda$ and we call $\nabla(\lambda)$ the induced module or costandard module of highest weight $\lambda$.
\end{proposition}
\begin{rmrk}
A different convention can be found in the litterature where we set the dominance to be relative to $B$.
\end{rmrk}
\begin{definition}[{\cite[Part II, sect. 2.13]{MR2015057}}]
Let $\lambda \in X^*(T)$ be a character. The standard module of highest weight $\lambda$ can be defined 
\begin{equation*}
\Delta(\lambda) := \nabla(-w_0\lambda)^{\vee},
\end{equation*}
where $w_0$ is the longest element of the Weyl group $W$ of $G$ and $\vee$ denotes the linear dual in $\Rep_k(G)$.
\end{definition}
\vspace{0.2cm}
As a consequence from the definitions, $\nabla(\lambda)$ and $\Delta(\lambda)$ must have the same characters but they are usually not simple and not isomorphic. However $L(\lambda)$ is the socle of $\nabla(\lambda)$ and the head of $\Delta(\lambda)$ (see \cite[Part II, Chap. 2]{MR2015057}). We recall Kempf's vanishing theorem.
\begin{proposition}[{\cite[Part II, sect. 4.5]{MR2015057}}]\label{prop4}
Let $\lambda$ be a dominant character. For each $i>0$, we have
\begin{equation*}
H^i(G/B,\Lc_{\lambda}) = 0.
\end{equation*}
More generally, let $P$ be a standard parabolic of type $I \subset \Delta$ and $\lambda$ a $I$-dominant character of $P$(i.e. $\langle \lambda , \alpha^{\vee} \rangle \geq 0$ for all $\alpha \in \Delta \backslash I$ and $\langle \lambda , \alpha^{\vee} \rangle = 0$ for all $\alpha \in I$). There is an associated line bundle $\mathcal{L}_{\lambda}$ on $G/P$ and we have 
\begin{equation*}
H^i(G/P,\mathcal{L}_{\lambda}) = 0,
\end{equation*}
for all $i>0$.
\end{proposition}
\begin{proof}
We give a sketch of the argument. The first step is to show that $\mathcal{L}_{\lambda}$ is ample over the flag variety $G/B$ exactly when $\lambda$ is strictly dominant by reducing to the case $G = \SL_2$ and $G/B = \mathbb{P}^1_k$. Then, in characteristic $0$, we can conclude with the Kodaira-Nakano vanishing theorem since the canonical bundle $\omega_{G/B}$ of $G/B$ is anti-ample. Indeed, we have an isomorphism
\begin{equation*}
\omega_{G/B} = \mathcal{L}_{-2\rho}
\end{equation*}
and if we consider a dominant character $\lambda$, the line bundle
\begin{equation*}
\omega_{G/B}^{-1} \otimes_{\mathcal{O}_{G/B}} \mathcal{L}_{\lambda} = \mathcal{L}_{2\rho+\lambda}
\end{equation*}
is ample since $2\rho+\lambda$ is strictly dominant. The Kodaira-Nakano vanishing theorem applied to $\mathcal{L}_{2\rho+\lambda}$ says that
\begin{equation*}
H^i(G/B,\underbrace{\omega_{G/B} \otimes \mathcal{L}_{2\rho+\lambda}}_{= \mathcal{L}_{\lambda}}) = 0,
\end{equation*}
for all $i > 0$. In positive characteristic, we can conclude with Serre's cohomological criterion for ampleness and the formula in \cite[Part II, sect. 3.19]{MR2015057} with the Steinberg module $\nabla((p^r-1)\rho)$. 
\end{proof}
We insist on the fact that the proof in \cite[Part II, sect. 5.3]{MR2015057} of the more general Borel-Weil-Bott theorem which gives information on the higher cohomology groups of $\Lc_\lambda$ when $\lambda$ is no longer dominant requires to divide by binomial numbers $\binom{n}{k}$ with $n \geq p$, which is impossible in characteristic $p$. Actually, one can find counterexamples to the Borel-Weil-Bott theorem in positive characteristic (see \cite[Part II, sect. 15.8]{MR2015057}). In characteristic $0$, it is easier to understand tensor product of highest weight representations: we know that $L(\lambda) \otimes L(\mu)$ is a direct sum of $L(\lambda^\prime)$ where $\lambda^\prime$ can be expressed as $\lambda + \mu^\prime$ where $\mu^\prime \leq \mu$ is a weight of $L(\mu)$. Going back to our positive characteristic case, we would like to have a weaker but similar kind of result for $\nabla(\lambda)$'s.
\begin{definition}\label{def6}
Let $V$ be an algebaic representation of $G$. We say that:
\begin{enumerate}
\item $V$ admits a $\nabla$-filtration if there is 
a finite filtration \begin{equation*}
0 =V^n \subsetneq V^{n-1} \subsetneq
 \cdots \subsetneq
 V^0 = V
\end{equation*}
with graded pieces 
\begin{equation*}
V^i/V^{i+1} \simeq \nabla(\nu_i),
\end{equation*}
for some dominant characters $\nu_i$.
\item $V$ admits a $\Delta$-filtration if there is 
a finite filtration \begin{equation*}
0 =V^n \subsetneq V^{n-1} \subsetneq
 \cdots \subsetneq
 V^0 = V
\end{equation*}
with graded pieces 
\begin{equation*}
V^i/V^{i+1} \simeq \Delta(\nu_i),
\end{equation*}
for some dominant characters $\nu_i$.
\end{enumerate}
\end{definition}
\begin{rmrk}
In the setting of a highest weight category, tilting modules are defined as modules that admit both a $\nabla$- and a $\Delta$-filtration.
\end{rmrk}
\vspace{0.2cm}
The following proposition states the existence of a $\nabla$-filtration for a tensor product $\nabla(\lambda) \otimes \nabla(\mu)$ and gives some details about its graded pieces. This result is due to Donkyn \cite{MR804233} when $G$ does not contain any components of type $E_7, E_8$ or that $p \neq  2$. His approach relies on a case by case analysis of each Dynkin diagram and requires long and difficult calculations. A more general proof, without the technical restrictions, was given later by Mathieu. We first need a lemma.
\begin{lemma}\label{lem3}
Let $\lambda$, $\mu$ denote $T$-characters such that $\Ext^1_{G}(\nabla(\lambda),\nabla(\mu)) \neq 0$. Then, $\lambda > \mu$.
\end{lemma}
\begin{proof}
We have
\begin{equation*}
\begin{aligned}
\Ext^1_{G}(\nabla(\lambda),\nabla(\mu)) &= H^1(G,\Delta(-w_0\lambda)\otimes \nabla(\mu)) \\
&= H^1(B,\Delta(-w_0\lambda)\otimes \mu) \text{ by \cite[Part II, sect. 4.7]{MR2015057}}.\\
\end{aligned}
\end{equation*}
Filtering $\Delta(-w_0\lambda)\otimes \mu$ by $B$-submodules with one-dimensional graded pieces, we find a weight $\nu$ of $\Delta(-w_0\lambda)$ such that $H^1(B,\chi) \neq 0$ for $\chi = \nu + \mu$. Since $H^1(B,k) = H^1(G/B,\mathcal{O}) = 0$, we have $\chi \neq 0$, and by \cite[Part II, sect. 4.10 b)]{MR2015057}, $-\chi = -(\nu + \mu)$ is a non-zero $\N$-linear combination of positive roots $\phi^+$. In particular, we have $-\nu > \mu$. Since $w_0(-w_0\lambda) = - \lambda$ is the lowest weight of $\Delta(-w_0\lambda)$, we deduce that $\lambda \geq -\nu > \mu$.
\end{proof}
\begin{proposition}[{\cite{MR1072820}}]\label{prop8}
Let $\lambda, \mu$ be two dominant characters in $X^*(T)$. Then $\nabla(\lambda) \otimes \nabla(\mu)$ admits a $\nabla$-filtration $(V^i)_{i \geq 0}$ with graded pieces  
\begin{equation*}
V^i/V^{i+1} \simeq \nabla(\lambda+\mu_i),
\end{equation*}
where $(\mu_i)_{i}$ is a collection of weights of $\nabla(\mu)$ with $\mu_0 = \mu$. In particular, the first graded piece is given by $V^0/V^1 = \nabla(\lambda+\mu)$.
\end{proposition}
\begin{proof}
We add some details to the result of Mathieu to explain how to get a filtration with the desired properties. The result of Mathieu assumes that $G$ is a connected, simply-connected, semi-simple algebraic group over an algebraically closed field $k$ of characteristic $p > 0$ and it is not hard to reduce to this case. By \cite[Theorem 1]{MR1072820}, there exists a filtration 
\begin{equation*}
0 = V^n \subset V^{n-1} \subset \cdots \subset V^1 \subset \cdots V^0 = \nabla(\lambda) \otimes \nabla(\mu),
\end{equation*}
where for each $i$ the graded piece $V^i/V^{i+1}$ is a costandard module $\nabla(\nu_i)$ for some dominant character $\nu_i$. The character class of $\nabla(\lambda) \otimes \nabla(\mu)$ is 
\begin{equation*}
\ch(\nabla(\lambda) \otimes \nabla(\mu)) = \sum_{i} \ch \nabla(\lambda + \mu_i),
\end{equation*}
where the sum is taken over some weights $(\mu_i)_i$ of $\nabla(\mu)$. As the highest weight of this module, $\lambda + \mu$ contributes to the sum. Note that the non-zero terms are those such that $\lambda + \mu_i$ is dominant. We choose an ordering of the $(\mu_i)_i$ such that whenever $\mu_i < \mu_j$ for some $i,j$ then $i > j$. It implies that there exists a permutation $\sigma$ on $0,1,\cdots,n-1$ such that 
\begin{equation*}
V^i/V^{i+1} = \nabla(\lambda + \mu_{\sigma(i)}),
\end{equation*}
for all $i$ between $0$ and $n-1$. We remake the argument in \cite[Part 11, sect. 4.16, remark 4]{MR2015057} to explain how to reorganize the terms. If $\sigma(i) > \sigma(i+1)$ for some $i$, $0 \leq i \leq n-2$, then $\lambda + \mu_{\sigma(i+1)} \nless \lambda + \mu_{\sigma(i)}$ and the exact sequence 
\begin{equation*}
\begin{tikzcd}
0 \arrow[r] & \nabla(\lambda+\mu_{\sigma(i+1)}) \arrow[r] & V^{i}/V^{i+2} \arrow[r] & \nabla(\lambda+\mu_{\sigma(i)}) \arrow[r] & 0
\end{tikzcd}
\end{equation*} 
is split because $\Ext^1_G(\nabla(\lambda+\mu_{\sigma(i)}),\nabla(\lambda+\mu_{\sigma(i+1)})) = 0$ by lemma \ref{lem3}. It shows that 
\begin{equation*}
V^{i}/V^{i+2} = \nabla(\lambda+\mu_{\sigma(i+1)}) \oplus \nabla(\lambda+\mu_{\sigma(i)})
\end{equation*}
and we can replace $V^{i+1}$ by a submodule $\tilde{V}^{i+1}$ between $V^{i+2}$ and $V^i$ such that $\tilde{V}^{i+1}/V^{i+2} = \nabla(\lambda+\mu_{\sigma(i)})$ and $V^{i}/\tilde{V}^{i+1} = \nabla(\lambda+\mu_{\sigma(i+1)})$. We iterate this process to produce the desired filtration.
\end{proof}
\begin{rmrk}
\begin{enumerate}
\item Not all the weights $\mu^\prime \leq \mu$ of $\nabla(\mu)$ such that $\lambda + \mu^\prime$ is dominant will contribute to the filtration.
\item Even if we will not need it, we note that the dual statement says that tensor products of standard modules $\Delta(\lambda) \otimes \Delta(\mu)$ admit a $\Delta$-filtration.
\end{enumerate}
\end{rmrk}
\begin{corollary}\label{cor2}
Let $V$ and $W$ be two algebraic representations of $G$ that admit a $\nabla$-filtration. Then, $V \otimes W$ admits a $\nabla$-filtration.
\end{corollary}
\vspace{0.2cm}
We recall the Donkyn criterion.
\begin{proposition}
Let $V$ be an algebraic representation of $G$. The following proposition are equivalent.
\begin{enumerate}
\item $V$ admits a $\nabla$-filtration.
\item For all dominant characters $\lambda$ and $i>0$, $\Ext_G^i(\Delta(\lambda),V) = 0$.
\item For all dominant characters $\lambda$, $\Ext_G^1(\Delta(\lambda),V) = 0$.
\end{enumerate} 
\end{proposition}
\begin{proof}
See \cite[Part II, sect. 4.16]{MR2015057}.
\end{proof}
\begin{corollary}\label{cor1}
Let $V$ and $W$ be two algebraic representations of $G$. If $V$ admits a $\nabla$-filtration and $W$ is a direct factor of $V$, then $W$ admits a $\nabla$-filtration.
\end{corollary}

%% file: part3.tex
\section{Recollection on Siegel varieties}\label{part3}
In this section, we follow \cite{MR1083353} for generalities about Siegel varieties. We denote by $\Sch_{\F_p}$ the category of schemes over $\F_p$ and $\mathbb{A}_f$ the finite adeles of $\Q$.
\begin{definition}
Let $A$ and $A^\prime$ be abelian schemes of relative dimension $g$ over a scheme $S$. A quasi-isogeny $A \rightarrow A^\prime$ is an equivalence class of pairs $(\alpha,N)$ where $\alpha : A \rightarrow A^\prime$ is an isogeny over $S$ and $N$ is a positive integer with the relation 
\begin{equation*}
(\alpha,N) \sim (\alpha^\prime,N^\prime) \text{ if and only if } N^\prime\alpha = N \alpha^\prime.
\end{equation*}
\end{definition}
\begin{definition}\label{def7}
Let $V$ be the $\mathbb{Z}$-module $\mathbb{Z}^{2g}$ with the standard non-degenerate symplectic pairing 
\begin{equation*}
\psi : V \times V \rightarrow \mathbb{Z}
\end{equation*}
such that $\psi(x,y) = {}^txJx$ where
\begin{equation*}
J = \begin{pmatrix}
0 & I_g \\
-I_g & 0 
\end{pmatrix}.
\end{equation*}
We denote by $\Sp_{2g}$ the algebraic group over $\Z$ of $2g \times 2g$ matrices $M$ that preserves the symplectic pairing $\psi$, i.e. such that
\begin{equation*}
{}^tMJM = J.
\end{equation*}
\end{definition}
In the following proposition, we define the Siegel modular variety of level $K$ as a scheme over $\F_p$ when the level is small enough and $p$ is a prime such that $K_p = \Sp_{2g}(\Z_p)$.
\begin{proposition}[\cite{MR1083353}]\label{prop1}
Let $K \subset \Sp_{2g}(\mathbb{A}_f)$ be a subgroup that can be written as $K = K_pK^p$ for $K_p \subset \Sp_{2g}(\mathbb{Q}_p)$ hyperspecial and $K^p \subset \Sp_{2g}(\mathbb{A}_f^p)$ compact open. Consider the fibered category in groupoids $\mathcal{A}_{g,K}$ on $\Sch_{\F_p}$ whose $S$-points are groupoids with
\begin{itemize}
\item Objects: $(A/S,\lambda,\psi)$ where $A/S$ is abelian scheme over $S$ of relative dimension $g$, $\lambda : A \rightarrow A^{\vee}$ is a $\mathbb{Z}_{(p)}$-multiple of a principal polarization and for all primes $l \neq p$ and all geometric points $s \in S$, $\psi_l$ is a $K_l$-orbit of symplectic \emph{isomorphisms} from $H_1(A_s,\mathbb{Q}_l)$ to $V\otimes \mathbb{Q}_l$ which is invariant under $\pi_1(S,s)$. The structure of symplectic $\mathbb{Q}_l$-vector space on the $l$-adic étale homology group $H_1(A_s,\mathbb{Q}_l)$ (it is also the rational Tate module of $A_s$) is the one induced by the polarization (which is an isomorphism since we tensor by $\mathbb{Q}_l$) and the Weil paring.
\item Morphisms: A morphism $(A/S,\lambda,\psi) \rightarrow (A^\prime/S,\lambda^\prime,\psi^\prime)$ is a quasi-isogeny $\alpha : A \rightarrow A^\prime$ over $S$ such that the diagram
\begin{equation*}
\begin{tikzcd}
A \arrow[r,"\alpha"] \arrow[d,"\lambda"] & A^{\prime} \arrow[d,"\lambda^\prime"] \\
A^{\vee} & {A^\prime}^{\vee} \arrow[l,"\alpha^{\vee}"]
\end{tikzcd}
\end{equation*}
is commutative up to a constant in $\mathbb{Z}_{(p)}$ and the pullback of $\psi_l$ by the quasi-isogeny $\alpha$ is $\psi^\prime_l$.
\end{itemize}
If the level away from $p$, $K^p$, is small enough\footnote{It is the case in particular when $K$ is the kernel of the reduction map $\Sp_{2g}(\Z) \rightarrow \Sp_{2g}(\Z/N\Z)$ with $N \geq 3$ such that $p \nmid N$.}, then $\mathcal{A}_{g,K}$ is representable by a smooth integral quasi-projective scheme over $\F_p$.
\end{proposition}
\begin{rmrk}
Without the hypothesis on the smallness of $K$, $\mathcal{A}_{g,K}$ is only a Deligne-Mumford stack over $\F_p$.
\end{rmrk}
\begin{notation}\label{not1}
We fix some notation for the rest of this section. Let $G$ denote the algebraic group $\Sp_{2g}$ over $\F_p$ where $g \geq 1$. We fix a neat finite level $K$ that can be written as $K = K_pK^p$ for $K_p \subset G(\mathbb{Q}_p)$ hyperspecial and $K^p \subset G(\mathbb{A}_f^p)$ compact open. Let $\Sh$ denote the smooth quasi-projective variety $\mathcal{A}_{g,K}$ over $\mathbb{F}_p$. Let $\mu : \mathbb{G}_m \rightarrow G$ denote the minuscule cocharacter that stabilizes the Hodge filtration\footnote{It maps $z$ to $\begin{pmatrix}
zI_g & 0 \\ 0 & z^{-1}I_g
\end{pmatrix}$ with our choice of symplectic pairing in definition \ref{def7}.} and let $P^+ := P_{\mu}, P:=P_{-\mu}$ denote the associated opposite parabolic subgroups with common Levi subgroup $L = \GL_g$ over $\F_p$. We consider the Borel $B$ of upper triangular matrices in $G = \Sp_{2g}$, so that $B \subset P$. We write $\phi_L$ (resp. $\phi_L^+$) for the roots of $L$ (resp. positive roots of $L$).
\end{notation}
\vspace{0.2cm}
Denote by $\pi : A \rightarrow \Sh$ the universal abelian scheme and $e : \Sh \rightarrow A$ its neutral section. The universal polarization of $A$ gives to the algebraic de Rham cohomology $\mathcal{H}^1_{\dR}$ of $A$ over $\Sh$ the structure of a $\Sp_{2g}$-torsor over $\Sh$. 
\begin{proposition}
The de Rham cohomology $\mathcal{H}^1_{\dR}$ is equipped with a Hodge filtration 
\begin{equation*}
0 \rightarrow \Omega \rightarrow  \mathcal{H}^1_{\dR} \rightarrow \Omega^{\vee} \rightarrow 0,
\end{equation*}
where
\begin{equation*}
\Omega = e^*\Omega^1_{A/\Sh}
\end{equation*}
and
\begin{equation*}
\Omega^{\vee} = R^1\pi_*\mathcal{O}_A.
\end{equation*} 
We call $\Omega$ the Hodge vector bundle, it is a locally free sheaf of rank $g$ over $\Sh$. Moreover, the Hodge bundle $\Omega$ is totally isotropic for the symplectic pairing on $\mathcal{H}^1_{\dR}$ which allows us to identify the Hodge filtration with a $P$-torsor on $\Sh$.
\end{proposition}
\begin{proof}
The Hodge filtration comes from the degeneration at the second page of the Hodge-de Rham spectral sequence which is a result of Deligne and Illusie \cite{MR894379} in the case of abelian schemes. The vector bundle $\Omega$ is locally free of rank $g$ because $\pi : A \rightarrow \Sh$ is smooth. Actually, we also have an isomorphism 
\begin{equation*}
\Omega \simeq \pi_*\Omega^1_{A/\Sh}.
\end{equation*}
Indeed, as a group scheme $\pi$ satisfies
\begin{equation*}
\Omega^1_{A/\Sh} = \pi^*e^*\Omega^1_{A/\Sh}
\end{equation*}
and for any proper morphism $f : X \rightarrow Y$ with geometrically connected fibers, we have
\begin{equation*}
f_*\mathcal{O}_X = \mathcal{O}_Y.
\end{equation*}
From the projection formula, we deduce 
\begin{equation*}
\pi_*\Omega^1_{A/\Sh} = \pi_*(\pi^*e^*\Omega^1_{A/\Sh} \otimes \mathcal{O}_A) = e^*\Omega^1_{A/\Sh} \otimes \pi_*\mathcal{O}_A = \Omega.\qedhere
\end{equation*}
\end{proof}

The Siegel modular variety $\Sh$ is not proper but we can consider a toroidal compactification.
\begin{definition}[{\cite[Chapter 4]{MR1083353}\cite[Th. 2.15]{MR2968629}}]\label{prop5}
Let $C$ be the cone of all positive semi-definite symetric bilinear forms on $X^*(T)\otimes_{\mathbb{Z}} \mathbb{R}$ whose radicals are defined over $\mathbb{Q}$. Let $\Sigma = \{ \sigma_{\alpha} \}_{\alpha}$ be a smooth $GL(X^*(T))$-admissible decomposition in polyhedral cones of $C$ as defined in \cite[Chapter 4, Definition 2.2/2.3]{MR1083353}. We assume that $\Sigma$ admits a $GL(X^*(T))$-equivariant polarization function as defined in \cite[Chapter 4, Definition 2.4]{MR1083353}. See \cite{MR2590897} or \cite{MR0335518} for a proof of the existence of such polyhedral cone decompositions. We consider the corresponding toroidal compactification $\Sh^{\tor,\Sigma}$ of the Siegel modular variety $\Sh$. It is a smooth and projective scheme over $\mathbb{F}_p$ which satisfies the following properties.
\begin{enumerate}
\item The complementary $D_{\red} = \Sh^{\tor,\Sigma} - \Sh$, when endowed with its reduced structure, is a Cartier divisor with normal crossings, which we assume strict, as a suitable choice of $\Sigma$ allows.
\item The universal abelian scheme $f : A \rightarrow \Sh$ extends to a semi-abelian scheme $f^{\tor} : A^{\tor} \rightarrow \Sh^{\tor}$ such that
\item $\Omega^{\tor} := e^*\Omega^1_{A^{\tor}/\Sh^{\tor,\Sigma}}$ is a vector bundle that extends the Hodge bundle to $\Sh^{\tor,\Sigma}$.
\item By \cite[Chapter 4]{MR1083353} or \cite[Th. 2.15, (2)]{MR2968629} the semi-abelian scheme $f^{\tor} : A^{\tor} \rightarrow \Sh^{\tor}$ can be compactified into a proper and log-smooth scheme $\bar{f}^{\tor} : \bar{A}^{\tor} \rightarrow \Sh^{\tor}$ which is projective and smooth over $\mathbb{F}_p$
\begin{equation*}
\begin{tikzcd}
A \arrow[r,"f"] \arrow[d] & A^{\tor} \arrow[d,"f^{\tor}"] \arrow[r] &\bar{A}^{\tor} \arrow[ld,"\bar{f}^{\tor}"] \arrow[d] \\
\Sh \arrow[r] & \Sh^{\tor} \arrow[r] & \Spec \mathbb{F}_p
\end{tikzcd}
\end{equation*}
and we denote again $D_{\red}$ the normal crossing divisor $\bar{A}^{\tor} - A$.
\item Following \cite[Chapter 4]{MR1083353} or \cite[Th. 2.15, (3)]{MR2968629}, the log-de Rham complex $\bar{\Omega}^{\bullet}_{\bar{A}^{\tor}/\Sh^{\tor}}$ is the complex of log-differentials $\bar{\Omega}^i_{\bar{A}^{\tor}/\Sh^{\tor}} := \Lambda^i\bar{\Omega}^1_{\bar{A}^{\tor}/\Sh^{\tor}}$ where
\begin{equation*}
\bar{\Omega}^1_{\bar{A}^{\tor}/\Sh^{\tor}} = \Omega^1_{\bar{A}^{\tor}}(\log D_{\red}) /  {(\bar{f}^{\tor})}^{*}  \Omega^1_{\Sh^{\tor}}(\log D_{\red}).
\end{equation*}
\item and the log-de Rham cohomology
\begin{equation*}
\mathcal{H}^1_{\log-\dR} := R^1{(\bar{f}^{\tor})}_*\bar{\Omega}^{\bullet}_{\bar{A}^{\tor}/\Sh^{\tor}}
\end{equation*}
is a $\Sp_{2g}$-torsor that extends $\mathcal{H}^1_{\dR}$ on $\Sh$.
\item and the logarithmic Hodge-de Rham spectral
sequence 
\begin{equation*}
E^{i,j}_1 = R^j{(\bar{f}^{\tor})}_*\bar{\Omega}^{i}_{\bar{A}^{\tor}/\Sh^{\tor}} \Rightarrow \mathcal{H}^i_{\log-\dR} := R^i{(\bar{f}^{\tor})}_*\bar{\Omega}^{\bullet}_{\bar{A}^{\tor}/\Sh^{\tor}}
\end{equation*}
degenerates at page $1$. It defines a $P$-reduction of the $\Sp_{2g}$-torsor $\mathcal{H}^1_{\log-\dR}$ on $\Sh^{\tor,\Sigma}$ that extends the Hodge fitlration on $\Sh$.
\end{enumerate}
\end{definition}
From now on, we drop the upperscript $\Sigma$ to denote $\Sh^{\tor,\Sigma}$ since the coherent cohomology does not depend on this choice.
\begin{definition}[\cite{MR1083353}]\label{def4}
We define the minimal compactification as the projective scheme 
\begin{equation*}
\Sh^{\min}:= \Proj(\oplus_{n\geq 0} H^0(\Sh^{\tor},\omega^{\otimes n})),
\end{equation*}
where $\omega = \det \Omega^{\tor}$ is the Hodge line bundle. 
\end{definition}
\vspace{0.2cm}
The minimal compactification $\Sh^{\min}$ is a normal and projective variety (independent of the choice of $\Sigma$) but it is not smooth in general. Moreover, the Hodge line bundle $\omega$ descends to an ample line bundle on $\Sh^{\min}$. From this construction, one can see that $\Sh^{\tor}$ is the normalization of the blow-up of $\Sh^{\min}$ along a coherent sheaf of ideals $\mathcal{J}$ of $\mathcal{O}_{\Sh^{\min}}$ and we write
\begin{equation*}
\varphi : \Sh^{\tor} \rightarrow \Sh^{\min}
\end{equation*}
for the induced morphism. The pullback $\varphi^*\mathcal{J}$ is of the form $\mathcal{O}_{\Sh^{\tor}}(-D)$ where $D$ is an effective Cartier divisor whose associated reduced Cartier divisor is $D_{\red}$. In particular, we deduce that there exists $\eta_0 > 0$ such that $\omega^{\otimes \eta}(-D)$ is ample on $\Sh^{\tor}$ for every $\eta \geq \eta_0$. In general, $\omega$ fails to be ample on $\Sh^{\tor}$. 
\begin{rmrk}
The effective Cartier divisor $D$ depends on the choice of the $GL(X^*(T))$-equivariant polarization function on the decomposition in polyhedral cones $\Sigma$.
\end{rmrk}
\vspace{0.2cm}
We are now able to define automorphic vector bundles with contracted products.
\begin{definition}\label{def3}
Let $V$ be a finite dimensional algebraic representation of $L = \GL_{g}$. We define the associated vector bundle $\mathcal{W}(V)$ on $\Sh$ (resp. its canonical extension to $\Sh^{\tor}$) to be the contracted product of $V$ with the $\GL_{g}$-torsor $\Omega$ (resp. $\Omega^{\tor}$). If $\lambda \in X^*(T)$ is an $L$-dominant character, we write simply $\nabla(\lambda)$ for the vector bundle corresponding to the induced representation $H^0(L/B_L, \mathcal{L}_{\lambda})$ of $L$. It corresponds to the costandard representation of highest weight $w_0w_{0,L}\lambda$.
\end{definition}
\begin{rmrk}
We apologize for the weird convention in definition \ref{def3}. The advantage of this convention is to keep an easy formula in proposition \ref{prop13}. 
\end{rmrk}
\vspace{0.2cm}
We recall the Kodaira-Spencer isomorphism. 
\begin{proposition}[{\cite[Chap. 3, sect. 9]{MR1083353}}]\label{prop19}
The Kodaira-Spencer map
\begin{equation*}
\rho_{\text{KS}} : \Sym^2 \Omega \rightarrow \Omega^1_{\Sh^{\tor}}(\log D_{\red})
\end{equation*}
is an isomorphism. This allows us to identify the logarithmic differentials $\Omega^1_{\Sh^{\tor}}(\log D_{\red})$ with the automorphic vector bundle $\mathcal{W}(\Sym^2 \std_L) = \nabla(0, \cdots, 0,-2)$. In particular, we have an isomorphism of line bundles
\begin{equation*}
\Omega^d_{\Sh^{\tor}}(\log D_{\red}) \simeq \nabla({-2\rho^L}),
\end{equation*} 
where $d$ is the dimension of $\Sh^{\tor}$ and
\begin{equation*}
\rho^L = \frac{1}{2} \sum_{\alpha \in \phi^{+} \backslash \phi_L^{+}} \alpha.
\end{equation*}
\end{proposition}
Recall that the Hodge filtration on $\Sh$ is canonically identified with a $P$-torsor that extends to the toroidal compactification $\Sh^{\tor}$. From now on, $I_0$ denotes a subset of $I$ and $P_0$ denotes its associated intermediate parabolic subgroup $B \subset P_0 \subset P$.
\begin{definition}\label{def2}
Let $S \rightarrow \Sh$ be a $S$-valued point of $\Sh$ and denote by $A/S$ the corresponding abelian scheme. The flag bundle $Y_{I_0}$ is the scheme over $\Sh$ whose $S$-valued points are $P_0$-reductions of the $P$-torsor corresponding to the Hodge filtration of $A$. 
\end{definition}
\vspace{0.2cm}
From the definition of $Y_{I_0}$, we get a smooth proper morphism $\pi : Y_{I_0} \rightarrow \Sh$ where each fibre is isomorphic to the flag variety $P/P_0$.
\begin{rmrk}
In the special case $I_0 = I$, the flag bundle $Y_{I_0}$ coincide with the Siegel modular variety $\Sh$.
\end{rmrk}
\vspace{0.2cm}
The flag bundle $Y_{I_0}$ extends to a flag bundle $Y_{I_0}^{\tor}$ over the toroidal compactification $\Sh^{\tor}$ because we have seen that the Hodge filtration over $\Sh$ extends to $\Sh^{\tor}$ in definition \ref{prop5} $(7)$. It implies by base change that the universal $P_0$-torsor on $Y_{I_0}$ extends to $Y_{I_0}^{\tor}$. This allows us to define automorphic vector bundles on $Y_{I_0}$ from algebraic representations of $P_0$. 
\begin{definition}\label{def9}
Let $V$ be a finite dimensional algebraic representation of $P_0$. We define the associated vector bundle $\mathcal{L}(V)$ on $Y_{I_0}$ (resp. its canonical extension on $Y_{I_0}^{\tor}$) as the contracted product of $V$ with the universal $P_0$-torsor on $Y_{I_0}$ (resp. $Y_{I_0}^{\tor}$). If $\lambda \in X^*(P_0)$ is a character, we write simply $\mathcal{L}_{\lambda}$ for the corresponding line bundle.
\end{definition}
\begin{rmrk}
In the special case where $I_0 = \emptyset$, note that we have $X^*(P_0) = X^*(T)$. 
\end{rmrk}
\vspace{0.2cm}
There is an easy relation between $\mathcal{L}_{\lambda}$ and $\nabla(\lambda)$ that we want to explain. We first recall the proper base change theorem for coherent cohomology.

\begin{proposition}[Proper base change, non reduced case]\label{prop3}
Let $f : X\rightarrow S$ be a proper morphism between locally noetherian schemes. Let $\mathcal{F}$ be a coherent sheaf over $X$ which is flat over $S$. Let $p\geq 0$ and $s\in S$. If $\theta_s^p : (R^pf_*\mathcal{F})_s \otimes_{\mathcal{O}_{S,s}} k(s) \rightarrow H^p(X_s,\mathcal{F}_{|X_s})$ is surjective, then there is an open neighbourhood $U$ of $s$ such that for all $s^\prime \in U$, $\theta^p_{s^\prime}$ is an isomorphism and the following conditions are equivalent
\begin{enumerate}
\item $\theta_{s}^{p-1}$ is surjective,
\item $R^pf_*\mathcal{F}$ is free on $U$,
\end{enumerate}
and under these conditions, the formation of $R^pf_*\mathcal{F}$ commutes under base change. This means that for any $g : S^\prime \rightarrow S$, we have $g^*R^pf_*\mathcal{F} \simeq R^pf^\prime_*{g^\prime}^*\mathcal{F}$ where the maps are defined in the following cartesian diagram
\begin{equation*}
\begin{tikzcd}
X^\prime \arrow[d,"f^\prime"]  \arrow[r,"g^\prime"] & X \arrow[d,"f"] \\
S^\prime \arrow[r,"g"] & S
\end{tikzcd}
\end{equation*}
\end{proposition}
\begin{proof}
See \cite[Part III, Theorem 12.11]{MR0463157}.
\end{proof}
\begin{rmrk}
\begin{enumerate}
\item We assume $\theta^{-1}_s$ to be the zero morphism.
\item The reference in proposition \ref{prop3} states a coherent base change theorem only for geometric points of $S$. To see how it implies the base change for any morphism $S^\prime \rightarrow S$, see \cite[Proposition 2.1]{conradbasechange}.
\end{enumerate}
\end{rmrk}
\begin{lemma}\label{lem2}
Let $\mathcal{X}$ and $\mathcal{Y}$ be two Artin stacks and $\pi : \mathcal{Y} \rightarrow \mathcal{X}$ a proper representable morphism. Let $\mathcal{L}$ be a coherent sheaf over $\mathcal{Y}$, flat over $\mathcal{X}$, such that for all geometric points $x : \Spec \bar{k} \rightarrow \mathcal{X}$ fitting in the cartesian diagram
\begin{equation*}
\begin{tikzcd}
\mathcal{Y}_x := \mathcal{Y}\times_{\mathcal{X},x}\Spec \bar{k} \arrow[d,"\pi_x"] \arrow[r,"i"] & \mathcal{Y} \arrow[d,"\pi"] \\
\Spec \bar{k} \arrow[r,"x"] & \mathcal{X}
\end{tikzcd}
\end{equation*}
the complex $R{(\pi_x)}_*\mathcal{L}_{|\mathcal{Y}_x}$ is concentrated in degree $0$. Then, the complex $R\pi_*\mathcal{L}$ is also concentrated in degree $0$.
\end{lemma}
\begin{proof}
Consider a presentation $f : X \rightarrow \mathcal{X}$ of the Artin stack $\mathcal{X}$ where $X$ is a scheme and $f$ is a surjective and smooth morphism. Consider the double cartesian diagram 
 \begin{equation*}
\begin{tikzcd}
{Y}_x := {Y}\times_{{X},x}\Spec \bar{k} \arrow[d,"\pi_x"] \arrow[r,"i"] & Y := X \times_{\mathcal{X}} \mathcal{Y} \arrow[d,"\tilde{\pi}"]\arrow[r,"f^\prime"] & \mathcal{Y} \arrow[d,"\pi"] \\
\Spec \bar{k} \arrow[r,"x"] & X \arrow[r,"f"] &  \mathcal{X}
\end{tikzcd}
\end{equation*}
where $x$ is a geometric point of $X$. For any $i > 0$, we have $H^i(Y_x,\mathcal{L}_{|Y_x}) = 0$ by hypothesis. As a consequence, the base change morphism for the first cartesian diagram 
\begin{equation*}
\theta_x^i : R^i\tilde{\pi}_*\mathcal{L}_{|Y} \otimes_{\mathcal{O}_{X,x}} k(x) \rightarrow H^i(Y_x,\mathcal{L}_{|Y_x}) 
\end{equation*}
is surjective. By proposition \ref{prop3}, we deduce that $\theta_{x^\prime}^i$ is an isomorphism for all $x^\prime$ in a neighborhood of $x$. We deduce that $R^i\tilde{\pi}_*\mathcal{L}_{|Y}$ is zero for all $i>0$. Since $f$ is flat, the base change theorem for the second cartesian diagram says that there is an isomorphism
\begin{equation*}
f^*\circ R\pi_*\mathcal{L} \rightarrow R\tilde{\pi}_*\circ {(f^\prime)}^*\mathcal{L}.
\end{equation*}
Since $f$ is faithfully flat, it implies that $R\pi_*\mathcal{L}$ is concentrated in degree 0. 
\end{proof}
\begin{proposition}\label{prop13}
Let $\lambda$ be a character of $P_0$. Denote by $\pi : Y_{I_0} \rightarrow \Sh$ the flag bundle defined before. We have a canonical isomorphism
\begin{equation*}
\pi_* \mathcal{L}_{\lambda} \simeq \nabla(\lambda),
\end{equation*}
where we see $\lambda$ as a character of $T$ to construct $\nabla(\lambda)$. More generally, for every finite dimensional algebraic representation $V$ of $P_0$ and every $k \geq 0$, the sheaf $R^k\pi_*\mathcal{L}(V)$ is the vector bundle on $\Sh$ associated to the $P$-module $H^k(P/P_0,\mathcal{L}(V))$ by the universal $P$-torsor. These isomorphisms extend to the toroidal compactifications $Y_{I_0}^{\tor}$ and $\Sh^{\tor}$.
\end{proposition}
\begin{proof}
This isomorphism is a formal consequence of the definition of automorphic vector bundles in definitions \ref{def3} and \ref{def9} and standard base change theorem. We have a cartesian diagram
\begin{equation*}
\begin{tikzcd}
Y_{I_0} \arrow[d,"\pi"] \arrow[r,"\tilde{\zeta}"] & \lfloor P_0 \backslash * \rfloor \arrow[d,"\tilde{\pi}"] \\
\Sh \arrow[r,"\zeta"] & \lfloor P \backslash * \rfloor
\end{tikzcd}
\end{equation*}
where the horizontal arrows corresponds to the universal $P$-torsor on $\Sh$ and the universal $P_0$-torsor on $Y_{I_0}$ and where the vertical arrow $\tilde{\pi}$ between the classifying stacks is induced by the inclusion $P_0 \subset P$. For every $\lambda \in X^*(P_0)$, we have a line bundle $\mathcal{L}_{\lambda}$ on the classifying stack of $P_0$. We denote by $\nabla(\lambda)$ the vector bundle on the classifying stack of $P$ associated to the $P$-module $H^0(P/P_0,\mathcal{L}_{\lambda})$. By definition, we have isomorphisms
\begin{equation*}
\left\{
\begin{aligned}
&\tilde{\pi}_* \mathcal{L}_{\lambda} = \nabla(\lambda), \\
&\tilde{\zeta}^*\mathcal{L}_{\lambda} = \mathcal{L}_{\lambda}, \\
&\zeta^*\nabla(\lambda)= \nabla(\lambda),
\end{aligned}
\right.
\end{equation*}
on $ \lfloor P \backslash * \rfloor$. Since $\zeta$ is flat, we have a base change theorem in the derived category of quasi-coherent sheaves over $\Sh$ which says that the natural map
\begin{equation*}
\zeta^* \circ {R}\tilde{\pi}_* \mathcal{L}_{\lambda}  \rightarrow {R}\pi_*\circ \tilde{\zeta}^* \mathcal{L}_{\lambda} 
\end{equation*}
is an isomorphism. Since $\zeta^*$ is exact, taking cohomology sheaves gives
\begin{equation*}
\zeta^* R^k\tilde{\pi}_* \mathcal{L}_{\lambda} \simeq R^k\pi_* \mathcal{L}_{\lambda}
\end{equation*}
for every $k \geq 0$, and for $k = 0$ we get
\begin{equation*}
\pi_* \mathcal{L}_{\lambda} \simeq \nabla(\lambda).
\end{equation*}
The same argument applies verbatim to $\mathcal{L}(V)$ for any finite dimensional algebraic representation $V$ of $P_0$, the sheaf $R^k\tilde{\pi}_*\mathcal{L}(V)$ on $\lfloor P \backslash * \rfloor$ being the vector bundle associated to the $P$-module $H^k(P/P_0,\mathcal{L}(V))$. If $\lambda$ is $I_0$-dominant, Kempf's vanishing theorem from proposition \ref{prop4} combined with lemma \ref{lem2} implies moreover that $R^k\pi_* \mathcal{L}_{\lambda} = 0$ for all $k>0$. For the toroidal compactifications, the proof is exactly the same.
\end{proof}

%% file: part4.tex
\section{$G$-Zips and stratifications}\label{part4}
Let $\Sh$ denote the Siegel modular variety over $\F_p$ of genus $g \geq 1$ and neat level $K \subset \Sp_{2g}(\mathbb{A}_f)$ such that $K_p$ is hyperspecial. The Siegel modular variety $\Sh$ has the Ekedahl-Oort stratification (EO stratification) which is a genuine new structure which does not exist in characteristic $0$. For the modular curve defined over $\mathbb{F}_p$, there are two strata: the ordinary locus and the supersingular locus. The ordinary locus is an open subscheme corresponding to ordinary elliptic curves over $\mathbb{F}_p$ and the supersingular locus is a reduced closed subscheme corresponding to supersingular elliptic curves over $\mathbb{F}_p$. Hence, the closure of the ordinary locus is the whole modular curve. In the series of papers \cite{MR1710754} \cite{MR2104263} \cite{MR2804513} \cite{MR3347958}, Moonen, Wedhorn, Pink and Ziegler define an Artin stack $G\Zip^{\mu}$ which depends on the reductive group $G$ over $\mathbb{F}_p$ and a cocharacter $\mu$ of $G$. The underlying topological space of this stack is finite and its topology captures the closure relations of the EO stratification. Furthermore, one can construct a morphism
\begin{equation*}
\zeta : \Sh \rightarrow G\Zip^{\mu}
\end{equation*}
from the $G$-torsor $\mathcal{H}^1_{\dR}$ corresponding to de Rham cohomology of the universal abelian scheme. Zhang, in his thesis \cite{MR3759007}, has proven that $\zeta$ is a smooth morphism. One can recover the EO stratification on $\Sh$ through a pullback of some substack $w$ of $G\Zip^{\mu}$. Recall that $P$ is the parabolic associated to $-\mu$ defined in notation \ref{not1} and $I \subset \Delta$ is its type. The EO stratification on the Shimura variety can be further generalized on the flag bundle $Y_{I_0}$ corresponding to a standard parabolic subgroup $P_0 \subset P$ of type $I_0 \subset I$.  In \cite{MR3973104}, Goldring and Koskivirta define a stack $G\ZipFlag^{\mu,I_0}$, a smooth morphism
\begin{equation*}
\zeta_{I_0} : Y_{I_0} \rightarrow G\ZipFlag^{\mu,I_0}
\end{equation*}
and one can define a stratification of the flag bundle $Y_{I_0}$ through the pullback of a collection of some substacks $[w]$ of $G\ZipFlag^{\mu,I_0}$. For the convenience of the reader, we recall how \cite{MR3989256} and \cite{MR3973104} use the formalism of $G\Zip$s and  $G\ZipFlag$s to define and study the stratifications on the Siegel modular variety $\Sh$ and its flag bundle $Y_{I_0}$. In particular, we recall their result on the existence of generalized Hasse invariants .
\subsection{General theory}
In order to consider the stratification of the stack $G\ZipFlag^{\mu,I_0}$ for all $I_0 \subset I$, it is convenient to use a general zip datum $\mathcal{Z}$ and to define a stack $G\Zip^{\mathcal{Z}}$ for a general reductive group $G$ over $k$, a field of positive characteristic.
\begin{definition}
A zip datum of exposant $n \geq 1$ is a tuple
\begin{equation*}
\mathcal{Z} = (G,P,L,Q,M,\varphi^n),
\end{equation*}
where $G$ is a reductive group over $\mathbb{F}_p$, $\varphi : G \rightarrow G$ is the relative Frobenius map and $P,Q \subset G \times_{\mathbb{F}_p} \bar{\mathbb{F}}_p$ are parabolics over $\bar{\mathbb{F}}_p$ with Levi subgroups $L \subset P$, $M \subset Q$ such that $\varphi^n(L) = M$. We write $U$ and $V$ for the unipotent radical of $P$ and $Q$.
\end{definition}
\begin{definition}
A morphism of zip data of exposant $n \geq 1$
\begin{equation*}
\mathcal{Z} = (G,P,L,Q,M,\varphi^n) \rightarrow \mathcal{Z}^\prime = (G^\prime,P^\prime,L^\prime,Q^\prime,M^\prime,{\varphi^\prime}^n)
\end{equation*}
is the data of a group morphism $f : G \rightarrow G^\prime$ such that $f(\lozenge) \subset \lozenge^\prime$ for $\lozenge = G,P,L,Q,M,U,V$.
\end{definition}
\vspace{0.2cm}
Recall that for each $g \geq 1$, we have a minuscule cocharacter $\mu : \mathbb{G}_m \rightarrow \Sp_{2g}$ defined over $\mathbb{F}_p$. The couple $(\Sp_{2g},\mu)$ ($\Sp_{2g}$ is defined over $\mathbb{F}_p$) is a cocharacter datum according to the following definition.
\begin{definition}\label{def8}
A cocharacter datum is a couple $(G,\mu)$ where $G$ is a reductive group over $\mathbb{F}_p$ and $\mu :  \mathbb{G}_m \rightarrow G$ is a cocharacter defined over $\bar{\mathbb{F}}_p$. A morphism of cocharacter data
\begin{equation*}
(G,\mu) \rightarrow (G^\prime,\mu^\prime)
\end{equation*}
is a group morphism $f : G \rightarrow G^\prime$ such that $\mu^\prime = f\circ \mu$. A cocharacter data $(G,\mu)$ determines a opposite parabolic subgroup $P_{\mu},P_{-\mu}$ with common Levi subgroup $L =  P_{-\mu} \cap  P_{\mu}$.
\end{definition}
\vspace{0.2cm}
From a cocharacter datum $(G,\mu)$ we can construct a zip datum of exposant $n$
\begin{equation*}
\mathcal{Z}_{\mu} = (G,P,L,Q,M,\varphi^{n})
\end{equation*}
by setting $P = P_{-\mu}$, $Q=\varphi^n(P_{\mu})$, $L = P_{-\mu} \cap  P_{\mu}$, $M = \varphi^n(L)$. We explain how to define a stack $G\Zip^{\mathcal{Z}}$ from a zip datum $\mathcal{Z}$.
\begin{definition}
Let $\mathcal{Z}$ be a zip datum and $S$ be a scheme over $\mathbb{F}_p$. A zip of type $\mathcal{Z}$ over $S$ is a tuple
\begin{equation*}
\underline{I} = (\mathcal{I},\mathcal{I}_P,\mathcal{I}_Q,\psi),
\end{equation*}
where $\mathcal{I}$ is a $G$-torsor over $S$, $\mathcal{I}_P \subset \mathcal{I}$ is a $P$-reduction of $\mathcal{I}$, $\mathcal{I}_Q \subset \mathcal{I}$ is a $Q$-reduction of $\mathcal{I}$ and 
\begin{equation*}
\psi : {(\varphi^n)}^*(\mathcal{I}_P/U) \rightarrow \mathcal{I}_Q/V
\end{equation*}
is an isomorphism of $M$-torsors over $S$. A morphism of zips of type $\mathcal{Z}$ over $S$
\begin{equation*}
\underline{I} = (\mathcal{I},\mathcal{I}_P,\mathcal{I}_Q,\psi) \rightarrow \underline{I}^\prime = (\mathcal{I}^\prime,\mathcal{I}_P^\prime,\mathcal{I}_Q^\prime,\psi^\prime)
\end{equation*}
is a morphism of $G$-torsors $f : \mathcal{I} \rightarrow \mathcal{I}^\prime$ over $S$ such that $f(\lozenge) \subset {\lozenge}^\prime$ for $\lozenge = \mathcal{I}_P,\mathcal{I}_Q$ and such that the following diagram commutes
\begin{equation*}
\begin{tikzcd}
{(\varphi^n)}^*(\mathcal{I}_P/U) \arrow[d] \arrow[r,"\psi"] & \mathcal{I}_Q/V \arrow[d]\\
{({(\varphi^\prime)}^n)}^*(\mathcal{I}_P^\prime/U^\prime) \arrow[r,"\psi^\prime"] & \mathcal{I}_Q^\prime/V^\prime
\end{tikzcd}
\end{equation*}
where the vertical arrows are induced by $f$.
\end{definition}
\begin{proposition}
Let $\mathcal{Z}$ be a zip datum and $S$ be a scheme over $\mathbb{F}_p$. The category $G\Zip^{\mathcal{Z}}(S)$ of zips of type $\mathcal{Z}$ over $S$ is a groupoid. The association $S \rightarrow G\Zip^{\mathcal{Z}}(S)$ defines an algebraic stack over $\mathbb{F}_p$ that we simply denote $G\Zip^{\mathcal{Z}}$.
\end{proposition}
\begin{proof}
See \cite[Proposition 3.2/3.11]{MR3347958}.
\end{proof}
Note that the association $\mathcal{Z} \rightarrow G\Zip^{\mathcal{Z}}$ defines a functor from the category of zip data to the category of algebraic stacks over $\mathbb{F}_p$. We simply write $G\Zip^{\mu}$ instead of $G\Zip^{\mathcal{Z}_\mu}$ when the zip datum comes from a cocharacter datum $(G,\mu)$. Most of the interesting properties of $G\Zip^{\mathcal{Z}}$ can be deduced from its presentation as a quotient stack. From now on, we fix a zip datum of exposant $n$
\begin{equation*}
\mathcal{Z} = (G,P,L,Q,M,\varphi^n).
\end{equation*}
\begin{proposition}
$G\Zip^{\mathcal{Z}}$ is a smooth stack of dimension $0$ over $\mathbb{F}_p$ and it is presented as a quotient stack
\begin{equation*}
\lfloor E_{\mathcal{Z}}  \backslash G \rfloor
\end{equation*}
where $E_{\mathcal{Z}} = \{ (x,y) \in P \times Q \ | \ \varphi^n(\bar{x}) = \bar{y} \}$, $x \rightarrow \bar{x}$ denotes the natural projection $P \rightarrow L$, $Q \rightarrow M$ and $(x,y) \in E_{\mathcal{Z}}$ acts on $g \in G$ by
\begin{equation*}
(x,y)g = xgy^{-1}.
\end{equation*}
\end{proposition}
\begin{proof}
See \cite[Proposition 3.2/3.11]{MR3347958}.
\end{proof}
Denote by $W$ the Weyl group of $G$, $I \subset \Delta$ the type of the parabolic $P$, $J \subset \Delta$ the type of the parabolic $Q$, $W_I \subset W$ the subgroup generated by the reflexions in $I$, $\prescript{I}{}{W}$ the set of elements $w$ that are of minimal length in $W_Iw$, $W_J \subset W$ the subgroup generated by the reflexions in $J$ and ${W}^J$ the set of elements $w$ that are of minimal length in $wW_J$. The element of maximal length in $W$ (resp. $W_I$ and $W_J$) is denoted $w_0$ (resp. $w_{0,I}$ and $w_{0,J}$). Denote by $z$ the element $w_0w_{0,J}$. 
\begin{proposition}
If there exists a Borel pair $(B,T)$ of $G$ defined over $\mathbb{F}_p$, then there exists an element $z \in W$ such that the triple $(B,T,z)$ is a $W$-frame for $\mathcal{Z}$. It means that the following conditions are satisfied
\begin{enumerate}
\item $B \subset P$,
\item $zBz^{-1} \subset Q$,
\item $\varphi(B\cap L) =zBz^{-1}\cap M$.
\end{enumerate}
\end{proposition}
\begin{proof}
See the proof of \cite[Proposition 3.7]{MR2804513}.
\end{proof}
For each $w \in W$, we choose a lift $\dot{w}$ in $N_G(T)$. The following proposition explains how $G$ decomposes in $E_\mathcal{Z}$-orbits.
\begin{proposition}
The map $w \mapsto G_w := E_\mathcal{Z}\dot{w}\dot{z}^{-1}$ restricts to bijections\footnote{In the case $G = \Sp_{2g}$, these two bijections coincide.} between
\begin{enumerate}
\item $\prescript{I}{}{W}$ and the $E_\mathcal{Z}$-orbits of $G$,
\item ${W}^J$ and the $E_\mathcal{Z}$-orbits of $G$. 
\end{enumerate}
Moreover, we have the following dimension formula for all $w \in \prescript{I}{}{W} \cup {W}^J$
\begin{equation*}
\dim G_w = l(w) + \dim(P).
\end{equation*}
\end{proposition}
\begin{proof}
See \cite[Theorem 7.5/11.2]{MR2804513}.
\end{proof}
\begin{corollary}
The stack $G\Zip^{\mathcal{Z}}$ decomposes 
\begin{equation*}
G\Zip^{\mathcal{Z}} = \underset{w \in \prescript{I}{}{W} }{\bigsqcup} \lfloor E_{\mathcal{Z}}  \backslash G_w \rfloor.
\end{equation*}
\end{corollary}
The stack $G\Zip^{\mathcal{Z}}$ has a topology which describes the closure relations between the $G_w$. 
\begin{proposition}\label{prop21}
The underlying topological space of $G\Zip^{\mathcal{Z}}$ is homeomorphic to the finite topological space $\prescript{I}{}{W}$ where the topology is given by the partial order:
\begin{equation*}
w \preccurlyeq w^\prime  \Leftrightarrow \text{ if and only if there is } v \in W_I \text{ such that } vwxv^{-1}x^{-1} \leq w^\prime,
\end{equation*}
where $x$ is the unique element of minimal length in $W_Jw_0W_I$ and $\leq$ is the Bruhat order.
\end{proposition}
\begin{proof}
The result follows from the isomorphism 
\begin{equation*}
\overline{G_w} =  \underset{w^\prime \preccurlyeq w \text{, } w^\prime \in \prescript{I}{}{W} }{\bigsqcup} G_{w^\prime},
\end{equation*}
for all $w \in \prescript{I}{}{W}$ which is proven in \cite[Theorem 6.2]{MR2804513}.
\end{proof}
We simply write $[w]$ for the locally closed substack $\lfloor E_{\mathcal{Z}}  \backslash G_w \rfloor$ of $G\Zip^{\mathcal{Z}}$. Now, we describe how to define a more general stack $G\ZipFlag^{\mathcal{Z},P_0}$ which depends on the zip datum $\mathcal{Z}$ and an auxiliary parabolic subgroup $B \subset P_0 \subset P$.
\begin{definition}
Let $B \subset P_0 \subset P$ be a parabolic subgroup of $P$ and $S$ be a scheme over $\mathbb{F}_p$. A zip flag of type $(\mathcal{Z},P_0)$ over $S$ is a tuple
\begin{equation*}
\underline{J} = (\underline{I},\mathcal{J}),
\end{equation*}
where $\underline{I} = (\mathcal{I},\mathcal{I}_P,\mathcal{I}_Q,\psi)$ is a zip of type $\mathcal{Z}$ over $S$ and $\mathcal{J} \subset \mathcal{I}_P$ is a $P_0$-reduction of the $P$-torsor $\mathcal{I}_P$. A morphism of zip flags of type $(\mathcal{Z},P_0)$ over $S$
\begin{equation*}
\underline{J} = (\underline{I},\mathcal{J}) \rightarrow \underline{J}^\prime = (\underline{I}^\prime,\mathcal{J}^\prime)
\end{equation*}
is a morphism of zip $\underline{I} \rightarrow \underline{I}^\prime$ of type $\mathcal{Z}$ over $S$ such that the underlying morphism of $G$-torsor $\mathcal{I} \rightarrow \mathcal{I}^\prime$ restricts to a morphism of $P_0$-torsor $\mathcal{J} \rightarrow \mathcal{J}^\prime$ over $S$. 
\end{definition} 
\begin{proposition}\label{prop13bis}
Let $B \subset P_0 \subset P$ be a parabolic subgroup of $P$ and $S$ be a scheme over $\mathbb{F}_p$. The category $G\ZipFlag^{\mathcal{Z},P_0}(S)$ of zip flags of type $(\mathcal{Z},P_0)$ over $S$ is a groupoid. The association $S \rightarrow G\ZipFlag^{\mathcal{Z},P_0}(S)$ defines an algebraic stack over $\mathbb{F}_p$ that we simply denote $G\ZipFlag^{\mathcal{Z},P_0}$.
\end{proposition}
From now on, we fix an auxiliary parabolic subgroup $B \subset P_0 \subset P$.
\begin{proposition}\label{prop14}
The stack $G\ZipFlag^{\mathcal{Z},P_0}$ is a smooth stack of dimension $\dim (P/P_0)$ over $\mathbb{F}_p$ and it can be presented as the quotient stack
\begin{equation*}
\lfloor E_{\mathcal{Z},P_0}  \backslash G \rfloor
\end{equation*}
where $E_{\mathcal{Z},P_0} := E_{\mathcal{Z}} \cap (P_0 \times G) \subset P_0 \times Q$ acts on $G$ by restriction of the $E_{\mathcal{Z}}$-action on $G$. It can also be presented as the quotient stack
\begin{equation*}
\lfloor E_{\mathcal{Z}} \times P_0  \backslash G \times P \rfloor
\end{equation*}
where $((x,y),p_0) \in E_{\mathcal{Z}} \times P_0$ acts on $(g,p) \in G \times P$ through the formula 
\begin{equation*}
((x,y),p_0).(g,p) = (xgy^{-1},xpp_0^{-1}).
\end{equation*}
\end{proposition}
\begin{definition}
The map sending a zip flag $\underline{J} = (\underline{I},\mathcal{J})$ of type $(\mathcal{Z},P_0)$ over $S$ to the zip $\underline{I}$ of type $\mathcal{Z}$ over $S$ defines a morphism of algebraic stacks over $\mathbb{F}_p$
\begin{equation*}
\pi : G\ZipFlag^{\mathcal{Z},P_0} \rightarrow G\Zip^{\mathcal{Z}}.
\end{equation*}
\end{definition}
\begin{proposition}\label{prop14bis}
The inclusion $E_{\mathcal{Z},P_0} \subset E_{\mathcal{Z}}$ induces a morphism
\begin{equation*}
\lfloor E_{\mathcal{Z},P_0}  \backslash G \rfloor \rightarrow \lfloor E_{\mathcal{Z}}  \backslash G \rfloor
\end{equation*}
which corresponds to $\pi$ through the isomorphisms in proposition \ref{prop14}.
\end{proposition}
\begin{proposition}\label{prop15}
The morphism $\pi : G\ZipFlag^{\mathcal{Z},P_0} \rightarrow G\Zip^{\mathcal{Z}}$ is proper and smooth with fibers isomorphic to the flag variety $P/P_0$.
\end{proposition}
\begin{proof}[Proof for propositions \ref{prop13bis}, \ref{prop14}, \ref{prop14bis} and \ref{prop15}]
See \cite[Theorem 2.1.2]{MR3973104}.
\end{proof}
It is natural to hope for a stratification of $G\ZipFlag^{\mathcal{Z},P_0}$ that generalizes the stratification on $G\Zip^{\mathcal{Z}}$ however the $E_{\mathcal{Z},P_0}$-orbits of $G$  are not as easy to understand as the $E_{\mathcal{Z}}$-orbits. Instead, we define a smooth surjective map
\begin{equation*}
G\ZipFlag^{\mathcal{Z},P_0} \rightarrow G\Zip^{\mathcal{Z}_0}
\end{equation*}
where $\mathcal{Z}_0$ is a zip datum constructed from $\mathcal{Z}$ and $P_0$ and then pullback the stratification of $G\Zip^{\mathcal{Z}_0}$. 
\begin{definition}
We denote by $\mathcal{Z}_0$ the zip datum
\begin{equation*}
\mathcal{Z}_0 = (G,P_0,L_0,Q_0,M_0,\varphi^n)
\end{equation*}
where $Q_0$ is a parabolic subgroup of $Q$ defined by 
\begin{equation*}
Q_0 = \varphi^n(P_0\cap L)R_u(Q) \subset Q
\end{equation*}
with $R_u(Q)$ the unipotent radical of $Q$ and where $L_0$, $M_0$ are the Levi subgroups of $P_0$, $Q_0$.
\end{definition}
\begin{proposition}
We have inclusions
\begin{equation*}
E_{\mathcal{Z},P_0} \subset E_{\mathcal{Z}_0} \subset P_0 \times Q_0
\end{equation*}
and the induced maps
\begin{equation*}
\begin{tikzcd}
G\ZipFlag^{\mathcal{Z},P_0}  \arrow[r,"\psi_1"] &  G\Zip^{\mathcal{Z}_0}  \arrow[r,"\psi_2"] & \lfloor P_0 \backslash G / Q_0 \rfloor
\end{tikzcd}
\end{equation*}
are smooth and surjective.
\end{proposition}
\begin{proof}
See \cite[Section 3.1]{MR3973104}.
\end{proof}
\begin{definition}
The fine stratification of $G\ZipFlag^{\mathcal{Z},P_0}$ is the stratification of $G\Zip^{\mathcal{Z}_0}$ pulled back by $\psi_1$ and the coarse stratification of $G\ZipFlag^{\mathcal{Z},P_0}$ is the stratification of the Bruhat stack $\lfloor P_0 \backslash G / Q_0 \rfloor$ pulled back by $\psi_2 \circ \psi_1$. If $w \in \prescript{I_0}{}{W} \cup W^{J_0}$, then we write $G\ZipFlag^{\mathcal{Z},P_0}_w$ for the corresponding fine strata.
\end{definition}
\vspace{0.2cm}
In the special case where $P_0 = B$ is the Borel subgroup, the map $\psi_2$ induces a bijection on strata, so the coarse and the fine stratifications of $G\ZipFlag^{\mathcal{Z},P_0}$ coincide. Note that if we have an inclusion of auxiliary parabolics $B \subset P_0 \subset P_1 \subset P$, then there exists natural maps making the following diagram $2$-cartesian.
\begin{equation*}
\begin{tikzcd}
G\ZipFlag^{\mathcal{Z},P_0} \arrow[d]  \arrow[r] & \lfloor P_0 \backslash G / Q_0 \rfloor \arrow[d] \\
G\ZipFlag^{\mathcal{Z},P_1}  \arrow[r] & \lfloor P_1 \backslash G / Q_1 \rfloor
\end{tikzcd}
\end{equation*}
However, we don't know if a similar statement holds if we replace the Bruhat stacks with $G\Zip^{\mathcal{Z}_0}$ and $G\Zip^{\mathcal{Z}_1}$. 
\begin{corollary}\label{cor4.20}
The stack $G\ZipFlag^{\mathcal{Z},P_0}$ decomposes as
\begin{equation*}
G\ZipFlag^{\mathcal{Z},P_0} = \underset{w \in \prescript{I_0}{}{W} }{\bigsqcup} G\ZipFlag^{\mathcal{Z},P_0}_w
\end{equation*}
and for all  $w \in \prescript{I_0}{}{W}$, we have the closure relation
\begin{equation*}
\overline{G\ZipFlag^{\mathcal{Z},P_0}_w} =  \underset{w^\prime \preccurlyeq w \text{, } w^\prime \in \prescript{I_0}{}{W} }{\bigsqcup} G\ZipFlag^{\mathcal{Z},P_0}_{w^\prime},
\end{equation*}
where the order on $\prescript{I_0}{}{W}$ is the one introduced in proposition \ref{prop21} with $I$ replaced by $I_0$.
\end{corollary}
\begin{corollary}
Let $w \in \prescript{I_0}{}{W} \cup W^{J_0}$ and $G\ZipFlag^{\mathcal{Z},P_0}_w$ be the corresponding fine strata. Then $G\ZipFlag^{\mathcal{Z},P_0}_w$ is a smooth stack over $\mathbb{F}_p$ of pure dimension $l(w)+\dim P -\dim G$.
\end{corollary}
\vspace{0.2cm}
Now we want to construct some sections of vector bundles on $G\Zip^{\mathcal{Z}}$, $G\ZipFlag^{\mathcal{Z},P_0}$ and relate their non-vanishing loci to the stratification we have introduced. We start by introducing vector bundles on $G\Zip^{\mathcal{Z}}$.
\begin{definition}
Let $ \rho : L \rightarrow \GL(V)$ be a finite dimensional algebraic representation of the Levi $L$. Consider the map $f: E_{\mathcal{Z}}  \rightarrow L$ which is the composition of the first projection $E_{\mathcal{Z}} \rightarrow P$ with the quotient map $P \rightarrow L$. The composition $\rho \circ f$ is an algebraic representation of $E_{\mathcal{Z}}$. It induces a locally free sheaf $\mathcal{W}(V)$ of rank $\dim_{\mathbb{F}_p} V$ on $\lfloor E_{\mathcal{Z}} \backslash G \rfloor$. If $\lambda \in X^{*}(T)$ is a $I$-dominant character of $T$, we simply denote $\nabla(\lambda)$ the locally free sheaf $\mathcal{W}(H^0(L/B_L, \mathcal{L}_{\lambda}))$. Note that $H^0(L/B_L, \mathcal{L}_{\lambda})$ is the costandard $L$-representation of highest weight $w_0w_{0,L}\lambda$.
\end{definition}
\vspace{0.2cm}
More generally, we can define vector bundles on $G\ZipFlag^{\mathcal{Z},P_0}$.
\begin{definition}
Let $ \rho : P_0 \rightarrow \GL(V)$ be a finite dimensional algebraic representation of the parabolic $P_0$. Consider the first projection map $f: E_{\mathcal{Z},P_0}  \rightarrow P_0$. The composition $\rho \circ f$ is an algebraic representation of $E_{\mathcal{Z},P_0}$ and it induces a locally free sheaf $\mathcal{L}(V)$ of rank $\dim_{\overline{\mathbb{F}}_p} V$ on $\lfloor E_{\mathcal{Z},P_0} \backslash G \rfloor$. If $\lambda \in X^{*}(L_0) \subset X^{*}(T) $ is a character of $L_0$, we also denote by $\mathcal{L}_{\lambda}$ the line bundle $\mathcal{L}(\lambda)$ where we see $\lambda$ as a one-dimensional representation of $P_0$.
\end{definition}
\vspace{0.2cm}
We have defined vector bundles $\nabla(\lambda)$ on $G\Zip^{\mathcal{Z}}$ and line bundles $\mathcal{L}_{\lambda}$ on $G\ZipFlag^{\mathcal{Z},P_0}$ for certain characters $\lambda \in X^*(T)$. The next proposition gives a direct relation between them.
\begin{proposition}
Recall that $\pi : G\ZipFlag^{\mathcal{Z},P_0} \rightarrow G\Zip^{\mathcal{Z}}$ is the proper and smooth map that forgets the $P_0$-torsor from a zip flag of type $(\mathcal{Z},P_0)$. Let $\lambda \in X^*(L_0)$ be an $I_0$-dominant character of $L_0$. We have a canonical isomorphism
\begin{equation*}
\pi_* \mathcal{L}_{\lambda} \simeq \nabla(\lambda).
\end{equation*}
\end{proposition}
\begin{proof}
Consider the cartesian diagram
\begin{equation*}
\begin{tikzcd}
G\ZipFlag^{\mathcal{Z},P_0} \arrow[d,"\pi"] \arrow[r,"\tilde{\zeta}"] & \lfloor * / P_0 \rfloor \arrow[d,"\tilde{\pi}"] \\
G\Zip^{\mathcal{Z}} \arrow[r,"\zeta"] & \lfloor * / P \rfloor
\end{tikzcd}
\end{equation*}
where the horizontal maps are given by the universal $P_0$-torsor on $G\ZipFlag^{\mathcal{Z},P_0}$ and the universal $P$-torsor on $G\Zip^{\mathcal{Z}}$. For each character $\lambda \in X^*(P_0)$, we have a line bundle $\mathcal{L}_{\lambda}$ on $\lfloor * / P_0 \rfloor$ and a vector bundle $\nabla(\lambda)$ on $\lfloor * / P \rfloor$ (corresponding to the induced $P$-representation $H^0(P/P_0,\mathcal{L}_{\lambda})$) that satisfies
\begin{equation*}
\left\{
\begin{aligned}
&\tilde{\zeta}^*\mathcal{L}_{\lambda} = \mathcal{L}_{\lambda}, \\
&\zeta^*\nabla({\lambda}) = \nabla({\lambda}).
\end{aligned}
\right.
\end{equation*}
It is straightforward from the definitions that 
\begin{equation*}
\tilde{\pi}_*\mathcal{L}_{\lambda} = \nabla(\lambda).
\end{equation*}
As the map is fibered in $P/P_0$ by proposition \ref{prop15}, we know by proposition \ref{prop4} and lemma \ref{lem2} that for a $I_0$-dominant character $\lambda$, we have 
\begin{equation*}
\left\{
\begin{aligned}
&R\tilde{\pi}_* \mathcal{L}_{\lambda} = \tilde{\pi}_* \mathcal{L}_{\lambda}, \\
&R{\pi}_* \mathcal{L}_{\lambda} = {\pi}_* \mathcal{L}_{\lambda}.
\end{aligned}
\right.
\end{equation*}
Since $\zeta$ is flat, we conclude as in the end of the proof of proposition \ref{prop13} with the base change theorem in the derived category that says that the natural map
\begin{equation*}
\zeta^* \circ R\tilde{\pi}_* \mathcal{L}_{\lambda}  \rightarrow {R}\pi_*\circ \zeta^* \mathcal{L}_{\lambda} 
\end{equation*}
is an isomorphism.
\end{proof}
On the Bruhat stack $\Brh = \lfloor B \backslash G / B \rfloor$, we have the Bruhat stratification
\begin{equation*}
\Brh = \underset{w  \in \prescript{}{}{W} }{\bigsqcup} \Brh_w,
\end{equation*} 
where $\Brh_w = \lfloor B \backslash BwB / B \rfloor$ and for all $w \in W$ we have the closure relation
\begin{equation*}
\overline{\Brh_w} = \underset{w^\prime \leq w \text{, } w^\prime \in \prescript{}{}{W} }{\bigsqcup} \Brh_{w^\prime},
\end{equation*}
where $\leq$ is the Bruhat order. We consider the morphism 
\begin{equation*}
\psi : G\ZipFlag^{\mathcal{Z},B} \rightarrow \Brh
\end{equation*}
defined as the composition of the morphism induced by the inclusion
\begin{equation*}
E_{\mathcal{Z},B} \subset B \times zBz^{-1}
\end{equation*}
with the isomorphism
\begin{equation*}
\alpha_z : \lfloor B \backslash G / zBz^{-1} \rfloor \rightarrow \lfloor B \backslash G / B \rfloor,
\end{equation*}
that sends $x$ to $xz$. We use this stack to construct some sections on $G\ZipFlag^{\mathcal{Z},P_0}$.
\begin{proposition}
Given two characters $(\lambda,\eta) \in X^*(T) \times X^*(T)$, the associated line bundle $\mathcal{L}_{\lambda,\eta}$ on $\Brh$ has the following properties.
\begin{enumerate}
\item We have a canonical isomorphism $\psi^*\mathcal{L}_{\lambda,\eta} = \mathcal{L}_{\lambda+p\prescript{\sigma}{}{(z\eta)}}$ where $\sigma : \overline{\mathbb{F}}_p \rightarrow  \overline{\mathbb{F}}_p$ is the inverse of the Frobenius.
\item For all $w \in W$, we have $H^0(\Brh_w, \mathcal{L}_{\lambda,\eta}) \neq 0 \Leftrightarrow \eta = -w^{-1}\lambda$.
\item $\dim_{\mathbb{F}_p} H^0(\Brh_w, \mathcal{L}_{\lambda,-w^{-1}\lambda}) = 1$.
\item For any non-zero $s \in H^0(\Brh_w, \mathcal{L}_{\lambda,-w^{-1}\lambda})$ viewed as a rational function on $\overline{\Brh}_w$, one has
\begin{equation*}
\divi(s) = -\sum_{\alpha \in E_w} \langle\lambda,w\alpha^{\vee}\rangle\overline{\Brh}_{ws_{\alpha}}
\end{equation*}
where $E_w = \{ \alpha \in \phi^+ \ | \ ws_{\alpha} < w \text{ and } l(ws_{\alpha}) = l(w) -1 \}$. The set of $ws_{\alpha}$ for $\alpha \in E_w$ is called the set of lower neighbors of $w$.
\end{enumerate}
\end{proposition}
\begin{proof}
For $(i)$, see \cite[Lemma 3.1.1]{MR3989256}. For $(ii)$ to $(iv)$, see \cite[Theorem 2.2.1]{MR3989256}.
\end{proof}
\begin{definition}
Let $w \in W$ and $n \geq 0$. We define by induction on $n$, the element $w^{(n)}$ by setting
\begin{enumerate}
\item $w^{(0)} = e$,
\item $w^{(n)} = \prescript{\sigma}{}{(w^{(n-1)}w)}$ if $n\geq 1$.
\end{enumerate}
\end{definition}
\begin{proposition}
The function
\begin{equation*}
\begin{array}{cccc}
D_w : & X^*(T) & \to & X^*(T)\\
& \lambda & \mapsto & \lambda-p\prescript{\sigma}{}{(zw^{-1}\lambda)}
\end{array}
\end{equation*}
induces a $\mathbb{Q}$-linear automorphism of $X^*(T) \otimes_{\mathbb{Z}} \mathbb{Q}$. If $\chi$ is a character, its inverse by $D_w$ is given by the $\mathbb{Q}$-character 
\begin{equation*}
\lambda = \frac{-1}{p^{rn}-1}\sum_{i=0}^{rn-1} p^i(zw^{-1})^{(i)} \prescript{\sigma^i}{}{\chi},
\end{equation*}
where $r$ is an integer such that $(zw^{-1})^{(r)} = e$ and $n$ is an integer such that $\chi$ is defined over $\mathbb{F}_{p^n}$. 
\end{proposition}
\begin{proof}
See \cite[Lemma 3.1.3]{MR3989256}.
\end{proof}
Having defined sections on stacks $G\Zip^{\mathcal{Z}}$, we can study their vanishing locus.
\begin{definition}\label{def4.28}
Let $\lambda \in X^*(P_0)$ be a $L_0$-dominant character of $P_0$ and
\begin{equation*}
s \in H^0(G\ZipFlag^{\mathcal{Z},P_0}_w,\mathcal{L}_{\lambda}),
\end{equation*}
a non-zero section. We say that $s$ is a generalized Hasse invariant for $G\ZipFlag^{\mathcal{Z},P_0}_w$ if there exists some $d \geq 1$ such that $s^d$ extends to $\overline{G\ZipFlag^{\mathcal{Z},P_0}_w}$ with non-vanishing locus $G\ZipFlag^{\mathcal{Z},P_0}_w$. We define the sets
\begin{multline*}
\mathcal{C}_{\Ha,I_0,w} = \{ \lambda \in X^*(P_0) \ | \ \mathcal{L}_{\lambda} \text{ has a generalized } \text{ Hasse invariant for } G\ZipFlag^{\mathcal{Z},P_0}_w \}
\end{multline*}
and
\begin{equation*}
\mathcal{C}_{\Ha,I_0} = \bigcap_{w \in \prescript{I_0}{}{W}} \mathcal{C}_{\Ha,I_0,w}.
\end{equation*}
\end{definition}
Now, we give a strong result for the existence of generalized Hasse invariants on the stack $G\ZipFlag^{\mathcal{Z},P_0}$.
\begin{proposition}\label{prop22}
Let $\lambda \in X^*(P_0)$ be a $L_0$-dominant character, $w$ be an element of $\prescript{I_0}{}{W}$ and $s$ be a non zero section of $H^0(G\ZipFlag_w^{\mathcal{Z},P_0},\mathcal{L}_{\lambda})$. Then, the following statements are equivalent:
\begin{enumerate}
\item $s$ is a generalized Hasse invariant for $G\ZipFlag^{\mathcal{Z},P_0}_w$.
\item For all $\alpha \in E_w$, we have
\begin{equation*}
\sum_{i=0}^{rn-1} \langle (zw^{-1})^{(i)} (\prescript{\sigma^i}{}{\lambda})	,w\alpha^{\vee}\rangle p^i > 0,
\end{equation*}
where $r$ is an integer such that $(zw^{-1})^{(r)} = e$ and $n$ is an integer such that $\lambda$ is defined over $\mathbb{F}_{p^n}$. 
\end{enumerate}
\end{proposition}
\begin{proof}
See \cite[Proposition 3.2.1]{MR3989256}.
\end{proof}
\begin{example}
We give more details in the case $G = \Sp_4$ and $\mathcal{Z} = \mathcal{Z}_{\mu}$ with $\mu$ the cocharacter that stabilizes the Hodge filtration of the Siegel datum. The Levi $L$ is $\GL_2$. We denote by $s_1$ and $s_2$ the simple reflections associated to the simple roots $(1,-1)$ and $(0,2)$. We represent the elements of $W$ in the diagram
\begin{equation*}
\begin{tikzcd}
& w_0 = s_2s_1s_2s_1 \arrow[ld] \arrow[rd] & \\
w_1 = s_2s_1s_2 \arrow[d] \arrow[rrd] & & w_1^{\prime} = s_1s_2s_1 \arrow[d] \arrow[lld]\\
w_2 =  s_2s_1 \arrow[d] \arrow[rrd] & & w_2^{\prime} = s_1s_2 \arrow[d] \arrow[lld]\\
w_3 = s_2 \arrow[rd] & & w_3^{\prime} = s_1 \arrow[ld] \\
& e &
\end{tikzcd}
\end{equation*}
where an arrow is drawn from $w$ to $w^{\prime}$ if $w^{\prime} \leq w$ and $l(w^\prime) = l(w)-1$. For each $w \in W$ and $\lambda \in X^*$, we denote by $s_{\lambda,w}$ the quasi-section\footnote{It means a section of a certain positive tensorial power of $\Lc_{\lambda}$.} in $H^0(G\ZipFlag_w^{\mathcal{Z},B},\mathcal{L}_{\lambda})$ obtained via pullback from a non-zero quasi-section of $H^0(\Brh_w, \mathcal{L}_{\chi,-w^{-1}\chi})$ where $\chi$ is a $\Q$-character such that $D_w(\chi) = \lambda$. We write $\lambda =(k_1,k_2)$ and we compute $\divi(s_{\lambda,w})$ for each $w$.
\begin{equation*}
\left\{
\begin{aligned}
\divi(s_{\lambda,w_0}) &= \frac{1}{p^2-1}\left((p-1)(k_1-k_2)\overline{[w_1]} - (k_2+pk_1) \overline{[w_1^\prime]} \right), \\
\divi(s_{\lambda,w_1}) &= \frac{1}{p-1}\left(-k_1\overline{[w_2]} -k_2\overline{[w_2^\prime]} \right),\\
\divi(s_{\lambda,w_1^{\prime}}) &= \frac{1}{p^2+1}\left( - ((p-1)k_1+(p+1)k_2)\overline{[w_2]} + ((p+1)k_1-(p-1)k_2) \overline{[w_2^\prime]} \right), \\
\divi(s_{\lambda,w_2}) &= \left( \frac{k_1}{p+1}-\frac{k_2}{p-1}\right) \overline{[w_3]} +\frac{-k_2}{p-1}\overline{[w_3^\prime]}, \\
\divi(s_{\lambda,w_2^\prime}) &=  -\left(\frac{k_1}{p-1}+\frac{k_2}{p+1} \right)\overline{[w_3]} + \frac{-k_1}{p-1}\overline{[w_3^\prime]}, \\
\divi(s_{\lambda,w_3}) &= \frac{1}{p^2+1} (k_2-pk_1)\overline{[e]}, \\
\divi(s_{\lambda,w_3^\prime}) &= \frac{1}{p+1}(k_1-k_2) \overline{[e]}.
\end{aligned}
\right.
\end{equation*}
We deduce that the following set of characters
\begin{equation*}
\mathcal{C}_1 := \{ \lambda = (k_1,k_2) \ | \ 0 > k_1 > \frac{p-1}{p+1}k_2 \text{ and } k_2 > pk_1 \} 
\end{equation*}
has generalized Hasse invariants for all strata of $G\ZipFlag_w^{\mathcal{Z},B}$. Since the character of the Levi are $X^*(L) = \mathbb{Z}(1,1)$ and the minimal length left coset representatives are
\begin{equation*}
\prescript{I}{}{W} = \left\{ e, w_3, w_2, w_1\right\},
\end{equation*}
we deduce that the following set of characters
\begin{equation*}
\mathcal{C}_2 := \{ \lambda = (k_1,k_1) \ | \  k_1 < 0 \} 
\end{equation*}
has generalized Hasse invariant for all strata of $G\Zip^{\mathcal{Z}}$.
\end{example}
\begin{example}
We give some details in the case $G = \Sp_6$. The Levi $L$ is $\GL_3$ and we denote by $s_1, s_2$  and $s_3$ the simple reflections associated to the simple roots $(1,-1,0),(0,1,-1)$ and $(0,0,2)$. The Weyl group $W$ is isomorphic to $S_3 \ltimes (\Z/2\Z)^3$ (48 elements). The computations can be painful without a computer, so we have implemented an algorithm in SageMath that computes the divisor of all the $s_{w,\lambda}$ for any $g \geq 2$. See \href{https://github.com/ThibaultAlexandre/generalized-hasse-invariants}{github.com/ThibaultAlexandre/generalized-hasse-invariants} to download the algorithm. Take $p = 7$, $w = s_1s_2s_3$, and $\lambda = (-1,-3,-5)$. We get
\begin{equation*}
\divi(s_{w,\lambda}) = \frac{5}{6}\overline{[s_2s_3]} + \frac{1}{2}\overline{[s_3s_1]}  + \frac{1}{6}\overline{[s_1s_2]}.
\end{equation*}
\end{example}
Koskivirta and Goldring introduced a notion called orbitally $p$-closeness that guarantees a character to have generalized Hasse invariants for all strata without having to compute all the $\divi(s_{w,\lambda})$. However, this notion is not necessary for a character to have Hasse invariants.
\begin{definition}
Let $\lambda$ be a character of $T$. For every coroot such that $\langle \lambda,\alpha^{\vee}\rangle \neq 0$, we set
\begin{equation*}
\Orb(\lambda,\alpha^{\vee}) = \{ \frac{ | \langle \lambda,w\alpha^{\vee}\rangle |}{|\langle \lambda,\alpha^{\vee} \rangle | } \ | \ w \in W \rtimes \text{Gal}(\overline{\mathbb{F}}_p/\mathbb{F}_p) \}
\end{equation*}
and we say that $\lambda$ is 
\begin{enumerate}
\item orbitally $p$-close if $\max_{\alpha \in \phi} \Orb(\lambda,\alpha^{\vee}) \leq p-1$,
\item $\mathcal{Z}_0$-ample if $\langle \lambda, \alpha^{\vee}\rangle > 0$ for all $\alpha \in I \backslash I_0$ and $\langle \lambda, \alpha^{\vee}\rangle <0$ for all $\alpha \in \phi^+ \backslash \phi^+_L$.
\end{enumerate}
\end{definition}
\begin{proposition}\label{prop16}
Let $\lambda$ be a character of $P_0$. If $\lambda$ is orbitally $p$-close and $\mathcal{Z}_0$-ample then, there exists $d \geq 1$ such that for all $w \in \prescript{I_0}{}{W}$ and all non-zero section $s$ in $$H^0(G\ZipFlag^{\mathcal{Z},P_0}_w,\mathcal{L}_{\lambda}),$$ the $d^{\text{th}}$-power $s^d$ extends to $\overline{G\ZipFlag^{\mathcal{Z},P_0}_w}$ with non-vanishing locus $G\ZipFlag^{\mathcal{Z},P_0}_w$.
\end{proposition}
\begin{proof}
See \cite[Proposition 3.2.3]{MR3989256}.
\end{proof}

\subsection{$G\Zip$ associated to the universal abelian scheme}
In this subsection, we specialize our discussion to the Siegel case. Recall that the Siegel modular variety $\Sh$ is a smooth scheme over $k = \mathbb{F}_p$. We denote by $\pi : Y_{I_0} \rightarrow \Sh$ the Siegel flag bundle of type $I_0 \subset I$. Recall that $\pi$ extends to the toroidal compactifications $\pi : Y_{I_0}^{\tor}  \rightarrow \Sh^{\tor}$. We also have a minimal compactification $\Sh^{\min}$ for the Shimura variety but not for the flag bundle. The goal of this subsection is to define the maps $\zeta$ and $\zeta_{I_0}$. We need some results on the Hodge and  conjugate filtrations of abelian schemes. We recall a result due to Deligne and Illusie.
\begin{proposition}
Let $S$ be a scheme of characteristic $p$. Let $f : A \rightarrow S$ be an abelian scheme over $S$. Consider the Hodge to de Rham spectral sequence
\begin{equation*}
E_1^{i,j} = R^jf_*\Omega^i_{A/S} \Rightarrow  H^{i+j}_{\dR}(A/S).
\end{equation*}
and the conjugate spectral sequence
\begin{equation*}
{E^{\prime}_2}^{i,j} = R^if_*(\mathcal{H}^j(\Omega^{\bullet}_{A/S})) \Rightarrow  H^{i+j}_{\dR}(A/S).
\end{equation*}
Then
\begin{enumerate}
\item $E_1^{i,j}$ degenerates at page $1$,
\item ${E^{\prime}_2}^{i,j}$ degenerates at page $2$.
\end{enumerate}
\end{proposition}
\begin{proof}
See \cite[Corollaire 2.4 and Remarques 2.6 (iv)]{MR894379}.
\end{proof}
\begin{definition}
Let $S$ be a scheme of characteristic $p$. Let $f : A \rightarrow S$ be an abelian scheme over $S$. The Hodge filtration of $A$ over $S$ is the two-step underlying filtration on $H^1_{\dR}(A/S)$ coming from the degeneration of the Hodge to de Rham spectral sequence
\begin{equation*}
0 \rightarrow \pi_*{\Omega^1_{A/S}} \rightarrow H^1_{\dR}(A/S) \rightarrow R^1\pi_*\mathcal{O}_A \rightarrow 0.
\end{equation*}
\end{definition}
\begin{definition}
The conjugate filtration is the two-step underlying filtration on $H^1_{\dR}(A/S)$ coming from the degeneration of the conjugate spectral sequence
\begin{equation*}
0 \rightarrow R^1\pi_*{\mathcal{H}^0(\Omega^{\bullet}_{A/S})} \rightarrow H^1_{\dR}(A/S) \rightarrow \pi_*{\mathcal{H}^1(\Omega^{\bullet}_{A/S})} \rightarrow 0.
\end{equation*}
\end{definition}
The Hodge and the conjugate filtration are related on their graded pieces by the Cartier isomorphism which of we recall the definition.
\begin{definition}
Let $S$ be a scheme of characteristic $p$ and $f : A \rightarrow S$ be a smooth morphism\footnote{Such as an abelian scheme $A$ over $S$.}. The Cartier morphism is a map of graded algebras
\begin{equation*}
C^{-1} : \bigoplus_{i} \Omega^{i}_{A^{(p)}/S} \rightarrow \bigoplus_{i}\mathcal{H}^{i}(F_*\Omega^{\bullet}_{A/S}),
\end{equation*}
where $F : A \rightarrow A^{(p)}$ is the relative geometric Frobenius of $A$ over $S$. It is enough to define it in degree $0$ and $1$ and then use the graded algebra structure to extend it. In degree $0$, it is the map $F^{*} : \mathcal{O}_{A^{(p)}} \rightarrow F_*\mathcal{O}_{A}$. In degree $1$, it is a map 
\begin{equation*}
\Omega^{1}_{A^{(p)}/S} \rightarrow \mathcal{H}^{1}(F_*\Omega^{\bullet}_{A/S})
\end{equation*}
coming from the $S$-derivation $\delta : \mathcal{O}_{A^{(p)}} \rightarrow \mathcal{H}^{1}(F_*\Omega^{\bullet}_{A/S})$ satisfying (we use the isomorphism $\mathcal{O}_{A^{(p)}} = \mathcal{O}_A \otimes_{\mathcal{O}_S,F^*}\mathcal{O}_S$)
\begin{enumerate}
\item $\delta(fs\otimes s^\prime) = \delta(f\otimes s^p s^\prime)$,
\item $\delta(fg\otimes s) = f^p\delta(g \otimes s) + g^p\delta(f\otimes s)$,
\end{enumerate}
for all $f,g \in \mathcal{O}_{A}$ and $s,s^\prime \in \mathcal{O}_S$. If $f \in \mathcal{O}_A$ and $s \in \mathcal{O}_S$, we define $\delta(f\otimes s)$ to be the cohomology class of $sf^{p-1}df$.
\end{definition}
\begin{proposition}[\cite{MR84497}]
The Cartier morphism $C^{-1}$ is an isomorphism and it satisfies
\begin{enumerate}
\item $C^{-1}(1) = 1$,
\item $C^{-1}(w \wedge w^\prime) = C^{-1}(w) \wedge C^{-1}(w^\prime)$ for all $w \in \Omega^{i}_{A^{(p)}/S}$,  $w^\prime \in \Omega^{i^\prime}_{A^{(p)}/S}$,
\item $C^{-1}(d(f\otimes 1)) = \left[ f^{p-1}df \right]$.
\end{enumerate}
\end{proposition}
\vspace{0.2cm}
The Hodge filtration and the conjugate filtration of an abelian scheme $f : A \rightarrow S$ of relative dimension $g$ can be seen as a $P$-reduction $\mathcal{I}$ and a $Q$-reduction $\mathcal{I}$ of the $G = \Sp_{2g}$-torsor $H^1_{\dR}(A/S)$ where $P$ and $Q$ are the maximal parabolic subgroups associated\footnote{See the paragraph after definition \ref{def8}.} to the cocharacter datum $(\Sp_{2g}, \mu)$. With the Cartier isomorphism, we can construct an isomorphism 
\begin{equation*}
\psi : {(\varphi)}^*(\mathcal{I}_{P}/U) \rightarrow \mathcal{I}_{Q}/V
\end{equation*}
of $M$-torsors. In other words, we can construct a zip $(H^1_{\dR}(A/S),\mathcal{I}_{P},\mathcal{I}_{Q},\psi)$ over $S$ of type $\mathcal{Z} = (G,P,L,Q,M,\varphi)$.
\begin{definition}
The morphism
\begin{equation*}
\zeta : \Sh \rightarrow G\Zip^{\mathcal{Z}}
\end{equation*}
is the classifying map of the universal zip $\underline{I} = (\mathcal{H}^1_{\dR},\mathcal{I}_{P},\mathcal{I}_{Q},\psi)$ associated to the universal abelian scheme $f : A \rightarrow \Sh$ over $\Sh$. For all $w \in \prescript{I}{}{W}$, we define the locally closed subscheme $\Sh_w := \zeta^{-1}(G\Zip^{\mathcal{Z}}_w)$. 
\end{definition}
\vspace{0.2cm}
Over $Y_{I_0}$, we have a universal $P_{I_0}$-reduction $\mathcal{J}$ of the $P$-torsor $\mathcal{I}_P$ corresponding to the Hodge filtration. The pair $(\underline{I},\mathcal{J})$ is a zip flag of type $(\mathcal{Z},I_0)$.
\begin{definition}
The morphism
\begin{equation*}
\zeta_{I_0} : Y_{I_0} \rightarrow G\ZipFlag^{\mathcal{Z},I_0}
\end{equation*}
is the classifying map of the universal zip flag $(\underline{I},\mathcal{J})$ of type $(\mathcal{Z},I_0)$. For all $w \in \prescript{I_0}{}{W}$, we define the locally closed subscheme ${(Y_{I_0})}_w := \zeta_{I_0}^{-1}(G\ZipFlag^{\mathcal{Z}}_w)$. 
\end{definition}
\begin{proposition}
The morphisms $\zeta$ and $\zeta_{I_0}$ are smooth and surjective.
\end{proposition}
\begin{proof}
See \cite[Theorem 3.1.2]{MR3759007} for the smoothness. See \cite{MR1827027} for the surjectivity.
\end{proof}
As generalizations lift along flat morphisms\footnote{See \cite[\href{https://stacks.math.columbia.edu/tag/03HV}{Tag 03HV}]{stacks-project}}, we deduce in particular that we have the following closure relations.
\begin{enumerate}
\item $\text{For all } w \in \prescript{I}{}{W} \text{, } \overline{\Sh}_w =  \underset{w^\prime \preccurlyeq w \text{, } w^\prime \in \prescript{I}{}{W} }{\bigsqcup} \Sh_{w^\prime}$.
\item $\text{For all } w \in \prescript{I_0}{}{W} \text{, } \overline{{(Y_{I_0})}}_w =  \underset{w^\prime \preccurlyeq w \text{, } w^\prime \in \prescript{I_0}{}{W} }{\bigsqcup} {(Y_{I_0})}_{w^\prime}$.
\end{enumerate}
We give a statement about the extension of these results on the toroidal compactifications.
\begin{proposition}
The universal zip $\underline{I}$ extends to a zip of type $\mathcal{Z}$ over the toroidal compactification $\Sh^{\tor}$ of $\Sh$. The corresponding classifying morphism $\zeta^{\tor} $ extends the morphism $\zeta$:
\begin{equation*}
\begin{tikzcd}
\Sh^{\tor} \arrow[r,"\zeta^{\tor}"] & G\Zip^{\mathcal{Z}} \\
\Sh\arrow[u] \arrow[ru,"\zeta",swap] & 
\end{tikzcd}
\end{equation*}
\end{proposition}
\begin{proof}
See \cite[Theorem 6.2.1]{MR3989256}.
\end{proof}
\begin{corollary}
The universal zip flag $(\underline{I},\mathcal{J})$ extends to a zip flag of type $(\mathcal{Z},I_0)$ over the toroidal compactification $\Sh^{\tor}$ of $\Sh$. The corresponding classifying morphism $\zeta_{I_0}^{\tor} $ extends the morphism $\zeta_{I_0}$
\begin{equation*}
\begin{tikzcd}
Y_{I_0}^{\tor} \arrow[r,"\zeta_{I_0}^{\tor}"] & G\ZipFlag^{\mathcal{Z},I_0} \\
Y_{I_0} \arrow[u] \arrow[ru,"\zeta_{I_0}",swap] &
\end{tikzcd}
\end{equation*}
\end{corollary}
\begin{proof}
The flag bundle $Y^{\tor}_{I_0}$ parametrizes the $P_0$-reductions of the $P$-torsor $\mathcal{I}_P$ of the Hodge filtration, which is defined over $\Sh^{\tor}$ by definition \ref{prop5}. The universal $P_0$-reduction $\mathcal{J}$ over $Y^{\tor}_{I_0}$, together with the pullback of the zip of the previous proposition, is a zip flag of type $(\mathcal{Z},I_0)$ over $Y^{\tor}_{I_0}$, and its classifying morphism $\zeta^{\tor}_{I_0}$ extends $\zeta_{I_0}$ since the two zip flags agree over $Y_{I_0}$.
\end{proof}
\begin{proposition}
The morphisms $\zeta^{\tor}$ and $\zeta^{\tor}_{I_0}$ are smooth.
\end{proposition}
\begin{proof}
See \cite[Theorem 1.2.]{andreatta2021mod}.
\end{proof}
For all $w \in \prescript{I_0}{}{W}$, we define the locally closed subscheme $Y^{\tor}_{I_0,w} := (\zeta^{\tor}_{I_0})^{-1}(G\ZipFlag^{\mathcal{Z},P_0}_w)$ of $Y^{\tor}_{I_0}$ and we write $\overline{Y}^{\tor}_{I_0,w}$ for its closure. By corollary \ref{cor4.20}, these subschemes partition $Y^{\tor}_{I_0}$ (no surjectivity is needed for this). As $\zeta^{\tor}_{I_0}$ is smooth, generalizations lift along it as above, so closures pull back and we have the closure relations
\begin{equation*}
\overline{Y}^{\tor}_{I_0,w} =  \underset{w^\prime \preccurlyeq w \text{, } w^\prime \in \prescript{I_0}{}{W} }{\bigsqcup} Y^{\tor}_{I_0,w^\prime}
\end{equation*}
for all $w \in \prescript{I_0}{}{W}$. In particular, the stratum $Y^{\tor}_{I_0,w}$ attached to the maximal element $w$ of $\prescript{I_0}{}{W}$ is dense in $Y^{\tor}_{I_0}$.

%% file: part5.tex
\section{Positive automorphic line bundles and Kodaira vanishing}\label{part5}
Let $\Sh^{\tor}$ be a smooth and projective toroidal compactification of the special fiber of the Siegel modular variety as in definition \ref{prop5}. Let $I_0 \subset I$ be a subset and $\pi : Y_{I_0}^{\tor} \rightarrow \Sh^{\tor}$ be the associated flag bundle that parametrizes $P_0$-reduction of the Hodge filtration over $\Sh^{\tor}$. In the last section, we have defined smooth morphisms
\begin{equation*}
\zeta^{\tor} : \Sh^{\tor} \rightarrow G\Zip^{\mathcal{Z}_{\mu}}
\end{equation*}
and 
\begin{equation*}
\zeta^{\tor}_{I_0} : Y^{\tor}_{I_0} \rightarrow G\ZipFlag^{\mathcal{Z}_{\mu},P_0},
\end{equation*}
that allowed us to construct generalized Hasse invariants on the stratification of $\Sh^{\tor}$ and $Y_{I_0}^{\tor}$. We denote by $D_{\red}$ the normal crossing Cartier divisors supported on the boundary of $\Sh^{\tor}$. Recall\footnote{See the paragraph after definition \ref{def4}.} that there exists an effective Cartier divisor $D$ whose associated reduced divisor is $D_{\red}$ and an integer $\eta_0 > 0$ such that $\omega^{\otimes \eta}(-D)$ is ample on $\Sh^{\tor}$ for every $\eta \geq \eta_0$. To lighten our notations, we write $D,D_{\red}$ instead of $\pi^{-1}D, \pi^{-1}D_{\red}$ when no confusion is possible. Following the result of \cite{StrohPrep}, we use the generalized Hasse invariants to prove that certain line bundles $\mathcal{L}_{\lambda}$ are $D$-ample\footnote{See definition \ref{def5}.} on the flag bundle $Y^{\tor}_{I_0}$. The motivation for this notion comes from the determinant $\omega$ of the Hodge bundle $\Omega^{\tor}$ which is not ample on $\Sh^{\tor}$ but only $D$-ample. Finally, we state a Kodaira-Nakano-like vanishing theorem for $D$-ample line bundles in positive characteristic from \cite{MR1193913}.
\subsection{Positive line bundles}
We recall the main positivity notion we will need for our automorphic line bundles. In this subsection $X$ is a projective variety over $k$, a field of any characteristic.
\begin{definition}[{\cite[Chap. 1]{MR2095471}}]
Let $\mathcal{L}$ a line bundle over $X$.
\begin{enumerate}
\item $\mathcal{L}$ is ample if for any coherent module $\mathcal{F}$ over $X$, there is an integer $n_0 \geq 1$ such that for all $n\geq n_0$, the sheaf $\mathcal{F} \otimes \mathcal{L}^{\otimes n}$ is globally generated.
\item Equivalently, $\mathcal{L}$ is ample if for any subvariety $V \subset X$, we have
\begin{equation*}
c_1(\mathcal{L})^{\dim V} \cdot [V] > 0,
\end{equation*}
where $c_1(\mathcal{L})$ denotes the first chern class of $\mathcal{L}$ and $\cdot$ the intersection product in the Chow ring of $X$.
\item Equivalently, $\mathcal{L}$ is ample if for any subvariety $V \subset X$ of positive dimension, there is an integer $d \geq 1$, a non-zero section $s$ of $\mathcal{L}^{\otimes d}_{|V}$ and a point $x \in V$ such that $s(x) = 0$.
\item $\mathcal{L}$ is nef if for any subvariety $V \subset X$, we have $c_1(\mathcal{L})^{\dim V} \cdot [V] \geq 0$.
\item $\mathcal{L}$ is nef as soon as for any subvariety $V \subset X$, there is an integer $d \geq 1$ and a non-zero section $s$ of $\mathcal{L}^{\otimes d}_{|V}$ (by Kleiman's criterion, since a non-zero section on a curve has degree $\geq 0$).
\item $\mathcal{L}$ is big if there is an integer $n \geq 1$ and an ample line bundle $\mathcal{A}$ such that $H^0(X,\mathcal{L}^{\otimes n} \otimes \mathcal{A}^{\otimes -1}) \neq 0$.
\end{enumerate}
\end{definition}
We now define the non-standard notion of $D$-ample line bundle on a pair $(X,D)$. It is a notion that appears in \cite{MR1193913} without being explicitly named. 
\begin{definition}\label{def5}
Let $D$ be an effective Cartier divisor of $X$ and $\mathcal{L}$ a line bundle over $X$. We say that $\mathcal{L}$ is $D$-ample if
\begin{equation*}
\exists \eta_0>0 \ \forall \eta \geq \eta_0 \ \mathcal{L}^{\otimes \eta}(-D) \text{ is ample.}
\end{equation*}
\end{definition}
We recall some known facts about $D$-ample line bundles. Since we have not find a reference, we reprove them.
\begin{proposition}
Le $D$ be an effective Cartier divisor of $X$ and $\mathcal{L}$ a line bundle over $X$. We have the following implication.
\begin{equation*}
\mathcal{L} \text{ is ample} \Rightarrow \mathcal{L} \text{ is } D\text{-ample.}
\end{equation*}
\end{proposition}
\begin{proof}
If $\mathcal{L}$ is ample, then $\mathcal{L}^{\otimes \eta}(-D) = \mathcal{L}^{\otimes \eta} \otimes \mathcal{O}_X(-D)$ must be ample for all $\eta \geq 1$ large enough.
\end{proof}
\begin{proposition}\label{prop17}
Let $\mathcal{L}$ be a line bundle over $X$. The following assertions are equivalent.
\begin{enumerate}
\item $\mathcal{L}$ is nef and big.
\item There exists an effective Cartier divisor $D$ on $X$ such that $\mathcal{L}$ is $D$-ample.
\end{enumerate}
\end{proposition}
\begin{proof}
Assume that there exists an effective Cartier divisor $D$ on $X$ such that $\mathcal{L}$ is $D$-ample. If $\mathcal{L}$ is not nef, then there is a curve $C \subset X$ such that the intersection product
\begin{equation*}
c_1(\mathcal{L}) \cdot [C]
\end{equation*}
is negative. It implies that the intersection product
\begin{equation*}
c_1(\mathcal{L}^{\otimes \eta}(-D))\cdot [C] = \eta \underbrace{\left( c_1(\mathcal{L})\cdot [C]\right)}_{<0} - {D}\cdot [C]
\end{equation*}
must be negative when $\eta$ is large enough, which contradicts the $D$-ampleness of $\mathcal{L}$. Moreover, since we can write $\mathcal{L}^{\otimes \eta_0}$ as a tensor product
\begin{equation*}
\mathcal{L}^{\otimes \eta_0} = \mathcal{L}^{\otimes \eta_0}(-D) \otimes \mathcal{O}_X(D)
\end{equation*}
of an ample line bundle with an effective line bundle, $\mathcal{L}$ is big. We are left to show the implication $(1) \Rightarrow (2)$. Since $\mathcal{L}$ is big, there exists an integer $n_0 \geq 1$ and an ample line bundle $\mathcal{A}$ such that $\mathcal{L}^{\otimes n_0}\otimes \mathcal{A}^{\otimes -1} = \mathcal{O}_X(D)$ with $D$ an effective divisor. In particular, the line bundle $\mathcal{L}^{\otimes n_0}(-D)$ is ample. Since $\mathcal{L}$ is nef and the tensor product of an ample line bundle with a nef line bundle is ample, the line bundle $\mathcal{L}^{\otimes n}(-D)$ is ample for all integer $n \geq n_0$.
\end{proof}
\begin{proposition}\label{prop_cone}
Let $D$ be an effective Cartier divisor and $\mathcal{L}$ a line bundle over $X$. If $\mathcal{L}$ and $\mathcal{L}^{\prime}$ are $D$-ample line bundles on $X$, then $\mathcal{L} \otimes \mathcal{L}^{\prime}$ is $D$-ample. If $\mathcal{L}^{\otimes n}$ is $D$-ample for a positive integer $n$, then $\mathcal{L}$ is $D$-ample.
\end{proposition}
\begin{proof}
Assume that $\mathcal{L}$ and $\mathcal{L}^\prime$ are $D$-ample line bundles. In particular, $\mathcal{L}^\prime$ is nef by proposition \ref{prop17}. For $n \geq 1$ large enough, the bundle
\begin{equation*}
{(\mathcal{L} \otimes \mathcal{L}^{\prime})}^{\otimes n}(-D) = \mathcal{L}^{\otimes n}(-D) \otimes {(\mathcal{L}^{\prime})}^{\otimes n}
\end{equation*}
is ample as the tensor product of an ample line bundle with a nef line bundle. If $\mathcal{L}^{\otimes n}$ is $D$-ample for some $n \geq 1$, it implies that $\mathcal{L}^{\otimes n}$, hence $\mathcal{L}$, is nef. It also means that there is an integer $\eta_0 \geq 1$ such that $\mathcal{L}^{\otimes n\eta_0}(-D)$ is ample. Thus, the bundle 
\begin{equation*}
\mathcal{L}^{\otimes \eta - n\eta_0} \otimes \mathcal{L}^{\otimes n\eta_0}(-D) = \mathcal{L}^{\otimes \eta}(-D)
\end{equation*}
is ample for all $\eta \geq n\eta_0$.
\end{proof}
\begin{proposition}\label{prop_nD}
Let $D$ be an effective Cartier divisor, $n \geq 1$ an integer and $\mathcal{L}$ a line bundle over $X$. The following assertions are equivalent.
\begin{enumerate}
\item $\Lc$ is $D$-ample.
\item $\Lc$ is $nD$-ample.
\end{enumerate}
\end{proposition}
\begin{proof}
Assume that $\Lc$ is $D$-ample. In particular $\Lc$ is nef and consider $\eta_0 \geq 1$ such that $\Lc^{\otimes \eta}(-D)$ is ample for all $\eta \geq \eta_0$. The bundle 
\begin{equation*}
\Lc^{\otimes \eta}(-nD) = \Lc^{\otimes \eta-n\eta_0} \otimes \left(\Lc^{\otimes \eta_0}(-D)\right)^{\otimes n}
\end{equation*}
is ample for all $\eta \geq n\eta_0$ as a tensor product of a nef line bundle with an ample line bundle. Assume that $\Lc$ is $nD$-ample. It implies that the bundle
\begin{equation*}
\left(\Lc^{\otimes \eta}(-D)\right)^{\otimes n} = \Lc^{\otimes n\eta}(-nD)
\end{equation*}
is ample for all $\eta$ large enough.
\end{proof}
\subsection{$D$-ample automorphic line bundles}
It is now convenient to introduce a relevant subset of characters of $P_0$.
\begin{definition}
We set
\begin{equation*}
\mathcal{C}_{\amp,I_0} = \left\{\lambda\in X^*(P_0) \ | \ \mathcal{L}_{\lambda} \text{ is }D\text{-ample on } Y^{\tor}_{I_0}\right\}.
\end{equation*}
\end{definition}
\begin{proposition}
The subset $\mathcal{C}_{\amp,I_0}$ is a saturated cone of $X^*(P_0)$ by proposition \ref{prop_cone} and we call it the $D$-ample cone of $Y^{\tor}_{I_0}$.
\end{proposition}
\begin{definition}
Let $\lambda  \in X^*(P_0)$ and $w \in \prescript{I_0}{}{W}$. We call a generalized Hasse invariant for $\Lc_{\lambda}$ any section $s$ of $\Lc_{\lambda}^{\otimes d}$ (for some $d\geq 1$) over $\overline{Y}^{\tor}_{I_0,w}$ that vanishes exactly on the boundary $\overline{Y}^{\tor}_{I_0,w} - Y^{\tor}_{I_0,w}$. Any $\Lc_{\lambda}$ with $\lambda \in \mathcal{C}_{\Ha,I_0,w}$ admits a generalized Hasse invariant obtained as a pullback by $\zeta^{\tor}_{I_0}$.
\end{definition}
\vspace{0.2cm}
We can now state and give a proof of the main result of this section.
\begin{theorem}
If $\lambda \in X^*(P_0)$ is a character in $\mathcal{C}_{\Ha,I_0}$, then $\mathcal{L}_{\lambda}$ is $D$-ample on $Y^{\tor}_{I_0}$. 
\end{theorem}
\begin{proof}
We start by proving that $\mathcal{L}_{\lambda}$ is nef on $Y^{\tor}_{I_0}$ for any $\lambda \in \mathcal{C}_{\Ha,I_0}$. Let $V$ be a subvariety of $Y^{\tor}_{I_0}$ and consider the minimal element $w$ of $\prescript{I_0}{}{W}$ such that $V \subset \overline{Y}^{\tor}_{I_0,w}$. Such an element always exists because $\overline{Y}^{\tor}_{I_0,w} = Y^{\tor}_{I_0}$ for the maximal element $w$ of $\prescript{I_0}{}{W}$. We consider a generalized Hasse invariant $s \in H^0(\overline{Y}^{\tor}_{I_0,w},\mathcal{L}^{\otimes d}_{\lambda})$ (for some $d \geq 1$) and we claim that the restriction $s_{|V}$ is not identically zero. If it were, we would have $$V \subset \overline{Y}^{\tor}_{I_0,w} - Y^{\tor}_{I_0,w} = \underset{w^{\prime}  \preccurlyeq w, w^\prime \neq w}{\bigsqcup} Y^{\tor}_{I_0,w^\prime},$$ which would contradict the minimality of $w$ ($V$ is irreducible). In particular, we have shown that $\mathcal{L}_{\lambda}$ is nef.

We now show that $\mathcal{L}_{\lambda}$ is $\pi$-ample, that is ample on the fibre $P/P_0$, which amounts to $\langle\lambda,\alpha^{\vee}\rangle>0$ for all $\alpha\in I\setminus I_0$. By definition \ref{def4.28}, the character $\lambda$ carries a generalized Hasse invariant on the stratum attached to $w$ for every $w\in\prescript{I_0}{}{W}$, and by proposition \ref{prop22}, the presence of one at $w$ forces
\begin{equation*}
  \sum_{i=0}^{rn-1} p^{i}\,
  \big\langle (zw^{-1})^{(i)}\,\prescript{\sigma^i}{}{\lambda},\ w\alpha^{\vee}\big\rangle>0
  \qquad\text{for every } \alpha\in E_w. \tag{$*$}
\end{equation*}
It is therefore enough to exhibit one $w$ and one $\alpha\in E_w$ at which $(*)$ becomes the desired inequality, and a single $w$ serves all of $I\setminus I_0$ at once. As $G$ is split, $\sigma$ acts trivially, $n=1$ and $(zw^{-1})^{(i)}=(zw^{-1})^{i}$. The type of $Q$ is again $I$, so $z=w_0w_{0,I}$, with $w_0=-1$ central. Setting $v:=w_{0,I}w^{-1}$, we get $zw^{-1}=w_0v$, which is an involution as soon as $v$ is one. For such a $w$, one has $r=2$, and $(*)$ collapses to
\begin{equation*}
  \langle\lambda,\ \beta-p\,v\beta\rangle>0,
  \qquad \beta:=w\alpha^{\vee}.
\end{equation*}
For a simple root $\alpha$, membership in $E_w$ is the condition $w\alpha<0$. Take $w=w_{0,I_0}w_{0,I}$, the longest element of $\prescript{I_0}{}{W_I}$, for which $v=w_{0,I_0}$ is an involution and, $w_{0,I}$ reversing the indices $j\mapsto g-j$, $w\alpha_{g-j}^{\vee}=-w_{0,I_0}\alpha_j^{\vee}$. The simple root $\alpha_{g-j}$ lies in $E_w$ precisely when $w\alpha_{g-j}=-w_{0,I_0}\alpha_j<0$, that is when $w_{0,I_0}\alpha_j>0$. Since $w_{0,I_0}$ preserves $\phi^+\setminus\phi_{I_0}^+$, this holds as soon as $\alpha_j\notin I_0$. As $v\beta=-w_{0,I_0}^2\alpha_j^{\vee}=-\alpha_j^{\vee}$, the collapsed form is
\begin{equation*}
  \langle\lambda,\ -w_{0,I_0}\alpha_j^{\vee}+p\,\alpha_j^{\vee}\rangle
  =(p-1)\langle\lambda,\alpha_j^{\vee}\rangle>0
\end{equation*}
by $W_{I_0}$-invariance of $\lambda$. This proves that $\mathcal{L}_{\lambda}$ is $\pi$-ample.

Let $\lambda$ be a character in $\mathcal{C}_{\Ha,I_0}$. Recall that $\eta_0\geq 1$ is an integer such that $\omega^{\otimes \eta}(-D)$ is ample for all $\eta \geq \eta_0$. Since $\mathcal{L}_{\lambda}$ is $\pi$-ample, we deduce that
\begin{equation*}
\mathcal{L}_{\lambda} \otimes (\pi^*\omega^{\otimes \eta}(-D))^{\otimes m}  
\end{equation*}
is ample on $Y^{\tor}_{I_0}$ for $m$ large enough. Since $\lambda$ belongs to $\mathcal{C}_{\Ha,I_0}$ and the inequalities that define $\mathcal{C}_{\Ha,I_0}$ are strict, we know that for all $n$ large enough,
\begin{equation*}
\mathcal{L}_{\lambda}^{\otimes n} \otimes \pi^*\omega^{\otimes -\eta}
\end{equation*}
has generalized Hasse invariants for all strata $Y^{\tor}_{I_0,w}$, so it is nef. Hence, we know
\begin{equation*}
\mathcal{L}_{\lambda} \otimes (\pi^*\omega^{\otimes \eta}(-D))^{\otimes m} \otimes (\mathcal{L}_{\lambda}^{\otimes n} \otimes \pi^*\omega^{\otimes -\eta})^{\otimes m} = \mathcal{L}_{\lambda}^{\otimes nm +1}(-mD)
\end{equation*}
is ample on $Y_{I_0}^{\tor}$ for $n,m$ large enough. We consider some integer $n_0, m_0 \geq 1$ such that $\mathcal{L}_{\lambda}^{\otimes n_0m_0 +1}(-m_0D)$ is ample. Since $\mathcal{L}_{\lambda}$ is nef, we must have $\mathcal{L}_{\lambda}^{\otimes \eta}(-m_0D)$ ample for all $\eta \geq n_0m_0+1$. In particular, $\mathcal{L}_{\lambda}$ is $m_0D$-ample, hence $D$-ample by proposition \ref{prop_nD}.
\end{proof}
\begin{rmrk}
The theorem can be rephrased as an inclusion
\begin{equation*}
\mathcal{C}_{\Ha,I_0} \subset \mathcal{C}_{\amp,I_0}.
\end{equation*}
\end{rmrk}
Using proposition \ref{prop16} for the existence of generalized Hasse invariants on the stack $G\ZipFlag^{\mathcal{Z}_{\mu},P_0}$, we get
\begin{theorem}\label{th_amp_auto}
Let $\lambda$ be a character of $P_0$. If $\lambda$ is orbitally $p$-close and $\mathcal{Z}_0$-ample, then $\mathcal{L}_{\lambda}$ is $D$-ample on $Y_{I_0}^{\tor}$.
\end{theorem}
\subsection{A logarithmic Kodaira-Nakano vanishing theorem in positive characteristic}
In this subsection, we review the Kodaira-Nakano vanishing theorem in positive characteristic due to Deligne and Illusie in \cite{MR894379} and a logarithmic version due to Esnault and Viehweg in \cite{MR1193913}. Let $X$ be a smooth projective variety of dimension $n$ over a perfect field $k$ of characteristic $p>0$. Let $D_{\red}$ be a normal crossing divisor of $X$. We have an open immersion $\tau : U := X-D_{\red} \rightarrow X$.
\begin{proposition}
Recall that $X$ is a smooth projective variety over $k$ and let $\mathcal{L}$ be an ample line bundle over $X$. Denote by $d$ the dimension of $X$. Assume that $(X,\mathcal{L})$ lifts to $W_2(k)$ and $p \geq d$, then
\begin{equation*}
\forall i+j > d \ H^i(X,\Omega^j_{X} \otimes \mathcal{L}) = 0.
\end{equation*}
\end{proposition}
\begin{proof}
The detailed proof can be found in \cite{MR894379}. 
\end{proof}
Over the flag bundle $Y^{\tor}_{I_0}$ of the toroidal compactification of the Siegel modular variety, we have seen that certain line bundles $\mathcal{L} _{\lambda}$ are $D$-ample for some effective divisor $D$ supported on the boundary. This motivates\footnote{This result already appears in \cite{MR3055995}.} this refined version of the result of Deligne and Illusie due to Esnault and Viehweg. 
\begin{proposition}\label{prop7}
Recall that $X$ is a smooth projective variety over $k$. Recall that $D_{\red}$ denotes a normal crossing divisor on $X$. Let $D$ be an effective Cartier divisor whose associated reduced divisor is $D_{\red}$ and let $\mathcal{L}$ be a $D$-ample line bundle on $X$. Denote by $d$ the dimension of $X$. Assume that the triple $(X,D_{\red},\mathcal{L})$ lifts to $W_2(k)$ and $p \geq d$, then
 \begin{equation*}
 \forall i+j > d \ H^i(X,\Omega^j_{X}(\log D_{\red}) \otimes \mathcal{L}(-D_{\red})) = 0.
 \end{equation*}
\end{proposition}
\begin{proof}
The proof of \cite[Proposition 11.5]{MR1193913} shows that 
\begin{equation*}
\forall i + j < \min(d,p) \ H^{i}(X,\Omega^j_{X}(\log D_{\red}) \otimes \mathcal{L}^{-1}) = 0,
\end{equation*}
which is equivalent to 
\begin{equation*}
\forall i + j > \max(2d-p,d) \ H^{i}(X,\Omega^j_{X}(\log D_{\red}) \otimes \mathcal{L}(-D_{\red})) = 0
\end{equation*}
by Serre duality. We use that for all $i + j = n$, the pairing $\Omega^i_X(\log D_{\red}) \otimes \Omega^j_X(\log D_{\red}) \rightarrow \Omega^{n}_X(\log D_{\red})$ mapping $\alpha \otimes \beta$ to $\alpha \wedge \beta$ is perfect.
\end{proof}
\begin{rmrk}
Motivated by proposition \ref{prop17}, one might be tempted to replace the assumption $D$-ample by nef and big. However, the proposition \ref{prop7} requires a normal crossing divisor.
\end{rmrk}

%% file: part6.tex
\section{Vanishing for automorphic vector bundles}\label{part6}
In this section, we prove our vanishing results announced in section \ref{sect1}. We start with some preliminary results concerning the spectral sequence associated to the cohomology of a filtered sheaf. Next, we construct the function $g_{I_0,e}$ on the power set of characters and prove that it produces new vanishing results from old ones. Finally, we give more details in the special case $g = 2$ as it is easier than the general case.
\subsection{Spectral sequence associated to a filtered sheaf}
We consider a scheme morphism $f : X \rightarrow S$ and a sheaf $\mathcal{F}$ on $X$ endowed with an increasing filtration $\mathcal{F}_{\bullet}$ with graded pieces
\begin{equation*}
\forall k \in \mathbb{Z} \ \gr_k  = \mathcal{F}_{k}/\mathcal{F}_{k-1}.
\end{equation*}
\begin{proposition}\label{prop6}
There is a spectral sequence starting at page $2$
\begin{equation*}
E_{2}^{t,k} = R^{t+k}f_*(\gr_k) \Rightarrow R^{t+k}f_*(\mathcal{F}).
\end{equation*}
\end{proposition}
\begin{proof}
This result is well-known: see the appendix of \cite{MR1193913} for example. We just recall how the differentials of the second page are defined. For all $k \in \mathbb{Z}$, there is an exact sequence 
\begin{equation*}
0 \rightarrow \mathcal{F}_{k-1}/\mathcal{F}_{k-2} \rightarrow \mathcal{F}_{k}/\mathcal{F}_{k-2} \rightarrow \mathcal{F}_{k}/\mathcal{F}_{k-1} \rightarrow 0
\end{equation*}
and the differentials are the connecting morphisms
\begin{equation*}
\forall i \geq 0 \ R^if_*(\gr_k) \rightarrow R^{i+1}f_*(\gr_{k-1}).
\end{equation*}
\end{proof}
From the study of this spectral sequence, we deduce several results.
\begin{lemma}\label{prop10}
Let $i_0 \geq 0$ and assume 
\begin{equation*}
\forall k \in \mathbb{Z} \ R^{i_0}f_*(\gr_k) = 0.
\end{equation*}
Then,
\begin{equation*}
R^{i_0}f_*(\mathcal{F}) = 0.
\end{equation*}
\end{lemma}
\begin{proof}
We pass from a page of a spectral sequence to the next one by taking cohomology and since $E_2^{i_0-k,k} = R^{i_0}f_*(\gr_k) = 0$ for all $k \in \mathbb{Z}$, we have 
\begin{equation*}
\forall a \geq 2 \ \forall k \in \mathbb{Z} \ E_a^{i_0-k,k} = 0.
\end{equation*}
Thus,
\begin{equation*}
\forall a \geq 2 \ \forall k \in \mathbb{Z} \ E_{\infty}^{i_0-k,k} = 0 
\end{equation*}
and
\begin{equation*}
R^{i_0}f_*(\mathcal{F}) = 0.
\end{equation*}
\end{proof}

\begin{lemma}\label{prop11}
Let $i_0 \geq 0$ and assume that there exists $n \in \mathbb{Z}$ such that for all $k > n$, $\gr_{k} = 0$. If
\begin{equation*}
\left\{
\begin{aligned}
& R^{i_0}f_*(\mathcal{F}) = 0, \\
&\forall k \leq n-1 \ R^{i_0+1}f_*(\gr_k) = 0,
\end{aligned}
\right.
\end{equation*}
then 
\begin{equation*}
R^{i_0}f_*(\gr_n) = 0.
\end{equation*}
\end{lemma}
\begin{proof}
For a visual support, see the figure \ref{fig1}. From the hypothesis on the graded pieces, we know that for all $a\geq 2$, the differential with target $E_{a}^{-n+i_0,n}$ vanishes. Since for all $k \leq n-1$ we have $R^{i_0+1}f_*(\gr_k) = 0$, then for all $a \geq 2$ the differential with source $E_a^{-n+i_0,n}$ must vanish. Thus, we get $R^{i_0}f_*(\gr_n) = E_2^{-n+i_0,n} = E_{\infty}^{-n+i_0,n}$ and $E_{\infty}^{-n+i_0,n} = 0$ as a graded piece of $R^{i_0}f_*(\mathcal{F})$. 
\begin{figure}
\centering
\vspace{3.5cm}
\begin{tikzpicture}[transform canvas={scale=0.8}]
  \matrix (m) [matrix of math nodes, nodes in empty cells,nodes={minimum width=5ex, minimum height=5.5ex,outer sep=-5pt}, column sep=1ex,row sep=1ex]{
           k  &   &   &   &   & \\
          n+1 \ \ \ \    & 0 &  0 & 0 & 0  & 0\\
          n     &  f_*(\gr_n) &  R^1f_*(\gr_n)   & R^2f_*(\gr_n)  &  R^3f_*(\gr_n) & \cdots \\
          n-1  \ \ \ \    &  0 &  f_*(\gr_{n-1})  & R^1f_*(\gr_{n-1}) & R^2f_*(\gr_{n-1})  & \cdots \\
          n-2  \ \ \ \    &  0 & 0&  f_*(\gr_{n-2})   & R^1f_*(\gr_{n-2})  & \cdots  \\
          n-3  \ \ \ \    &  0 & 0&  0 & f_*(\gr_{n-3})  & \cdots \\
    \quad\strut &   -n  &  -n+1  &  -n+2  &  -n+3 & \cdots & t \strut  \\};
  	\draw[-stealth] (m-2-2.south east) -- (m-3-4.north west);
    \draw[-stealth] (m-2-3.south east) -- (m-3-5.north west);
	\draw[-stealth] (m-3-2.south east) -- (m-4-4.north west);
    \draw[-stealth] (m-3-3.south east) -- (m-4-5.north west);
    \draw[-stealth] (m-4-2.south east) -- (m-5-4.north west);
    \draw[-stealth] (m-4-3.south east) -- (m-5-5.north west);
    \draw[-stealth] (m-5-3.south east) -- (m-6-5.north west);

\draw[thick,-stealth] (m-7-1.east) -- (m-1-1.east) ;
\draw[thick,-stealth] (m-7-1.north) -- (m-7-7.north) ;
\end{tikzpicture}
\vspace{3.5cm}
\caption{$E_2$-page of the spectral sequence.}\label{fig1}
\end{figure}
\FloatBarrier
\end{proof}
\subsection{The general case}
The goal of this subsection is to explain how to deduce new vanishing results for the coherent cohomology from known ones. Other Shimura varieties could be considered but we have restricted ourselves to the Siegel case for simplicity. We recall the notations. Let $\Sh^{\tor}$ be the special fiber over $\mathbb{F}_p$ of the Siegel modular variety of genus $g \geq 2$ and $\pi : Y_{I_0}^{\tor} \rightarrow \Sh^{\tor}$ the flag bundle in $P/P_0$ where $P_0 \subset \Sp_{2g}$ is a parabolic subgroup of type $I_0 \subset I \subset \Delta$ which is contained in the parabolic $P \subset \Sp_{2g}$ of type $I$. We denote by $D_{\red}$ the normal crossing divisor supported on the boudary of $\Sh^{\tor}$. We use the same notation $D_{\red}$ for the normal crossing divisor $\pi^{-1}D_{\red}$ of $Y_{I_0}^{\tor}$ when no confusion is possible. We denote by $d$, $d_0$ the dimension of $\Sh^{\tor}$, $Y_{I_0}^{\tor}$ and $r_0 = d_0-d$ the relative dimension of $\pi$. We choose a system of positive roots in a way to obtain
\begin{equation*}
I = \{e_i-e_{i+1} \ | \ i = 1, \cdots g-1 \}  \subset \Delta = \{e_i-e_{i+1} \ | \ i = 1, \cdots g-1 \} \cup \{2e_g\}.
\end{equation*}
The Levi subgroup $L$ of $P \subset \Sp_{2g}$ is $\GL_g$ and to each representation $V$ of $L$, we have an associated vector bundle $\mathcal{W}(V)$ on $\Sh^{\tor}$. With our conventions, the Hodge bundle $\Omega$ is the vector bundle of rank $g$ associated to the standard representation $\std_L$ of $L$. To each character $\lambda$ of $P_0$, we have an associated line bundle $\mathcal{L}_{\lambda}$ on $Y_{I_0}^{\tor}$. Assuming that $p > d_0$, the basic idea is to use the logarithmic Kodaira-Nakano vanishing theorem (see proposition \ref{prop7}) on the flag bundle $Y^{\tor}_{I_0}$ with $D$-ample line bundle $\mathcal{L}_{\lambda}$. Since the determinant of 
\begin{equation*}
\Omega^1_{\smash{Y^{\tor}_{I_0}}}(\log D_{\red})
\end{equation*}
is a line bundle over $\smash{Y^{\tor}_{I_0}}$, it is not hard to express it as an automorphic bundle and it provides vanishing results for cohomology groups $H^i$ with $i>0$. The accessible weights with this method are regular. To access less regular weights, a natural idea is to use the logarithmic Kodaira-Nakano vanishing theorem for 
\begin{equation*}
\Omega^{m}_{\smash{Y^{\tor}_{I_0}}}(\log D_{\red})
\end{equation*}
with $m < d_{0}$. However, this bundle is not a line bundle and doesn't seem related to automorphic bundles (see remark \ref{rmrk1}). A solution is to filter it by automorphic vector bundles and then use the associated spectral sequence. The following result is well-known but since we haven't found a reference, we give a proof.
\begin{lemma}
We have an exact sequence of vector bundles
\begin{equation*}
\begin{tikzcd}[column sep = small, row sep = small]
0 \arrow[r] &\pi^* \Omega^1_{\Sh^{\tor}}(\log D_{\red}) \arrow[r] & \Omega^1_{Y^{\tor}_{I_0}}(\log D_{\red}) \arrow[r] & \Omega^1_{Y^{\tor}_{I_0}/\Sh^{\tor}} \arrow[r] & 0.
\end{tikzcd}
\end{equation*}
\end{lemma}
\begin{proof}
By \cite[II. §3.]{MR0417174}, we have a commutative diagram 
\begin{equation*}
\begin{tikzcd}[column sep = small]
0 \arrow[r] &\pi^* \Omega^1_{\Sh^{\tor}} \arrow[d,"a"] \arrow[r] & \pi^*\Omega^1_{\Sh^{\tor}}(\log D_{\red}) \arrow[d,"b"] \arrow[r] & \pi^*\mathcal{O}_{D_{\red}} \arrow[r] \arrow[d,equal] & 0 \\
0 \arrow[r] &\Omega^1_{Y^{\tor}_{I_0}} \arrow[r] & \Omega^1_{Y^{\tor}_{I_0}}(\log D_{\red}) \arrow[r] & \mathcal{O}_{\pi^{-1}D_{\red}} \arrow[r] & 0
\end{tikzcd}
\end{equation*}
where the rows are exact (use also that $\pi$ is flat for the first row) and $\mathcal{O}_{D_{\red}}$ abbreviates the push-forward of the structure sheaf of the normalisation of $D_{\red}$. Since $\pi$ is smooth, $\ker a = 0$ and by the snake lemma, the sequence
\begin{equation*}
\begin{tikzcd}[column sep = small, row sep = small]
 0 \arrow[r] & \ker b \arrow[r] & 0 \arrow[r] & \Omega^1_{Y^{\tor}_{I_0}/\Sh^{\tor}} \arrow[r] & \Omega^1_{Y^{\tor}_{I_0}}(\log D_{\red})/\pi^*\Omega^1_{\Sh^{\tor}}(\log D_{\red}) \arrow[r] & 0
\end{tikzcd}
\end{equation*}
is exact. The desired exact sequence is obtained from $b$.
\end{proof}
\begin{definition}\label{def:filt}
Let $e \geq 0$ be an integer. We define an increasing filtration $F_\bullet$ of $\Omega^{d_0-e}_{Y^{\tor}_{I_0}}(\log D_{\red})$ by
\begin{equation*}
F_k = \pi^*\Omega^{d_{0}-e-k}_{\Sh^{\tor}} (\log D_{\red}) \wedge \Omega^k_{Y^{\tor}_{I_0}} (\log D_{\red}),
\end{equation*}
with graded pieces
\begin{equation*}
\gr_k = \pi^*\Omega^{d_{0}-e-k}_{\Sh^{\tor}} (\log D_{\red})\otimes \Omega^k_{Y^{\tor}_{I_0}/\Sh^{\tor}}.
\end{equation*}
From the proposition \ref{prop6}, we get an associated spectral sequence starting at page $2$ for each $\lambda \in X^*(P_0)$
\begin{multline}\label{sp1}
E_{2,e,\lambda}^{t,k} = H^{t+k}(Y^{\tor}_{I_0},\gr_k \otimes \mathcal{L}_{\lambda}(-D_{\red})) \\ \Rightarrow H^{t+k}(Y^{\tor}_{I_0},\Omega^{d_{0}-e}_{Y^{\tor}_{I_0}}(\log D_{\red}) \otimes \mathcal{L}_{\lambda}(-D_{\red})).
\end{multline}
\end{definition}
This spectral sequence doesn't degenerate in general, so we need to consider weights $\lambda$ that ensure partial degeneration results. This will allow us to deduce vanishing results for tensor product of the form
\begin{equation*}
\Omega^k_{\Sh^{\tor}}(\log D_{\red}) \otimes \nabla(\lambda).
\end{equation*}
Another difficulty arises because, in positive characteristic, algebraic representations of reductive groups are not semi-simple, so we can't easily deduce vanishing results for automophic bundles from vanishing results for such tensor products. However, from proposition \ref{prop18}, $\Omega^k_{\Sh^{\tor}}(\log D_{\red})$ admits a $\nabla$-filtration if $p > d$ and we can use corollary \ref{cor2} to see that the tensor product $\Omega^k_{\Sh^{\tor}}(\log D_{\red}) \otimes \nabla(\lambda)$ admits also a $\nabla$-filtration : this will allow us to deduce new vanishing results for automophic bundles. Since our method relies heavily on partial degeneration results that requires vanishing results, we can think of it as a way to deduce new vanishing results from known ones. This is why we present them in two steps:
\begin{itemize}
\item \emph{Degeneration:} We determine the vanishing results we need to ensure the degeneration of relevant spectral sequences.
\item \emph{Propagation:} Given a set of known vanishing results, we determine the new vanishing results we can deduce from them.
\end{itemize}
To lighten our notations, we will denote the subcanonical automorphic bundle by $\nabla^{\sub}(\lambda)$ instead of $\nabla(\lambda)(-D_{\red})$ and $\mathcal{L}^{\sub}(V)$ instead of $\mathcal{L}(V)(-D_{\red})$. We introduce some notations for the weights of our automorphic bundles.
\begin{definition}
For all $0 \leq n \leq d$, let $(\nu^{n}_j)_j$ be the highest weights of the graded pieces of a $\nabla$-filtration of the $L$-representation
\begin{equation*}
\wedge^{n} \Sym^2 \std_L,
\end{equation*}
which exists for $p > d$ by proposition \ref{prop18} below. They form a multiset which is read off the character of $\wedge^{n} \Sym^2 \std_L$, hence is independent of the filtration. Following the convention of definition \ref{def3}, we set
\begin{equation*}
\mu^{n}_j = w_0w_{0,L}\nu^{n}_j.
\end{equation*}
\end{definition}
\begin{proposition}\label{prop18}
If $p > d = g(g+1)/2$, then for any $n \geq 1$ the vector bundle $\Omega^n_{\Sh^{\tor}}(\log D_{\red})$ admits a filtration
\begin{equation*}
0 =\mathcal{V}^s \subsetneq \mathcal{V}^{s-1} \subsetneq
 \cdots \subsetneq
 \mathcal{V}^0 = \Omega^n_{\Sh^{\tor}}(\log D_{\red}),
\end{equation*}
where the graded pieces are automorphic vector bundles of the form $\nabla(\mu^n_j)$ with $\mu^n_j$ dominant.
\end{proposition}
\begin{proof}
Recall from proposition \ref{prop19} that the Kodaira-Spencer map induces an isomorphism
\begin{equation*}
\Omega^1_{\Sh^{\tor}}(\log D_{\red}) = \mathcal{W}(\Sym^2 \std_L) = \nabla(0, \cdots, 0,-2).
\end{equation*}
We only need to see that for any $1 \leq n \leq \frac{g(g+1)}{2}$, the $\GL_g$-module $\Lambda^n \Sym^2 \std_L$ admits a $\nabla$-filtration. The module $$\Sym^2 \std_L = \nabla(2,0,\cdots,0)$$ is already a costandard module. From proposition \ref{prop8}, the module $$(\Sym^2 \std_L)^{\otimes n}$$ admits also a $\nabla$-filtration. Since $p> \frac{g(g+1)}{2} \geq n$, $p$ does not divide $n!$ and the surjection of $G$-modules
\begin{equation*}
(\Sym^2 \std_L)^{\otimes n} \rightarrow \Lambda^n \Sym^2 \std_L
\end{equation*}
admits a $\GL_g$-equivariant section $s$ defined by the formula
\begin{equation*}
s(v_1 \wedge \cdots \wedge v_{n}) = \frac{1}{n!}\sum_{\sigma \in \mathfrak{S}_n}\varepsilon(\sigma)v_{\sigma(1)}\otimes \cdots \otimes v_{\sigma(n)}.
\end{equation*}
As a direct factor of $(\Sym^2 \std_L)^{\otimes n}$, the corollary \ref{cor1} implies that \\$\Lambda^n \Sym^2 \std_L$ admits a $\nabla$-filtration.
\end{proof}
In general, $\Lambda^n \Sym^2 \std_L$ has no unique highest weight. For instance, for $g = 3$ and $n = 3$, we have $\Lambda^3\Sym^2\std_{\GL_3}=S^{(4,1,1)}\oplus S^{(3,3,0)}$ in characteristic $0$, and $(4,1,1)-(3,3,0)=(1,-2,1)$ is not a non-negative combination of positive roots. More precisely, by \cite[I.8, Ex. 9]{MR1354144}, the weights $\nu^n_j$ are the partitions of $2n$ with Frobenius symbol $(\alpha+1\mid\alpha)$ and at most $g$ rows, each occurring once. Comparing the partial sums of these partitions with those of their conjugates shows that they are pairwise incomparable for the dominance order. Since $w_0w_{0,L}$ preserves the dominance order, the same holds for the $\mu^n_j$: for every $n$, the $\mu^n_j$ are pairwise distinct and pairwise incomparable, so that each of them is maximal.

From now on, we use the following notations. We write $\alpha_i = e_i-e_{i+1}$ for $1 \leq i \leq g-1$, so that $I = \{\alpha_1,\cdots,\alpha_{g-1}\}$, and we denote by $\phi_L^+$ (resp. $\phi_{I_0}^+$) the positive roots of $L$ (resp. of the Levi subgroup $L_0$ of $P_0$). We set
\begin{equation*}
S = \phi_L^+ - \phi_{I_0}^+, \qquad s_M = \sum_{\beta \in M} \beta, \qquad M^c = S - M \quad \text{for } M \subset S,
\end{equation*}
so that $r_0 = |S|$. The type $I_0$ cuts $\{1,\cdots,g\}$ into \emph{blocks}, the intervals of consecutive indices delimited by the $i$ with $\alpha_i\notin I_0$, and $L_0$ is a product of general linear groups, one for each block. We write $B_i$ for the block containing $i$. A weight is $L_0$-dominant if it is weakly decreasing along each block. For a $L_0$-dominant weight $\nu$, we denote by $\nabla_{L_0}(\nu)$ the corresponding costandard $L_0$-module, seen as a $P_0$-module via $P_0 \twoheadrightarrow L_0$.
\begin{xproposition}{6.7a}\label{prop:goodfilt}
Assume $p>r_0$ and write $\mathfrak u=\Lie L/\Lie(P_0\cap L)$. The $P_0$-module $\Lambda^j(\mathfrak u)^{\vee}$ carries a filtration by $P_0$-submodules whose graded pieces are $L_0$-modules, inflated to $P_0$ along $P_0\twoheadrightarrow L_0$, each admitting a good $L_0$-filtration. Consequently $\Lambda^j(\mathfrak u)^{\vee}$ has a $P_0$-stable filtration with graded pieces $\nabla_{L_0}(-s_M)$, where each $-s_M$ is an $L_0$-dominant weight with $M\subset S$ and $|M|=j$. These $-s_M$, the highest weights of the pieces, form a multiset $\mathcal{W}_{I_0}(j)$ independent of the filtration. For $g\leq 3$ the hypothesis $p>r_0$ is unnecessary.
\end{xproposition}
\begin{proof}
Since $\Lambda^j(\mathfrak u)^{\vee}=\Lambda^j(\mathfrak u^{\vee})$ we work with $\mathfrak u^{\vee}$. Write $L_0=\prod_a\GL_{n_a}$, the product running over the blocks cut out by $I_0$, $V_a$ for the standard module of the $a$-th block, and $z_a\in X^*(Z(L_0)^\circ)$ for the character by which the connected centre acts on $V_a$. Decomposing $\mathfrak u^{\vee}$ into eigenspaces for $Z(L_0)^\circ$ gives
\begin{equation*}
\mathfrak u^{\vee}=\bigoplus_{a<b}V_a^{\vee}\otimes V_b,
\end{equation*}
the summand $V_a^{\vee}\otimes V_b$ carrying the character $z_b-z_a$. The cocharacter acting on the $a$-th block by $t\mapsto t^{a}$ sends $z_b-z_a$ to $b-a$, collecting the summands into
\begin{equation*}
\mathfrak u^{\vee}=\bigoplus_{m\geq 1}\mathfrak u_m,\qquad \mathfrak u_m=\bigoplus_{b-a=m}V_a^{\vee}\otimes V_b.
\end{equation*}
The cocharacter is central in $L_0$, so each $\mathfrak u_m$ is $L_0$-stable, while the unipotent radical of $P_0$ has strictly positive degree and therefore raises it. The subspaces $F^c=\bigoplus_{m\geq c}\mathfrak u_m$ are accordingly $P_0$-stable, with $\gr^c=\mathfrak u_c$ an $L_0$-module on which that radical acts trivially. The induced filtration on $\Lambda^j$ is again $P_0$-stable, with graded pieces the $L_0$-modules $\bigotimes_c\Lambda^{j_c}\mathfrak u_c$, $\sum_c j_c=j$.

For the good filtrations, one factor at a time: for the standard module of $\GL_n$ and its dual, the Weyl, simple and costandard modules coincide. Since $V_a^{\vee}$ and $V_b$ are costandard on their own factors, so is $V_a^{\vee}\otimes V_b$ on $L_0$, and each $\mathfrak u_c$, a direct sum of these, has a good $L_0$-filtration. As $p>r_0\geq j\geq j_c$, the exterior power $\Lambda^{j_c}\mathfrak u_c$ is a direct summand of $\mathfrak u_c^{\otimes j_c}$, as in the proof of proposition \ref{prop18}; that tensor power has a good filtration by corollary \ref{cor2}, and good filtrations pass to direct summands by corollary \ref{cor1}. Corollary \ref{cor2} again gives one to the graded piece $\bigotimes_c\Lambda^{j_c}\mathfrak u_c$. Multiplicities are determined by the character, whose weights are the $-s_M$ with $|M|=j$.

At $g\leq 3$ the appeals to corollary \ref{cor2}, and with them the hypothesis $p>r_0$, are avoidable. At $I_0=I$ one has $\mathfrak u=0$. At $I_0=\emptyset$ one has $P_0\cap L=B_L$ and $L_0=T$, so $\Lambda^j(\mathfrak u^{\vee})$ has a complete flag of $B_L$-submodules whose graded pieces are the characters $-s_M$ with $|M|=j$, and $\nabla_T(-s_M)$ is that character. There remains $g=3$ with $I_0$ a singleton, which cuts two blocks, of sizes $2$ and $1$: the grading above has the single step $\mathfrak u^{\vee}=V_a^{\vee}\otimes V_b$, of dimension $2$, so the exterior powers are $k$, $\mathfrak u^{\vee}$ and $\det\mathfrak u^{\vee}$, costandard by the previous paragraph and because the outer two are lines.
\end{proof}
\begin{proposition}\label{prop9}
We have an isomorphism 
\begin{equation*}
\Omega^1_{Y^{\tor}_{I_0}/\Sh^{\tor}} = \mathcal{L}({\Lie L/\Lie (P_0 \cap L)})^{\vee},
\end{equation*}
and for all $i\geq 0$, $\Omega^{i}_{Y^{\tor}_{I_0}/\Sh^{\tor}} = \mathcal{L}(\Lambda^i(\Lie L/\Lie (P_0 \cap L))^{\vee})$. If $p > r_0$ or $g \leq 3$, the vector bundle $\Omega^{i}_{Y^{\tor}_{I_0}/\Sh^{\tor}}$ admits a filtration with graded pieces
\begin{equation*}
\mathcal{L}(\nabla_{L_0}(-s_M)) \text{ for } -s_M \in \mathcal{W}_{I_0}(i),
\end{equation*}
where $\mathcal{W}_{I_0}(i)$ is the multiset of proposition \ref{prop:goodfilt}. In particular, $\Omega^{r_{0}}_{Y^{\tor}_{I_0}/\Sh^{\tor}} \simeq \mathcal{L}_{-2\rho_{I_0}}$ with
\begin{equation*}
\rho_{I_0}= \frac{1}{2} \sum_{\alpha \in \phi_L^{+} \backslash \phi_{I_0}^{+}} \alpha.
\end{equation*}
\end{proposition}
\begin{proof}
Consider the cartesian diagram
\begin{equation*}
\begin{tikzcd}
Y^{\tor}_{I_0} \arrow[d,"\pi"] \arrow[r,"\tilde{\zeta}"] & \lfloor P_0 \backslash * \rfloor \arrow[d,"\tilde{\pi}"] \\
\Sh^{\tor} \arrow[r,"\zeta"] & \lfloor P \backslash * \rfloor
\end{tikzcd}
\end{equation*}
where the horizontal arrows correspond to the universal $P$-torsor on $\Sh^{\tor}$ and the universal $P_0$-torsor on $Y^{\tor}_{I_0}$ and where the vertical arrow $\tilde{\pi}$ between the classifying stacks is induced by the inclusion $P_0 \subset P$. Coherent sheaves on the classifying stack $\lfloor P_0 \backslash * \rfloor$ are algebraic representations of $P_0$ and clearly, we have
\begin{equation*}
\Omega^1_{\tilde{\pi}} = \Lie(P)/\Lie(P_0)^{\vee},
\end{equation*}
where the action of $P_0$ on $\Lie(P)$ is induced by the restriction of the adjoint action of $P$. From the isomorphism
\begin{equation*}
\tilde{\zeta}^*\Omega^1_{\tilde{\pi}} =\Omega^1_{\pi},
\end{equation*}
we deduce that $$\Omega^{1}_{Y^{\tor}_{I_0}/\Sh^{\tor}} = \Lc(\Lie(P)/\Lie(P_0)^{\vee}),$$ and $\Lie(P)/\Lie(P_0) = \Lie L/\Lie (P_0 \cap L)$. Taking exterior powers, we get $\Omega^{i}_{Y^{\tor}_{I_0}/\Sh^{\tor}} = \mathcal{L}(\Lambda^i(\Lie L/\Lie (P_0 \cap L))^{\vee})$, and the filtration follows from proposition \ref{prop:goodfilt} since $V \mapsto \mathcal{L}(V)$ is exact. Since the $T$-weights on $\Lie(P)/\Lie(P_0)^{\vee}$ are the $-\alpha$ with $\alpha \in S$, the top exterior power is the character $-s_S = -2\rho_{I_0}$.
\end{proof}
\begin{rmrk}\label{rmrk1}
The exact sequence 
\begin{equation*}
0 \rightarrow \pi^*\Omega^1_{\Sh^{\tor}}(\log D_{\red}) \rightarrow \Omega^1_{Y^{\tor}_{I_0}}(\log D_{\red}) \rightarrow \Omega^1_{Y^{\tor}_{I_0}/\Sh^{\tor}} \rightarrow 0
\end{equation*}
doesn't seem to split and we cannot prove the vanishing of the abelian group $$\smash{\Ext^1_{\mathcal{O}_{Y^{\tor}_{I_0}}}(\Omega^1_{Y^{\tor}_{I_0}/\Sh^{\tor}},\pi^*\Omega^1_{\Sh^{\tor}}(\log D_{\red}))}$$ using known vanishing results because the vector bundle $$\pi^*\Omega^1_{\Sh^{\tor}}(\log D_{\red}) \otimes {\Omega^1}_{Y^{\tor}_{I_0}/\Sh^{\tor}}^{\vee}$$ is filtered by $\mathcal{L}_{\lambda}$'s with $\lambda \in X^*(P_0)$ outside the anti-dominant Weyl chamber for which the first cohomology is non-zero in general. Outside the case $I_0 = I$, we don't even know if $\Omega^1_{\smash{Y^{\tor}_{I_0}}}(\log D_{\red})$ is automorphic. In other words, we don't know if $\Omega^1_{\smash{Y^{\tor}_{I_0}}}(\log D_{\red})$ is of the form $\mathcal{L}(V)$ for an algebraic representation $V$ of $P_0$.
\end{rmrk}
\subsubsection{Degeneration}
To push the graded pieces of proposition \ref{prop9} down along $\pi$, we need to control the higher cohomology of line bundles on $L/B_L$ which are not dominant. Kempf's vanishing theorem is complemented by the following characteristic-free statement.
\begin{xlemma}{6.9a}[Bott-singular rescue]\label{lem:rescue}
Let $\nu$ be a character of $T$ such that $\langle\nu,\alpha^{\vee}\rangle=-1$ for some simple root $\alpha$ of $L$. Then $H^k(L/B_L,\mathcal{L}_\nu)=0$ for all $k \geq 0$, in every characteristic.
\end{xlemma}
\begin{proof}
Let $P_\alpha\supset B_L$ be the standard parabolic subgroup of $L$ of type $\{\alpha\}$, and factor the structure map of $L/B_L$ through the $\mathbb P^1$-bundle $q_\alpha:L/B_L\rightarrow L/P_\alpha$. On the fibre of $q_\alpha$ over a point $x$ the bundle $\mathcal{L}_\nu$ restricts to $\mathcal{O}(\langle\nu,\alpha^{\vee}\rangle)=\mathcal{O}(-1)$, so $H^k(q_\alpha^{-1}(x),\mathcal{L}_\nu)=0$ for every $k$ and every $x$. Now $q_\alpha$ is proper and flat and $\mathcal{L}_\nu$ is flat, so proper base change (proposition \ref{prop3}) gives $R^kq_{\alpha*}\mathcal{L}_\nu=0$ for every $k$. Hence $R\Gamma(L/B_L,\mathcal{L}_\nu)=R\Gamma(L/P_\alpha,Rq_{\alpha*}\mathcal{L}_\nu)=0$.
\end{proof}
Kempf's theorem needs $\nu$ dominant and lemma \ref{lem:rescue} needs $\langle\nu,\alpha^{\vee}\rangle=-1$ for some simple $\alpha$, so the two together apply as soon as $\langle\nu,\alpha^{\vee}\rangle\geq-1$ for every simple $\alpha$. The weights we have to push down are of the form $\lambda-s_M$, and we now measure how much positivity must be asked of $\lambda$ for this to hold. In $\Omega^{d_0-e}_{Y^{\tor}_{I_0}}(\log D_{\red})$, only the $\Omega^k_{Y^{\tor}_{I_0}/\Sh^{\tor}}$ with $k\geq r_0-e$ occur, so only the $M \subset S$ with $|M^c|\leq\min(e,r_0)$ are met. For $e\geq0$ and a simple root $\alpha\in I$, we set
\begin{align*}
N_e(\alpha,I_0)&:=\max\big\{-\langle s_{M^c},\alpha^{\vee}\rangle\ :\ M\subset S,\ |M^c|\leq\min(e,r_0),\ -s_M\ L_0\text{-dominant}\big\},\\[2pt]
\delta_{I_0,e}(\alpha)&:=\max\{0,\,N_e(\alpha,I_0)-1\}.
\end{align*}
\begin{xlemma}{6.9b}\label{lem:threshold}
Let $e\geq0$ and $\alpha_i\in I$. Then
\begin{equation*}
N_e(\alpha_i,I_0)=
\begin{cases}
0, & \alpha_i\in I_0,\\[2pt]
\min\big\{e,\ g-|B_i|-|B_{i+1}|\big\}, & \alpha_i\notin I_0.
\end{cases}
\end{equation*}
Consequently $\delta_{I_0,e}$ vanishes on $I_0$ and equals $\max\big\{0,\ \min\{e,\ g-|B_i|-|B_{i+1}|\}-1\big\}$ at $\alpha_i\notin I_0$. Moreover $N_e(\alpha_i,I_0)\leq\min\{e,g-2\}$. Both $N_e(\alpha_i,I_0)$ and $\delta_{I_0,e}(\alpha_i)$ weakly decrease as $I_0$ grows, and neither depends on $e$ once $e\geq g-2$. Finally $\delta_{I_0,e}$ vanishes identically for $e\leq1$ at every genus and for $g\leq3$ at every $e$, and is at most $1$ for $e=2$.
\end{xlemma}
\begin{proof}
Let $\alpha_m\in I_0$, so that $m$ and $m+1$ lie in a common block $B$. A root of $S$ pairs nontrivially with $\alpha_m^{\vee}$ only if one of its two feet is $m$ or $m+1$, and none has both, the pair $\{m,m+1\}$ lying inside $B$. Those that remain cancel two by two: for each $q\notin B$ the root of $S$ with feet $\{m,q\}$ and the root with feet $\{m+1,q\}$ contribute $+1$ and $-1$, in an order depending on which side of $B$ the index $q$ lies. Hence $\langle 2\rho_{I_0},\alpha_m^{\vee}\rangle=\langle s_S,\alpha_m^{\vee}\rangle=0$, so $-s_S$ is $L_0$-dominant. Now $s_{M^c}=s_S-s_M$, so $-\langle s_{M^c},\alpha_m^{\vee}\rangle=\langle s_M,\alpha_m^{\vee}\rangle$, which the $L_0$-dominance of $-s_M$ makes $\leq0$, with equality at $M=S$, admitted since $|M^c|=0$. This is the first case.

Fix now $\alpha_i\in I\setminus I_0$. For $\beta=e_j-e_k\in\phi_L^+$ with $j<k$, one has $\langle\beta,\alpha_i^{\vee}\rangle=\delta_{ji}-\delta_{j,i+1}-\delta_{ki}+\delta_{k,i+1}$, so the roots of $\phi_L^+$ pairing negatively with $\alpha_i^{\vee}$ are the $e_j-e_i$ with $j<i$ and the $e_{i+1}-e_k$ with $k>i+1$, each with pairing $-1$, and $\langle\beta,\alpha_i^{\vee}\rangle\geq-1$ for all $\beta\in\phi_L^+$. Put $M^c_0:=\{\beta\in S:\langle\beta,\alpha_i^{\vee}\rangle<0\}$. A root $e_j-e_i$ (resp. $e_{i+1}-e_k$) as above lies in $\phi_{I_0}^+$, hence outside $S$, exactly when $j\in B_i$ (resp. $k\in B_{i+1}$). As $\alpha_i\notin I_0$, the block $B_i$ ends at $i$ and $B_{i+1}$ begins at $i+1$, so $e_j-e_i\mapsto j$ and $e_{i+1}-e_k\mapsto k$ identify
\begin{equation*}
M^c_0\ \xrightarrow{\ \sim\ }\ \{1,\dots,g\}\setminus(B_i\cup B_{i+1}),
\qquad N:=|M^c_0|=g-|B_i|-|B_{i+1}|,
\end{equation*}
an index set meeting only the blocks strictly below $B_i$ and strictly above $B_{i+1}$. As $B_i$ and $B_{i+1}$ are distinct and nonempty, $N\leq g-2$, and $N\leq|S|=r_0$.

\emph{Upper bound.} For every $T\subset S$, we have $-\langle s_T,\alpha_i^{\vee}\rangle\leq|T\cap M^c_0|\leq\min\{|T|,N\}$, and maximising over $|T|\leq\min(e,r_0)$ gives $N_e(\alpha_i,I_0)\leq\min\{e,N\}$.

\emph{Lower bound.} Call $T\subset S$ \emph{admissible} when $-s_{S-T}$ is $L_0$-dominant. The sets competing in $N_e(\alpha_i,I_0)$ are then the admissible $T$ with $|T|\leq\min(e,r_0)$. As $\langle s_S,\alpha_m^{\vee}\rangle=0$ for $\alpha_m\in I_0$, admissibility reads $\langle s_T,\alpha_m^{\vee}\rangle\geq0$ for every $\alpha_m\in I_0$. It is enough to produce an admissible $T\subset M^c_0$ of cardinality $\min\{e,N\}$, since such a $T$ satisfies $-\langle s_T,\alpha_i^{\vee}\rangle=|T|$ and $|T|\leq\min(e,r_0)$. We construct one of every cardinality $c\leq N$. For $T\subset M^c_0$, let $J$ be its index set under the identification above, and let $a$ and $b$ count the elements of $J$ below $B_i$ and above $B_{i+1}$. The coordinates of $s_T$ are then $+1$ at each element of $J$ below, $-1$ at each element above, $-a$ at $i$, $+b$ at $i+1$, and $0$ elsewhere. Admissibility asks exactly that these coordinates be weakly decreasing along each block. In $B_i$ the sole nonzero entry is $-a$, at its last position, and in $B_{i+1}$ it is $+b$, at its first, so both blocks impose nothing. In a block below $B_i$ the entries are $+1$ on $J$ and $0$ off it, so the condition is that $J$ meet it in an initial segment, and in a block above $B_{i+1}$, where they are $0$ off $J$ and $-1$ on it, that $J$ meet it in a final segment. Full and empty blocks satisfy both conditions, so filling such blocks in turn from the prescribed end, stopping once $c$ elements have been taken, realises every $c$ with $0\leq c\leq N$.

The remaining assertions follow from the closed form. As $N\leq g-2$, one has $N_e(\alpha_i,I_0)=N$ once $e\geq g-2$. From $N_e(\alpha_i,I_0)\leq e$, we get $\delta_{I_0,e}=0$ for $e\leq1$ and $\delta_{I_0,2}\leq1$, and from $N\leq g-2\leq1$ at $g\leq3$ the same vanishing there. Finally, enlarging $I_0$ only merges blocks, so $|B_i|+|B_{i+1}|$ weakly increases and $N_e(\alpha_i,I_0)$ weakly decreases.
\end{proof}
\begin{xremark}{6.9c}\label{rem:threshold}
\begin{enumerate}
\item For $e\geq N$ the maximum is reached at $M^c_0$ itself, admissible because it meets every block in full. For $e<N$ a proper sub-collection is needed, and such a one need not be admissible, so the segment condition is not an artefact of the construction. An inadmissible one occurs already at $g=5$: with $I_0=\{\alpha_1,\alpha_4\}$ and $\alpha_i=\alpha_2$ the blocks are $\{1,2\},\{3\},\{4,5\}$ and $M^c_0=\{e_3-e_4,\,e_3-e_5\}$, of which $T=\{e_3-e_4\}$ gives $s_T=(0,0,1,-1,0)$, increasing along the block $\{4,5\}$, while $T=\{e_3-e_5\}$ is admissible.
\item By the first case of lemma \ref{lem:threshold}, the threshold is identically zero along $I_0$, so positivity is asked only in the directions transverse to $I_0$.
\item Since $N_e(\alpha,I_0)$ weakly decreases as $I_0$ grows, for a character $\chi$ the parabolic of type $I_0=Z(\chi):=\{\alpha\in I:\langle\chi,\alpha^{\vee}\rangle=0\}$ is always the best choice.
\end{enumerate}
\end{xremark}
For $e \geq 0$, we set
\begin{equation*}
X^*(P_0)^{+,\delta_e}:=\{\chi\in X^*(P_0) \ | \
\langle\chi,\alpha^{\vee}\rangle\geq\delta_{I_0,e}(\alpha)\ \ \forall\,
\alpha\in I\setminus I_0\} \subset X^*(P_0)^+.
\end{equation*}
\begin{xcorollary}{6.9d}[dichotomy at level $e$]\label{cor:dichotomy}
Let $\chi\in X^*(P_0)^{+,\delta_e}$ and put $\lambda'=\chi+2\rho_{I_0}$. Then for every $L_0$-dominant $-s_M$ with $M\subset S$ and $|M^c|\leq\min(e,r_0)$, the weight $\nu_M:=\lambda'-s_M=\chi+s_{M^c}$ satisfies $\langle\nu_M,\alpha^{\vee}\rangle\geq-1$ for all $\alpha\in I$. Hence either $\nu_M$ is dominant, and $H^k(L/B_L,\mathcal{L}_{\nu_M})=0$ for $k>0$ by Kempf's theorem, or it is not, in which case some pairing equals $-1$ and $H^k(L/B_L,\mathcal{L}_{\nu_M})=0$ for all $k\geq0$ by lemma \ref{lem:rescue}.
\end{xcorollary}
\begin{proof}
For $\alpha\in I_0$: $\langle\chi,\alpha^{\vee}\rangle=0$ and $\langle s_{M^c},\alpha^{\vee}\rangle\geq0$ by lemma \ref{lem:threshold}. For $\alpha\in I\setminus I_0$: $\langle\nu_M,\alpha^{\vee}\rangle\geq\delta_{I_0,e}(\alpha)-N_e(\alpha,I_0)\geq-1$, the subtracted term bounding $-\langle s_{M^c},\alpha^{\vee}\rangle$ because $|M^c|\leq\min(e,r_0)$.
\end{proof}
For dominant characters $\mu,\nu$, we denote by $\mathrm{LR}(\mu,\nu)$ the set of characters $\eta$ such that $\nabla(\eta)$ occurs as a graded piece of a $\nabla$-filtration of $\nabla(\mu)\otimes\nabla(\nu)$ (see proposition \ref{prop8}). As the multiset of these graded pieces is read off the character, $\mathrm{LR}(\mu,\nu)$ is independent of the filtration, and it is the set of Littlewood--Richardson constituents of $\nabla(\mu)\otimes\nabla(\nu)$, computed in the convention of definition \ref{def3}: since $w_0w_{0,L}$ induces a ring automorphism of the representation ring, $\mathrm{LR}(\mu,\nu)$ may equally be computed with the tables in partition form, applying $w_0w_{0,L}$ at the end. It is symmetric in $\mu$ and $\nu$, and by proposition \ref{prop8}, every $\eta\in\mathrm{LR}(\mu,\nu)$ is of the form $\nu+\xi$ with $\xi$ a weight of the module with label $\mu$; in particular $\eta\leq\mu+\nu$, and $\mu+\nu$ occurs exactly once.

Multiple subsets of $X^*(P_0)$ will occur in the formulations of our degeneration results, we gather them in the following definition.
\begin{definition}
Consider an integer $0 \leq e \leq d-1$ and an index $j_0$. We denote $\mathcal{C}_{\deg,e}^{0}$, $\mathcal{C}_{\deg,e}^{1}$ and $\mathcal{C}_{\deg,e}^{2,j_0}$ the following sets of characters.
\begin{equation*}
\left\{
\begin{aligned}
\mathcal{C}_{\deg,e}^{0}\hspace{0.22cm} := \{ \lambda \in X^*(P_0) &\ | \ \lambda - 2\rho_{I_0} \in X^*(P_0)^{+,\delta_e} \}, \\
\mathcal{C}_{\deg, e}^{1} := \{ \lambda \in X^*(P_0) &\ | \ \forall i > e+1 \ \forall j \ \forall 1 \leq k \leq \min(e,r_0) \\
&\forall -s_M \in \mathcal{W}_{I_0}(r_0-k) \text{ such that } \lambda-s_M \text{ is dominant}, \\
&\forall \eta \in \mathrm{LR}(\mu_j^{d-e+k},\lambda-s_M) \ H^i(\Sh^{\tor},\nabla^{\sub}(\eta)) = 0 \}, \\
\mathcal{C}_{\deg, e}^{2,j_0} := \{ \lambda \in X^*(P_0) &\ | \ \forall i > e+1 \ \forall j \ \forall \eta \in \mathrm{LR}(\mu^{d-e}_j,\lambda-2\rho_{I_0}) - \{\mu^{d-e}_{j_0}+\lambda-2\rho_{I_0}\}\\
&H^i(\Sh^{\tor},\nabla^{\sub}(\eta)) = 0 \}.
\end{aligned}
\right.
\end{equation*}
\end{definition}
\begin{xremark}{6.10a}\label{rem:maximal}
Let $V$ be an $L$-module with a $\nabla$-filtration and let $\nabla(\mu)$ be one of its graded pieces with $\mu$ maximal among the highest weights of the graded pieces. Then $V$ admits a $\nabla$-filtration with the same graded pieces whose top quotient is $\nabla(\mu)$. Indeed, if $\nabla(\nu)$ is the graded piece just above $\nabla(\mu)$, the corresponding extension of $\nabla(\nu)$ by $\nabla(\mu)$ splits since $\Ext^1(\nabla(\nu),\nabla(\mu))\neq0$ would force $\nu>\mu$ by lemma \ref{lem3}, so the two layers may be exchanged. We apply this to the $\nabla$-filtration of $\wedge^{d-e}\Sym^2\std_L\otimes\nabla(\chi)$ for $\chi$ dominant: its graded pieces are the $\nabla(\eta)$ for $\eta\in\mathrm{LR}(\mu^{d-e}_j,\chi)$ and all $j$, and $\mu^{d-e}_{j_0}+\chi$ is maximal among them for every $j_0$. Indeed, if $\eta \in \mathrm{LR}(\mu^{d-e}_j,\chi)$ satisfies $\eta > \mu^{d-e}_{j_0}+\chi$, then $\mu^{d-e}_j + \chi \geq \eta > \mu^{d-e}_{j_0}+\chi$, which contradicts the incomparability of the $\mu^{d-e}_j$. The same argument shows that $\mu^{d-e}_{j_0}+\chi$ occurs only in $\mathrm{LR}(\mu^{d-e}_{j_0},\chi)$, and there exactly once.
\end{xremark}
\begin{lemma}\label{lem1}
Assume that $p > r_0$ or $g \leq 3$. Let $0 \leq e \leq d-1$, $\lambda \in \mathcal{C}_{\deg,e}^{0}$ and $\mathcal{F}$ be a coherent sheaf on $\Sh^{\tor}$. For all $r_0-\min(e,r_0) \leq i \leq r_{0}$ and $n \geq 0$, we have the following isomorphism
\begin{equation*}
 H^{n}(Y^{\tor}_{I_0},\pi^*\mathcal{F} \otimes \Omega^i_{Y^{\tor}_{I_0}/\Sh^{\tor}} \otimes \mathcal{L}^{\sub}_{\lambda}) = H^{n}(\Sh^{\tor},\mathcal{F} \otimes \pi_*(\Omega^i_{Y^{\tor}_{I_0}/\Sh^{\tor}} \otimes \mathcal{L}^{\sub}_{\lambda})).
\end{equation*}
Moreover, the vector bundle $\pi_*(\Omega^i_{Y^{\tor}_{I_0}/\Sh^{\tor}} \otimes \mathcal{L}_{\lambda})$ admits a filtration whose graded pieces are the automorphic bundles $\nabla(\lambda-s_M)$ where $-s_M$ runs over the elements of $\mathcal{W}_{I_0}(i)$ such that $\lambda-s_M$ is dominant.
\end{lemma}
\begin{proof}
Let $r_0-\min(e,r_0) \leq i \leq r_0$. By proposition \ref{prop9}, the vector bundle $\Omega^{i}_{Y^{\tor}_{I_0}/\Sh^{\tor}} \otimes \mathcal{L}_{\lambda}$ admits a filtration with graded pieces
\begin{equation*}
\mathcal{L}(\nabla_{L_0}(-s_M)) \otimes \mathcal{L}_{\lambda} = \mathcal{L}(\nabla_{L_0}(\nu_M)), \qquad \nu_M := \lambda - s_M, \qquad -s_M \in \mathcal{W}_{I_0}(i).
\end{equation*}
Since $\lambda$ is a character of $P_0$ and $-s_M$ is $L_0$-dominant, the weight $\nu_M$ is $L_0$-dominant. By proposition \ref{prop13}, $R^k\pi_*\mathcal{L}(\nabla_{L_0}(\nu_M))$ is the vector bundle associated to the $P$-module $H^k(P/P_0,\mathcal{L}(\nabla_{L_0}(\nu_M)))$. Since $P/P_0 = L/(P_0 \cap L)$ and $\nabla_{L_0}(\nu_M) = \Ind_{B_L}^{P_0 \cap L}\nu_M$ with vanishing higher derived functors by Kempf's theorem, the transitivity of induction gives
\begin{equation*}
H^k(P/P_0,\mathcal{L}(\nabla_{L_0}(\nu_M))) = R^k\Ind_{P_0\cap L}^{L}\nabla_{L_0}(\nu_M) = R^k\Ind_{B_L}^{L}\nu_M = H^k(L/B_L,\mathcal{L}_{\nu_M}).
\end{equation*}
As $|M^c| = r_0-i \leq \min(e,r_0)$ and $\lambda-2\rho_{I_0} \in X^*(P_0)^{+,\delta_e}$, corollary \ref{cor:dichotomy} shows that this vanishes for all $k > 0$, and that for $k=0$ it is $H^0(L/B_L,\mathcal{L}_{\nu_M})$ if $\nu_M$ is dominant and $0$ otherwise. In other words, we have
\begin{equation*}
\forall k > 0 \ R^k\pi_*\mathcal{L}(\nabla_{L_0}(\nu_M)) = 0 \quad \text{and} \quad \pi_*\mathcal{L}(\nabla_{L_0}(\nu_M)) = \begin{cases} \nabla(\nu_M) & \text{if } \nu_M \text{ is dominant,} \\ 0 & \text{otherwise.}\end{cases}
\end{equation*}
By lemma \ref{prop10} we deduce
\begin{equation*}
\forall k > 0 \ R^k\pi_*(\Omega^i_{Y^{\tor}_{I_0}/\Sh^{\tor}} \otimes \mathcal{L}_{\lambda}) = 0,
\end{equation*}
and since $R^1\pi_*$ vanishes on every graded piece, the functor $\pi_*$ is exact on the filtration, which gives the announced filtration of $\pi_*(\Omega^i_{Y^{\tor}_{I_0}/\Sh^{\tor}} \otimes \mathcal{L}_{\lambda})$. In particular, this sheaf is locally free. Since $\pi^*\mathcal{O}_{\Sh^{\tor}}(-D_{\red}) = \mathcal{O}_{Y^{\tor}_{I_0}}(-D_{\red})$, the projection formula implies
\begin{equation*}
\forall k > 0 \ R^k\pi_*(\Omega^i_{Y^{\tor}_{I_0}/\Sh^{\tor}} \otimes \mathcal{L}_{\lambda}(-D_{\red})) = 0.
\end{equation*}
Using again the projection formula, it implies that the Leray spectral sequence
\begin{multline*}
E_2^{t,k} = H^t(\Sh^{\tor},R^k\pi_*(\pi^*\mathcal{F} \otimes \Omega^i_{Y^{\tor}_{I_0}/\Sh^{\tor}} \otimes \mathcal{L}^{\sub}_{\lambda})) \\ \Rightarrow H^{t+k}(Y^{\tor}_{I_0},\pi^*\mathcal{F} \otimes \Omega^i_{Y^{\tor}_{I_0}/\Sh^{\tor}} \otimes \mathcal{L}^{\sub}_{\lambda})
\end{multline*}
is concentrated on one row and we get the desired isomorphisms.
\end{proof}
\begin{proposition}\label{prop12}
Assume that $p > d = g(g+1)/2$ and $p > r_0$. Let $0 \leq e \leq d-1$ and $i>e$ be integers and let $j_0$ be an index. For any character $$\lambda \in \mathcal{C}_{\deg,e}^{0} \cap \mathcal{C}_{\deg,e}^{1} \cap \mathcal{C}_{\deg,e}^{2,j_0},$$the vanishing
\begin{equation*}
H^i(Y^{\tor}_{I_0},\Omega^{d_{0}-e}_{Y^{\tor}_{I_0}}(\log D_{\red}) \otimes \mathcal{L}^{\sub}_{\lambda}) = 0
\end{equation*}
implies the vanishing
\begin{equation*}
H^i(\Sh^{\tor},\nabla^{\sub}({\mu_{j_0}^{d-e} + \lambda-2\rho_{I_0}})) = 0.
\end{equation*}
\end{proposition}
\begin{proof}
The nonzero graded pieces of the filtration of definition \ref{def:filt} on $\Omega^{d_{0}-e}_{Y^{\tor}_{I_0}}(\log D_{\red})$ are
\begin{equation*}
\pi^*\Omega^{d-e+k}_{\Sh^{\tor}}(\log D_{\red}) \otimes \Omega^{r_{0}-k}_{Y^{\tor}_{I_0}/\Sh^{\tor}}, \qquad 0 \leq k \leq \min(e,r_0),
\end{equation*}
the piece $k = 0$ being the top one. We first show that
\begin{equation*}
\forall 1 \leq k \leq \min(e,r_0) \ H^{i+1}(Y^{\tor}_{I_0}, \pi^*\Omega^{d-e+k}_{\Sh^{\tor}}(\log D_{\red}) \otimes \Omega^{r_0-k}_{Y^{\tor}_{I_0}/\Sh^{\tor}} \otimes \mathcal{L}^{\sub}_{\lambda}) = 0.
\end{equation*}
Since $\lambda \in \mathcal{C}_{\deg,e}^{0}$, we know by lemma \ref{lem1} that this group is
\begin{equation}\label{eq2}
H^{i+1}(\Sh^{\tor}, \Omega^{d-e+k}_{\Sh^{\tor}}(\log D_{\red}) \otimes \pi_*(\Omega^{r_0-k}_{Y^{\tor}_{I_0}/\Sh^{\tor}} \otimes \mathcal{L}^{\sub}_{\lambda})),
\end{equation}
and that $\pi_*(\Omega^{r_0-k}_{Y^{\tor}_{I_0}/\Sh^{\tor}} \otimes \mathcal{L}^{\sub}_{\lambda})$ admits a filtration with graded pieces $\nabla^{\sub}(\lambda-s_M)$, where $-s_M \in \mathcal{W}_{I_0}(r_0-k)$ and $\lambda-s_M$ is dominant. By lemma \ref{prop10}, it is enough to show the vanishing of $H^{i+1}(\Sh^{\tor},\Omega^{d-e+k}_{\Sh^{\tor}}(\log D_{\red}) \otimes \nabla^{\sub}({\lambda-s_M}))$ for each of these pieces. We use proposition \ref{prop18} and proposition \ref{prop8} to see that the bundle
\begin{equation*}
 \Omega^{d-e+k}_{\Sh^{\tor}}(\log D_{\red}) \otimes \nabla^{\sub}({\lambda-s_M})
\end{equation*}
admits a filtration where the graded pieces are the $\nabla^{\sub}(\eta)$ with $\eta \in \mathrm{LR}(\mu_j^{d-e+k},\lambda-s_M)$ for all $j$. By lemma \ref{prop10}, the vanishing of \eqref{eq2} follows from $\lambda \in \mathcal{C}_{\deg,e}^{1}$, as $i+1 > e+1$. The two filtrations used here are applied one after the other, so they need not be compatible. Since
\begin{equation*}
H^i(Y^{\tor}_{I_0},\Omega^{d_{0}-e}_{Y^{\tor}_{I_0}}(\log D_{\red}) \otimes \mathcal{L}^{\sub}_{\lambda}) = 0
\end{equation*}
by hypothesis, we can apply lemma \ref{prop11} to the filtration of definition \ref{def:filt} to deduce that the top piece satisfies
\begin{equation*}
H^i(Y^{\tor}_{I_0},\pi^*\Omega^{d-e}_{\Sh^{\tor}}(\log D_{\red}) \otimes \Omega^{r_0}_{Y^{\tor}_{I_0}/\Sh^{\tor}} \otimes \mathcal{L}^{\sub}_{\lambda}) = 0.
\end{equation*}
As $\Omega^{r_0}_{Y^{\tor}_{I_0}/\Sh^{\tor}} = \mathcal{L}_{-2\rho_{I_0}}$ and $\lambda-2\rho_{I_0}$ is dominant, lemma \ref{lem1} with $i = r_0$ gives
\begin{equation*}
H^i(\Sh^{\tor},\Omega^{d-e}_{\Sh^{\tor}}(\log D_{\red}) \otimes \nabla^{\sub}({\lambda-2\rho_{I_0}})) = 0.
\end{equation*}
Combining again the propositions \ref{prop8} and \ref{prop18}, we know that
\begin{equation*}
\Omega^{d-e}_{\Sh^{\tor}}(\log D_{\red}) \otimes \nabla^{\sub}({\lambda-2\rho_{I_0}})
\end{equation*}
admits a $\nabla$-filtration whose graded pieces are the $\nabla^{\sub}(\eta)$ with $\eta \in \mathrm{LR}(\mu^{d-e}_j,\lambda-2\rho_{I_0})$ for all $j$, and by remark \ref{rem:maximal} we may choose it so that its top graded piece is $\nabla^{\sub}(\mu_{j_0}^{d-e}+ \lambda-2\rho_{I_0})$, the other graded pieces being the $\nabla^{\sub}(\eta)$ appearing in the definition of $\mathcal{C}_{\deg,e}^{2,j_0}$. Since $\lambda \in \mathcal{C}_{\deg,e}^{2,j_0}$, we use again lemma \ref{prop11} for this $\nabla$-filtration to see that
\begin{equation*}
H^i(\Sh^{\tor},\nabla^{\sub}({\mu_{j_0}^{d-e}+ \lambda-2\rho_{I_0}})) = 0.
\end{equation*}
\end{proof}
\subsubsection{Propagation}\label{def_g}
In this section, we construct a non-decreasing function on the power set of characters that gives new vanishing results from known ones. Our main result is theorem \ref{th1}. 
\begin{definition}
For all $k \geq 0$,  we define a subset $\mathcal{C}^k_{\van}$ of $X^*$ as
\begin{equation*}
\mathcal{C}_{\van}^k = \left\{ \lambda \in X^* \ | \ \forall i > k \ H^i(\Sh^{\tor},\nabla^{\sub}({\lambda})) = 0 \right\}.
\end{equation*}
\end{definition}
\begin{rmrk}
$\mathcal{C}_{\van}^k$ always contains the non-dominant characters.
\end{rmrk}
\begin{definition}\label{def:615}
Let $0 \leq e \leq d-1$. We define a function $g_{I_0,e} : \mathcal{P}(X^*) \rightarrow \mathcal{P}(X^*)$ as follows. Let $\mathcal{C} \subset X^*$ and let $j_0$ be an index. A character $\lambda\in X^*$ lies in $g^{j_0}_{I_0,e}(\mathcal{C})$ if and only if the following hold:
\begin{enumerate}
\item[(D)] \emph{Dominance.} $\chi:=\lambda-\mu^{d-e}_{j_0}$ lies in $X^*(P_0)^{+,\delta_e}$;
\item[(A)] \emph{Ampleness.} $\lambda':=\chi+2\rho_{I_0}$ lies in $\mathcal{C}_{\amp,I_0}$;
\item[(C1)] \emph{First degeneration.} For all $k$ with $1\leq k\leq\min(e,r_0)$, all $-s_M\in\mathcal{W}_{I_0}(r_0-k)$ such that $\lambda'-s_M$ is dominant, and all $j$,
\begin{equation*}
\mathrm{LR}\big(\mu^{d-e+k}_j,\ \lambda'-s_M\big)\subset\mathcal{C};
\end{equation*}
\item[(C2)] \emph{Second degeneration.} For all $j$, $\mathrm{LR}\big(\mu^{d-e}_j,\ \chi\big)-\{\lambda\}\subset\mathcal{C}$.
\end{enumerate}
We set
\begin{equation*}
g_{I_0,e}(\mathcal{C}) = \bigcup_{j_0} g^{j_0}_{I_0,e}(\mathcal{C}).
\end{equation*}
\end{definition}
By lemma \ref{lem:threshold}, condition (D) only constrains $\chi$ in the directions transverse to $I_0$. It reduces to the dominance of $\chi \in X^*(P_0)$ in many cases.
\begin{xcorollary}{6.15a}\label{cor:vacuous}
\begin{enumerate}
\item For $e\leq1$ and every $I_0$, the threshold vanishes, so (D) is the condition $\chi \in X^*(P_0)^+$.
\item For every dominant $\chi$ and every $e\leq2$, condition (D) holds at $I_0=Z(\chi)$.
\item For $g\leq3$, condition (D) holds at every $I_0$, for every $\chi\in X^*(P_0)^+$ and every $e$. For $g\leq4$ it holds at $I_0=Z(\chi)$, again for every dominant $\chi$ and every $e$.
\end{enumerate}
\end{xcorollary}
\begin{proof}
Part (1) and the first assertion of (3) follow from the last clause of lemma \ref{lem:threshold}: where $\delta_{I_0,e}$ vanishes, condition (D) is the dominance of $\chi$. For (2), by that lemma $\delta_{I_0,2}\leq1$, while $\langle\chi,\alpha^{\vee}\rangle\geq1$ for every $\alpha\notin Z(\chi)$ by definition of $Z(\chi)$, and $\chi\in X^*(P_0)$ for $I_0=Z(\chi)$. For the rest of (3): at $g\leq4$ one has $N_e(\alpha,I_0)\leq g-2\leq2$, so $\delta_{I_0,e}\leq1$ and the argument of (2) applies at every $e$.
\end{proof}
\begin{theorem}\label{th1}
Assume that $p > d_0$. Let $\mathcal{C}$ be a subset of $\mathcal{C}_{\van}^{e+1}$. Then, we have
\begin{equation*}
g_{I_0,e}(\mathcal{C}) \subset \mathcal{C}_{\van}^e.
\end{equation*}
In other words, if we have a set $\mathcal{C}$ of characters $\lambda$ for which the cohomology 
\begin{equation*}
H^i(\Sh^{\tor},\nabla^{\sub}({\lambda}))
\end{equation*}
is concentrated in degrees $[0,e+1]$, then the image of $\mathcal{C}$ by the function $g_{I_0,e}$ is a set of characters $\lambda$ for which the cohomology $H^i(\Sh^{\tor},\nabla^{\sub}({\lambda}))$ is concentrated in degrees $[0,e]$.
\end{theorem}
\begin{proof}
Since $g_{I_0,e}$ is non-decreasing, it suffices to show $g_{I_0,e}(\mathcal{C}_{\van}^{e+1}) \subset \mathcal{C}_{\van}^{e}$. Let $\lambda \in g_{I_0,e}(\mathcal{C}_{\van}^{e+1})$ be a character, choose an index $j_0$ such that $\lambda \in g^{j_0}_{I_0,e}(\mathcal{C}_{\van}^{e+1})$ and define $\chi := \lambda - \mu^{d-e}_{j_0}$ and $\lambda^{\prime} := \chi + 2\rho_{I_0}$. Condition (D) says that $\lambda^\prime \in \mathcal{C}_{\deg,e}^{0}$. Since $\mathcal{C}_{\van}^{e+1}$ is the set of characters whose cohomology vanishes in degrees $>e+1$, conditions (C1) and (C2) say exactly that $\lambda^\prime \in \mathcal{C}_{\deg,e}^{1}$ and $\lambda^\prime \in \mathcal{C}_{\deg,e}^{2,j_0}$, as $\mu^{d-e}_{j_0} + \lambda^\prime - 2\rho_{I_0} = \lambda$. Condition (A) says that
\begin{equation*}
\lambda^\prime \in \mathcal{C}_{\amp, I_0}.
\end{equation*}
The triple $(Y_{I_0}^{\tor},D_{\red},\mathcal{L}_{\lambda^\prime})$ lifts to $\mathbb{Z}/p^2\mathbb{Z}$: the lift is the reduction of the models over $\mathbb{Z}_{(p)}$ of the objects of definition \ref{prop5}, the flag bundle and $\mathcal{L}_{\lambda^\prime}$ being built from them as over $\mathbb{F}_p$. Since $p > d_0$, we apply proposition \ref{prop7} to see that
\begin{equation*}
H^i(Y^{\tor}_{I_0},\Omega^{d_0-e}_{Y_{I_0}^{\tor}}(\log D_{\red}) \otimes \mathcal{L}^{\sub}_{\lambda^\prime}) = 0
\end{equation*}
for all $i+d_0-e > d_0$ (i.e. $i>e)$ and we use proposition \ref{prop12} (as $p > d_0 \geq \max(d,r_0)$) to see that
\begin{equation*}
H^i(\Sh^{\tor},\nabla^{\sub}({\mu_{j_0}^{d-e} + \lambda^\prime-2\rho_{I_0}})) = H^i(\Sh^{\tor},\nabla^{\sub}({\lambda})) = 0
\end{equation*}
for all $i > e$.
\end{proof}
\begin{rmrk}
The theorem is still valid if we use a subset $\mathcal{C}_{\amp,I_0}^{\prime}$ of $\mathcal{C}_{\amp,I_0}$ instead of $\mathcal{C}_{\amp,I_0}$ in the definition of $g_{I_0,e}$. In particular, by theorem \ref{th_amp_auto}, we can use it with the subset of orbitally $p$-close and $\mathcal{Z}_0$-ample characters.
\end{rmrk}
\subsection{The Siegel threefold case}
In this subsection, we give more details in the case $g = 2$ because we believe it already contains some of the idea of the general method and it requires less notations. Assume that $p$ is a prime larger than $g^2 = 4$. The Siegel threefold $\Sh^{\tor}$ is projective variety of dimension $d = 3$ over $\mathbb{F}_p$. From the Kodaira-Spencer isomorphism, we have an identification
\begin{equation*}
\Omega^1_{\Sh^{\tor}}(\log D_{\red}) = \nabla(0,-2).
\end{equation*} 
From proposition \ref{prop18}, we know that any exterior power of $\Sym^2 \std_{\GL_2}$ admits a $\nabla$-filtration. It direclty implies that we have
\begin{equation*}
\left\{
\begin{aligned}
\Omega^2_{\Sh^{\tor}}(\log D_{\red}) &= \nabla(-1,-3), \\
\Omega^3_{\Sh^{\tor}}(\log D_{\red}) &= \nabla(-3,-3),
\end{aligned}
\right.
\end{equation*}
and the weights of these three automorphic vector bundles are 
\begin{equation*}
\left\{
\begin{aligned}
(\xi^1_j)_j &= \left\{(-2,0),(-1,-1),(0,-2)\right\}, \\
(\xi^2_j)_j &= \left\{(-3,-1),(-2,-2),(-1,-3)\right\}, \\
(\xi^3_j)_j &= \left\{(-3,-3)\right\}.
\end{aligned}
\right.
\end{equation*}
We start with the case $I_0 = \emptyset$. The associated complete flag bundle $\pi : Y^{\tor} \rightarrow \Sh^{\tor}$ parametrizes quotient line bundles of the rank $2$ Hodge bundle $\Omega^{\tor}$. It is a $\mathbb{P}^1$-fibration and we have an identification:
\begin{equation*}
\Omega^1_{Y^{\tor}/\Sh^{\tor}} = \mathcal{L}_{-2\rho} = \mathcal{L}_{(-1,1)}.
\end{equation*}
For any integer $0 \leq e \leq d-1 = 2$, we have an increasing filtration on the bundle $\Omega^{4-e}_{Y^{\tor}}(\log D_{\red})$ given by
\begin{equation*}
F_k = \pi^*\Omega^{4-e-k}_{\Sh^{\tor}} (\log D_{\red}) \wedge \Omega^k_{Y^{\tor}} (\log D_{\red})
\end{equation*}
with graded pieces
\begin{equation*}
\gr_k = \pi^*\Omega^{4-e-k}_{\Sh^{\tor}} (\log D_{\red})\otimes \Omega^k_{Y^{\tor}/\Sh^{\tor}}.
\end{equation*}
For any character $\lambda = (k_1 \geq k_2)$, we have an associated spectral sequence:
\begin{multline*}
E_{2,e,\lambda}^{t,k} = H^{t+k}(Y^{\tor}, \pi^*\Omega^{4-e-k}_{\Sh^{\tor}} (\log D_{\red})\otimes \Omega^k_{Y^{\tor}/\Sh^{\tor}} \otimes \mathcal{L}_{\lambda}^{\sub}) \\ \Rightarrow H^{t+k}(Y^{\tor},\Omega^{4-e}_{Y^{\tor}}(\log D_{\red}) \otimes \mathcal{L}_{\lambda}^{\sub})
\end{multline*}
starting at page $2$. We will study this spectral sequence for each $e$, starting with $e = 0$. In this case (see the corresponding figure), the second page of the spectral sequence is concentrated in one row
\begin{figure}
\centering
\vspace{3.5cm}
\begin{tikzpicture}[transform canvas={scale=0.8}]
  \matrix (m) [matrix of math nodes, nodes in empty cells,nodes={minimum width=5ex, minimum height=5.5ex,outer sep=-5pt}, column sep=1ex,row sep=1ex]{
           k  &   &   &   &   & \\
          2 & 0 &  0 & 0 & 0  & 0\\
          1 &  H^0(\gr_1 \otimes \mathcal{L}_{\lambda}^{\sub})  &  H^1(\gr_1 \otimes \mathcal{L}_{\lambda}^{\sub}) & H^2(\gr_1 \otimes \mathcal{L}_{\lambda}^{\sub}) &  H^3(\gr_1 \otimes \mathcal{L}_{\lambda}^{\sub}) & 0 \\
          0 &  0 &  0  & 0 & 0 & 0 \\
          -1 \ \ &  0 & 0 & 0 & 0 & 0 \\
    \quad\strut &   -1  &  0  &  1  &  2 & 3 & t \strut  \\};
    
  	\draw[-stealth] (m-2-2.south east) -- (m-3-4.north west);
    \draw[-stealth] (m-2-3.south east) -- (m-3-5.north west);
	\draw[-stealth] (m-3-2.south east) -- (m-4-4.north west);
    \draw[-stealth] (m-3-3.south east) -- (m-4-5.north west);
    \draw[-stealth] (m-4-2.south east) -- (m-5-4.north west);
    \draw[-stealth] (m-4-3.south east) -- (m-5-5.north west);

\draw[thick,-stealth] (m-6-1.east) -- (m-1-1.east) ;
\draw[thick,-stealth] (m-6-1.north) -- (m-6-6.north) ;
\end{tikzpicture}
\vspace{2cm}
\caption{$E_2$-page of the spectral sequence when $e = 0$}
\end{figure}
as the only graded piece is $\gr_1 = \pi^*\Omega^3_{\Sh^{\tor}}(\log D_{\red}) \otimes \Omega^1_{Y^{\tor}/\Sh^{\tor}}$. In particular, the spectral sequence degenerates at page 2 and we get
\begin{equation*}
H^i(Y^{\tor},  \pi^*\Omega^3_{\Sh^{\tor}}(\log D_{\red}) \otimes \Omega^1_{Y^{\tor}/\Sh^{\tor}} \otimes \mathcal{L}_{\lambda}^{\sub}) = H^i(Y^{\tor}, \Omega^4_{Y^{\tor}}(\log D_{\red}) \otimes \mathcal{L}_{\lambda}^{\sub})
\end{equation*}
for all $i\geq 0$. Moreover, if we assume that $\lambda - 2\rho$ is dominant (which is equivalent to $k_1 \geq k_2 +2$), we get
\begin{equation*}
\begin{aligned}
&H^i(Y^{\tor},  \pi^*\Omega^3_{\Sh^{\tor}}(\log D_{\red}) \otimes \Omega^1_{Y^{\tor}/\Sh^{\tor}} \otimes \mathcal{L}_{\lambda}^{\sub}) \\
&= H^i(\Sh^{\tor},  \nabla(-3,-3) \otimes \nabla^{\sub}(k_1-1,k_2+1)) \\
&= H^i(\Sh^{\tor}, \nabla^{\sub}(k_1-4,k_2-2))
\end{aligned}
\end{equation*}
for all $i\geq 0$. We assume that $\mathcal{L}_{(k_1,k_2)}$ is $D$-ample on $Y^{\tor}$ and we use the logarithmic Kodaira-Nakano vanishing theorem to see that
\begin{equation*}
H^i(Y^{\tor}, \Omega^4_{Y^{\tor}}(\log D_{\red}) \otimes \mathcal{L}^{\sub}_{\lambda}) = 0
\end{equation*}
for all $i>0$. 
We summarize this discussion by saying that we have
\begin{equation*}
H^i(\Sh^{\tor},\nabla^{\sub}(k_1-4,k_2-2)) = 0
\end{equation*}
for all $i > 0$ and $(k_1,k_2)$ such that
\begin{itemize}
\item $k_1 \geq k_2 + 2$,
\item $(k_1,k_2) \in \mathcal{C}_{\amp,\emptyset}$.
\end{itemize}
Now, we consider the spectral sequence in the case $e = 1$ (see the corresponding figure).
\begin{figure}
\centering
\vspace{3.5cm}
\begin{tikzpicture}[transform canvas={scale=0.8}]
  \matrix (m) [matrix of math nodes, nodes in empty cells,nodes={minimum width=5ex, minimum height=5.5ex,outer sep=-5pt}, column sep=1ex,row sep=1ex]{
           k  &   &   &   &   & & \\
          2 & 0 &  0 & 0 & 0  & 0 & \\
          1 &  H^0(\gr_1 \otimes \mathcal{L}_{\lambda}^{\sub})  &  H^1(\gr_1 \otimes \mathcal{L}_{\lambda}^{\sub}) & H^2(\gr_1 \otimes \mathcal{L}_{\lambda}^{\sub}) &  H^3(\gr_1 \otimes \mathcal{L}_{\lambda}^{\sub}) & H^4(\gr_1 \otimes \mathcal{L}_{\lambda}^{\sub}) \\
          0 &  0 &  H^0(\gr_0 \otimes \mathcal{L}_{\lambda}^{\sub})   & H^1(\gr_0 \otimes \mathcal{L}_{\lambda}^{\sub}) & H^2(\gr_0 \otimes \mathcal{L}_{\lambda}^{\sub}) & H^3(\gr_0 \otimes \mathcal{L}_{\lambda}^{\sub}) \\
          -1 \ \ &  0 & 0 & 0 & 0 & 0 & \\
    \quad\strut &   -1  &  0  &  1  &  2 & 3 & t \strut  \\};
    
  	\draw[-stealth] (m-2-2.south east) -- (m-3-4.north west);
  	  	\draw[-stealth] (m-3-4.south east) -- (m-4-6.north west);

	\draw[-stealth] (m-3-2.south east) -- (m-4-4.north west);
    \draw[-stealth] (m-3-3.south east) -- (m-4-5.north west);
    \draw[-stealth] (m-4-3.south east) -- (m-5-5.north west);

\draw[thick,-stealth] (m-6-1.east) -- (m-1-1.east) ;
\draw[thick,-stealth] (m-6-1.north) -- (m-6-7.north) ;
\end{tikzpicture}
\vspace{2cm}
\caption{$E_2$-page of the spectral sequence when $e = 1$}
\end{figure}
The graded pieces are $\gr_1 = \pi^*\Omega^2_{\Sh^{\tor}}(\log D_{\red}) \otimes \Omega^1_{Y^{\tor}/\Sh^{\tor}}$ and $\gr_0 = \pi^*\Omega^3_{\Sh^{\tor}}(\log D_{\red})$. The limit is $H^{i}(Y^{\tor},\Omega^3_{Y^{\tor}}(\log D_{\red}) \otimes \mathcal{L}_{\lambda}^{\sub})$ and by the logarithmic Kodaira-Nakano theorem, it vanishes for all $i>1$ when $(k_1,k_2) \in \mathcal{C}_{\amp,\emptyset}$. The critical differential is
\begin{multline*}
d^{1,1} : H^2(Y^{\tor},\pi^*\Omega^2_{\Sh^{\tor}}(\log D_{\red}) \otimes \Omega^1_{Y^{\tor}/\Sh^{\tor}} \otimes \mathcal{L}_{\lambda}^{\sub})\\ \rightarrow H^3(Y^{\tor}, \pi^*\Omega^3_{\Sh^{\tor}}(\log D_{\red}) \otimes \mathcal{L}_{\lambda}^{\sub}),
\end{multline*}
because when $d^{1,1} = 0$, we have $E_2^{1,1} = E_{\infty}^{1,1}$ and $E_2^{2,1} = E_{\infty}^{2,1}$. Under the additional hypothesis $k_1 \geq k_2+2$, we deduce that
\begin{multline*}
H^i(Y^{\tor},\pi^*\Omega^2_{\Sh^{\tor}}(\log D_{\red}) \otimes \Omega^1_{Y^{\tor}/\Sh^{\tor}} \otimes \mathcal{L}_{\lambda}^{\sub}) \\ = H^i(\Sh^{\tor},\nabla(-1,-3) \otimes \nabla^{\sub}(k_1-1,k_2+1)) = 0 
\end{multline*}
for all $i > 1$ and $(k_1,k_2) \in \mathcal{C}_{\amp,\emptyset}$. Consider an integer $i = 2$ or $3$. The tensor product of automorphic vector bundles 
\begin{equation*}
\nabla(-1,-3) \otimes \nabla^{\sub}(k_1-1,k_2+1)
\end{equation*}
admits a filtration whose graded pieces are among the automorphic bundles
\begin{equation*}
\nabla^{\sub}(\xi_j^2 + (k_1-1,k_2+1))
\end{equation*}
where $j = 1, 2, 3$ and if we ask for the vanishing
\begin{equation*}
H^{i+1}(\Sh^{\tor},\nabla^{\sub}(\xi^2_j+(k_1-1,k_2+1))) = 0
\end{equation*}
for $j = 1,2$, it will imply
\begin{equation*}
H^i(\Sh^{\tor},\nabla^{\sub}((-1,-3)+(k_1-1,k_2+1))) = H^i(\Sh^{\tor},\nabla^{\sub}(k_1-2,k_2-2)) = 0.
\end{equation*}
To see that the critical differential $d^{1,1}$ is zero, it is sufficient to have 
\begin{equation*}
H^3(Y^{\tor}, \pi^*\Omega^3_{\Sh^{\tor}}(\log D_{\red}) \otimes \mathcal{L}_{\lambda}^{\sub}) = H^3(\Sh^{\tor}, \nabla^{\sub}(k_1-3,k_2-3)) = 0.
\end{equation*}
We summarize this discussion by saying that we have
\begin{equation*}
H^i(\Sh^{\tor},\nabla^{\sub}(k_1-2,k_2-2)) = 0
\end{equation*}
for all $i > 1$ and $(k_1,k_2)$ such that
\begin{itemize}
\item $k_1 \geq k_2 + 2$,
\item $(k_1,k_2) \in \mathcal{C}_{\amp,\emptyset}$,
\item $H^3(\Sh^{\tor},\nabla^{\sub}(\xi^2_j+(k_1-1,k_2+1))) = 0$ for $j = 1,2$,
\item $H^3(\Sh^{\tor}, \nabla^{\sub}(k_1-3,k_2-3)) = 0$.
\end{itemize}

Now, we consider the spectral sequence in the case $e = 2$. The graded pieces are\\ $\gr_1 = \pi^*\Omega^1_{\Sh^{\tor}}(\log D_{\red}) \otimes \Omega^1_{Y^{\tor}/\Sh^{\tor}}$ and $\gr_0 = \pi^*\Omega^2_{\Sh^{\tor}}(\log D_{\red})$. The limit is
\begin{equation*}
H^{i}(Y^{\tor},\Omega^2_{Y^{\tor}}(\log D_{\red}) \otimes \mathcal{L}_{\lambda}^{\sub})
\end{equation*}
and by the logarithmic Kodaira-Nakano theorem, it vanishes for all $i>2$ when\\ $(k_1,k_2) \in \mathcal{C}_{\amp,\emptyset}$. The differential
\begin{multline*}
d^{2,1} : H^3(Y^{\tor},\pi^*\Omega^1_{\Sh^{\tor}}(\log D_{\red}) \otimes \Omega^1_{Y^{\tor}/\Sh^{\tor}} \otimes \mathcal{L}_{\lambda}^{\sub})\\ \rightarrow \underbrace{H^4(Y^{\tor}, \pi^*\Omega^2_{\Sh^{\tor}}(\log D_{\red}) \otimes \mathcal{L}_{\lambda}^{\sub})}_{=0}
\end{multline*}
is already $0$, since $R^1\pi_*\mathcal{L}_{\lambda} = 0$ and the Siegel threefold has dimension $3$. Under the additional hypothesis $k_1 \geq k_2+2$, we deduce that
\begin{multline*}
H^i(Y^{\tor},\pi^*\Omega^1_{\Sh^{\tor}}(\log D_{\red}) \otimes \Omega^1_{Y^{\tor}/\Sh^{\tor}} \otimes \mathcal{L}_{\lambda}^{\sub}) \\ = H^i(\Sh^{\tor},\nabla(0,-2) \otimes \nabla^{\sub}(k_1-1,k_2+1)) = 0 
\end{multline*}
for all $i > 2$ and $(k_1,k_2) \in \mathcal{C}_{\amp,\emptyset}$. The tensor product of automorphic vector bundles 
\begin{equation*}
\nabla(0,-2) \otimes \nabla^{\sub}(k_1-1,k_2+1)
\end{equation*}
admits a filtration whose graded pieces are among the automorphic bundles
\begin{equation*}
\nabla^{\sub}(\xi_j^1 + (k_1-1,k_2+1))
\end{equation*}
where $j = 1, 2, 3$ and since the vanishing
\begin{equation*}
H^{i+1}(\Sh^{\tor},\nabla^{\sub}(\xi^1_j+(k_1-1,k_2+1))) = 0
\end{equation*}
for $j = 1,2$ is automatic, it implies the vanishing of
\begin{equation*}
H^i(\Sh^{\tor},\nabla^{\sub}((0,-2)+(k_1-1,k_2+1)))=H^i(\Sh^{\tor},\nabla^{\sub}(k_1-1,k_2-1)).
\end{equation*}
We summarize this discussion by saying that we have
\begin{equation*}
H^i(\Sh^{\tor},\nabla^{\sub}(k_1-1,k_2-1)) = 0
\end{equation*}
for all $i > 2$ and $(k_1,k_2)$ such that
\begin{itemize}
\item $k_1 \geq k_2 + 2$,
\item $(k_1,k_2) \in \mathcal{C}_{\amp,\emptyset}$.
\end{itemize}
We consider the case $I_0 = I = \left\{(1,-1)\right\}$. This case corresponds to $Y^{\tor}_{I_0} = \Sh^{\tor}$. In this degenerate case, the spectral sequence is trivial and the $D$-ample automorphic line bundle are powers of the determinant of the Hodge bundle : $\nabla(k,k)$ for all $k < 0$. By the logarithmic Kodaira-Nakano vanishing theorem, we have
\begin{equation*}
H^i(\Sh^{\tor}, \Omega^j_{\Sh^{\tor}}(\log D_{\red}) \otimes \nabla^{\sub}(k,k)) = 0
\end{equation*}
for all $i +j> 3$ and $k < 0$. In the case $j=3$ , we get 
\begin{equation*}
H^i(\Sh^{\tor}, \nabla^{\sub}(k-3,k-3)) = 0
\end{equation*}
for all $i>0$. In the case $j=2$ , we get 
\begin{equation*}
H^i(\Sh^{\tor}, \nabla^{\sub}(k-1,k-3)) = 0
\end{equation*}
for all $i>1$. In the case $j=1$ , we get 
\begin{equation*}
H^i(\Sh^{\tor}, \nabla^{\sub}(k,k-2)) = 0
\end{equation*}
for all $i>2$.

%% file: part7.tex
\section{Degeneration algorithm}\label{part7}
\subsection{Presentation}
From the description of the degeneration of the different spectral sequences in the case of the Siegel threefold, it is clear that an algorithm implemented on a computer could be useful to make the vanishing results more explicit. We present an algorithm written in SageMath that uses our main result (theorem \ref{th1}) to compute new vanishing results from known ones.\\ \\
See \href{https://github.com/ThibaultAlexandre/vanishing-results-over-the-siegel-variety}{github.com/ThibaultAlexandre/vanishing-results-over-the-siegel-variety} \\ \\
It depends on the following parameters.
\begin{enumerate}
\item The genus $g \geq 2$ (the case $g = 1$ is obvious).
\item A prime $p$ such that $p > g^2$.
\item A set of known vanishing result $\mathcal{C}_{\van}$ for each cohomological degree.
\item The integer $e$ that appears on the spectral sequence \ref{sp1}.
\item A subset $I_0 \subset I$ for the choice of the flag bundle $Y^{\tor}_{I_0}$ over the Siegel modular variety.
\end{enumerate}
In the special case where $e = 0$, our algorithm does not need any vanishing result for the degeneration process as the spectral sequence \ref{sp1} is concentrated on one row. In the special case where $e = d-1$, the degeneration is automatic as it is given by the vanishing of the coherent cohomology in degree $i > d$. Then, these results can be used to run the algorithm with $e = 1$ and with different $I_0 \subset I$ and so on. \\

Our SageMath code defines a class SiegelData, which contains the root data, and a class VanishingEngine with some methods that can be used to compute vanishing results. A subset $I_0 \subset I$ is encoded by the tuple of its indicator values on $\alpha_1, \cdots, \alpha_{g-1}$. We create the Siegel threefold $X$ over $\mathbb{F}_7$.
\begin{pyin}
data = SiegelData(g = 2, p = 7)
X = VanishingEngine(data)
\end{pyin}
If the next line returns True, it means that the automorphic line bundle $\mathcal{L}_{(-2,-8)}$ is $D$-ample on the complete flag variety $Y$ over $X$ (by theorem \ref{th_amp_auto}).
\begin{pyin}
data.ample((0,), (-2,-8))
\end{pyin}
\begin{pyout}
True
\end{pyout}
The next line compute vanishing results for characters $\lambda = (k_1,k_2)$ with $ -50 \leq k_2 \leq k_1 \leq 0$ using the function $g_{I_0,e}$ in the case where $I_0 = \emptyset$ and $e = 0$. The results are registered in the list $C_{\text{van}}$. It returns True if the algorithm has found new vanishing results.
\begin{pyin}
X.apply((0,), 0, -50, 0)
\end{pyin}
\begin{pyout}
True
\end{pyout}
The next line runs the apply method for each $I_0 \subset I$ and $0 \leq e \leq d-1$, and repeats this until no new vanishing result is found. We only need to specify the range of characters $\lambda = (k_1,k_2)$ we want to consider. It returns the number of passes.
\begin{pyin}
X.saturate(-50, 0)
\end{pyin}
\begin{pyout}
3
\end{pyout}
The next line returns True if we know that $$H^i(X,\nabla^{\sub}(-4,-6)) = 0$$ for all $i>1$.
\begin{pyin}
X.vanishes(1,(-4,-6))
\end{pyin}
\begin{pyout}
True
\end{pyout}
If the next line returns False, it means we don't know if $$H^i(X,\nabla^{\sub}(-4,-6)) = 0$$ for all $i>0$.
\begin{pyin}
X.vanishes(0,(-4,-6))
\end{pyin}
\begin{pyout}
False
\end{pyout}
\FloatBarrier
\subsection{Explicit vanishing for $G = \Sp_4$}
We plot some vanishing results we have obtained for the Siegel threefold with our algorithm. We have also added the $p$-small characters for $\Sp_4$ with a twist by $-w_0$ to have them in the anti-dominant Weyl chamber. 

\begin{figure}
\centering
\begin{tikzpicture}[scale = 1.5]
\begin{axis}[
 axis lines=middle,
	grid=both,
	grid style={black!10},
	xmin=-44,
	xmax=5,
	ymin=-44,
	ymax=5,
	legend style={at={(0.1,0.65)},anchor=west, font = \tiny},
 	xlabel=$k_2$,
	ylabel=$k_1$,
	minor tick num=9,
	x axis line style = {stealth-},
	y axis line style = {stealth-},
	xticklabel shift={0.0cm},
	xlabel style={yshift=-7.2cm},
	yticklabel shift={-0.9cm},
	ylabel style={xshift=-7.2cm},
]

\addplot [only marks,
	color=green,
	mark=x,
	mark options={scale=0.7, fill=white}]
	table{results/g2p5_psmall.txt};
	\addlegendentry[align = left]{$p$-small weights for $\Sp_4$}
	
\addplot [only marks,
	color=black,
	mark=o,
	mark options={scale=0.6, fill=white}]
	table{results/g2p5_0.txt};
	\addlegendentry{Concentrated in $[0]$}

\addplot [only marks,
	color=blue,
	mark=triangle,
	mark options={scale=0.6, fill=white}]
	table{results/g2p5_1.txt};
	\addlegendentry{Concentrated in $[0,1]$}

\addplot [only marks,
	color=red,
	mark=square,
	mark options={scale=0.6, fill=white}]
	table{results/g2p5_2.txt};
	\addlegendentry{Concentrated in $[0,2]$}

\draw[scale=0.5, domain=-70:70, smooth, variable=\x, black] plot ({\x}, {\x});

\node at (-2.8,1.3) {\tiny $(0,0)$};
\node at (-8.1,-4.0) {\tiny $(-4,-4)$};

\end{axis}

\end{tikzpicture}
\caption{$g = 2, p = 5$. \\ \\
The weights $\lambda = ( k_1 \geq k_2)$ such that the cohomology is concentrated in degree $0$ contains in particular the positive parallel weights $(k,k)$ below $(-4,-4)$. The vanishing results in the region located near the positive parallel line comes from the degeneration with $I = \{(1,-1)\}$ and the rest corresponds to the degeneration with $I = \emptyset$.
}\label{fig2}
\end{figure}
\FloatBarrier
\begin{figure}
\centering

\begin{tikzpicture}[scale = 1.5]

\begin{axis}[
 axis lines=middle,
	grid=both,
	grid style={black!10},
	xmin=-44,
	xmax=5,
	ymin=-44,
	ymax=5,
	legend style={at={(0.1,0.65)},anchor=west, font = \tiny},
 	xlabel=$k_2$,
	ylabel=$k_1$,
	minor tick num=9,
	x axis line style = {stealth-},
	y axis line style = {stealth-},
	xticklabel shift={0.0cm},
	xlabel style={yshift=-7.2cm},
	yticklabel shift={-0.9cm},
	ylabel style={xshift=-7.2cm},
]

\addplot [only marks,
	color=green,
	mark=x,
	mark options={scale=0.7, fill=white}]
	table{results/g2p11_psmall.txt};
	\addlegendentry[align = left]{$p$-small weights for $\Sp_4$}
	
\addplot [only marks,
	color=black,
	mark=o,
	mark options={scale=0.6, fill=white}]
	table{results/g2p11_0.txt};
	\addlegendentry{Concentrated in $[0]$}

\addplot [only marks,
	color=blue,
	mark=triangle,
	mark options={scale=0.6, fill=white}]
	table{results/g2p11_1.txt};
	\addlegendentry{Concentrated in $[0,1]$}

\addplot [only marks,
	color=red,
	mark=square,
	mark options={scale=0.6, fill=white}]
	table{results/g2p11_2.txt};
	\addlegendentry{Concentrated in $[0,2]$}

\draw[scale=0.5, domain=-70:70, smooth, variable=\x, black] plot ({\x}, {\x});

\node at (-2.8,1.3) {\tiny $(0,0)$};
\node at (-8.1,-4.0) {\tiny $(-4,-4)$};

\end{axis}

\end{tikzpicture}
\caption{$g = 2, p = 11$. \\ \\
Notice that since the orbitally $p$-close condition is less restrictive with $p = 11$ than with $p=5$, we are able to access more weights. However, if we look far enough, we notice the same phenomenon with the two regions corresponding to $I = \{(1,-1)\}$ and $I = \emptyset$.
}\label{fig3}
\end{figure}
\begin{figure}
\centering
\begin{tikzpicture}[scale = 1.5]

\begin{axis}[
 axis lines=middle,
	grid=both,
	grid style={black!10},
	xmin=-70,
	xmax=5,
	ymin=-70,
	ymax=5,
	legend style={at={(0.1,0.65)},anchor=west, font = \tiny},
 	xlabel=$k_2$,
	ylabel=$k_1$,
	minor tick num=9,
	x axis line style = {stealth-},
	y axis line style = {stealth-},
	xticklabel shift={0.0cm},
	xlabel style={yshift=-7.2cm },
	xlabel style={xshift = 0.25cm},
	yticklabel shift={-0.9cm},
	ylabel style={xshift=-7.2cm},
	ylabel style={yshift = 0.15cm},
]

\addplot [only marks,
	color=green,
	mark=x,
	mark options={scale=0.7, fill=white}]
	table{results/g2p31_psmall.txt};
	\addlegendentry[align = left]{$p$-small weights for $\Sp_4$}
	
\addplot [only marks,
	color=black,
	mark=o,
	mark options={scale=0.5, fill=white}]
	table{results/Large/g2p31_0.txt};
	\addlegendentry{Concentrated in $[0]$}

\addplot [only marks,
	color=blue,
	mark=triangle,
	mark options={scale=0.5, fill=white}]
	table{results/Large/g2p31_1.txt};
	\addlegendentry{Concentrated in $[0,1]$}

\addplot [only marks,
	color=red,
	mark=square,
	mark options={scale=0.5, fill=white}]
	table{results/Large/g2p31_2.txt};
	\addlegendentry{Concentrated in $[0,2]$}

\draw[scale=0.5, domain=-110:110, smooth, variable=\x, black] plot ({\x}, {\x});

\node at (-3.8,1.6) {\tiny $(0,0)$};

\end{axis}

\end{tikzpicture}
\caption{$g = 2, p = 31$.}\label{fig4}
\end{figure}
\FloatBarrier

\subsection{Explicit vanishing for $G = \Sp_6$}
We plot some vanishing results we have obtained in the case $g = 3$ with our algorithm. The weights live in a three-dimensional space and we need $6$ different labels.

\definecolor{darkgreen}{RGB}{0,128,0}
\definecolor{navy}{RGB}{0,0,128}
\definecolor{darkred}{RGB}{128,0,0}

\begin{figure}
\centering
\begin{tikzpicture}[scale = 1.5]

\begin{axis}[
view={30}{30},
	xmin=-25,
	xmax=5,
	ymin=-25,
	ymax=5,
	zmin=-25,
	zmax=5,
	legend style={at={(0.20,0.77)},anchor=west, font = \tiny},
 	xlabel=$k_1$,
	ylabel=$k_2$,
	zlabel=$k_3$,
	x axis line style = {stealth-},
	y axis line style = {stealth-},
	z axis line style = {stealth-},
	xlabel style={yshift=-0cm },
	xlabel style={xshift = 0.0cm},
	ylabel style={xshift=0.0cm},
	ylabel style={yshift = 0.0cm},
	zlabel style={yshift=0.0cm},
	zlabel style={xshift = 0.0cm},
	grid,
]
	
\addplot3 [only marks,
	color=darkgreen,
	mark=o,
	mark options={scale=0.5, fill=white}]
	table{results/g3p11_0.txt};
	\addlegendentry{Concentrated in $[0]$}

\addplot3 [only marks,
	color=green,
	mark=o,
	mark options={scale=0.5, fill=white}]
	table{results/g3p11_1.txt};
	\addlegendentry{Concentrated in $[0,1]$}
	
\addplot3 [only marks,
	color=navy,
	mark=o,
	mark options={scale=0.5, fill=white}]
	table{results/g3p11_2.txt};
	\addlegendentry{Concentrated in $[0,1,2]$}
	
\addplot3 [only marks,
	color=blue,
	mark=o,
	mark options={scale=0.5, fill=white}]
	table{results/g3p11_3.txt};
	\addlegendentry{Concentrated in $[0,1,2,3]$}

\addplot3 [only marks,
	color=darkred,
	mark=o,
	mark options={scale=0.5, fill=white}]
	table{results/g3p11_4.txt};
	\addlegendentry{Concentrated in $[0,1,2,3,4]$}
	
\addplot3 [only marks,
	color=red,
	mark=o,
	mark options={scale=0.5, fill=white}]
	table{results/g3p11_5.txt};
	\addlegendentry{Concentrated in $[0,1,2,3,4,5]$}
\end{axis}

\end{tikzpicture}
\caption{$g = 3$, $p = 11$.}\label{fig5}
\end{figure}

\begin{figure}
\centering
\begin{tikzpicture}[scale = 1.5]

\begin{axis}[
view={30}{30},
	xmin=-25,
	xmax=5,
	ymin=-25,
	ymax=5,
	zmin=-25,
	zmax=5,
	legend style={at={(0.20,0.77)},anchor=west, font = \tiny},
 	xlabel=$k_1$,
	ylabel=$k_2$,
	zlabel=$k_3$,
	x axis line style = {stealth-},
	y axis line style = {stealth-},
	z axis line style = {stealth-},
	xlabel style={yshift=-0cm },
	xlabel style={xshift = 0.0cm},
	ylabel style={xshift=0.0cm},
	ylabel style={yshift = 0.0cm},
	zlabel style={yshift=0.0cm},
	zlabel style={xshift = 0.0cm},
	grid,
]
	
\addplot3 [only marks,
	color=darkgreen,
	mark=o,
	mark options={scale=0.5, fill=white}]
	table{results/g3p691_0.txt};
	\addlegendentry{Concentrated in $[0]$}

\addplot3 [only marks,
	color=green,
	mark=o,
	mark options={scale=0.5, fill=white}]
	table{results/g3p691_1.txt};
	\addlegendentry{Concentrated in $[0,1]$}
	
\addplot3 [only marks,
	color=navy,
	mark=o,
	mark options={scale=0.5, fill=white}]
	table{results/g3p691_2.txt};
	\addlegendentry{Concentrated in $[0,1,2]$}
	
\addplot3 [only marks,
	color=blue,
	mark=o,
	mark options={scale=0.5, fill=white}]
	table{results/g3p691_3.txt};
	\addlegendentry{Concentrated in $[0,1,2,3]$}

\addplot3 [only marks,
	color=darkred,
	mark=o,
	mark options={scale=0.5, fill=white}]
	table{results/g3p691_4.txt};
	\addlegendentry{Concentrated in $[0,1,2,3,4]$}
	
\addplot3 [only marks,
	color=red,
	mark=o,
	mark options={scale=0.5, fill=white}]
	table{results/g3p691_5.txt};
	\addlegendentry{Concentrated in $[0,1,2,3,4,5]$}
\end{axis}

\end{tikzpicture}
\caption{$g = 3$, $p = 691$.}\label{fig6}
\end{figure}
\FloatBarrier